\documentclass[11pt]{amsart}
\usepackage[T1]{fontenc}

\usepackage{amsmath,amsfonts,amssymb,amsthm,color,dsfont,amsbsy}
\usepackage{mathrsfs}
\usepackage[mathscr]{eucal}
\usepackage{enumitem}

\setlist[itemize]{labelindent=\parindent, leftmargin=*}
\setlist[enumerate]{labelindent=\parindent, leftmargin=*}
\usepackage{geometry}
\usepackage[hyperindex,breaklinks]{hyperref}
\hypersetup{
  colorlinks = true,
  allcolors = blue,
}
\usepackage{mathtools, bm, graphicx}

\usepackage{mlmodern}
\usepackage{caption}
\usepackage[export]{adjustbox} 
\newtheorem{thm}{Theorem}[section]
\newtheorem{cor}[thm]{Corollary}
\newtheorem{prop}[thm]{Proposition}
\newtheorem{proposition}[thm]{Proposition}
\newtheorem{lemma}[thm]{Lemma}
\theoremstyle{definition}
\newtheorem{exmp}[thm]{Example}
\newtheorem{definition}[thm]{Definition}

\newtheorem{remark}[thm]{Remark}
\numberwithin{equation}{section}

\usepackage{tikz}
\usetikzlibrary{cd, calc, knots}
\usetikzlibrary{shapes.geometric}
\usetikzlibrary{decorations, decorations.markings} 
\usetikzlibrary{arrows, arrows.meta}
\tikzset{
partial ellipse/.style args={#1:#2:#3}{
    insert path={+ (#1:#3) arc (#1:#2:#3)}}
}
\tikzset{
->-/.style={decoration={
markings,
mark=at position #1 with {\arrow{>}}},postaction={decorate}}
}
\tikzset{
-<-/.style={decoration={
markings,
mark=at position #1 with {\arrow{<}}},postaction={decorate}}
}
\tikzset{
-<<-/.style={decoration={
markings,
mark=at position #1 with {\arrow{-latex}}},postaction={decorate}}
}
\tikzcdset{
  tikzcd to/.tip={cm to[width=4pt,length=5.6pt]}
}
\tikzset{
  vtx/.style={circle,fill,inner sep=2pt}
}

\def\bd{{\partial}}
\def\mtr{{\mathbf{t}}} 
\newcommand{\Proj}{\operatorname{Proj}}
\def\adm{\operatorname{adm}}
\def\bi{\operatorname{bi}}
\def\rb{\operatorname{RB}}
\def\cgpt{\mathscr{F}}
\def\cut{{\operatorname{cut}}}
\def\Cap{{\operatorname{cap}}}
\def\nb{\mathfrak{N}}
\def\ourBord{\mathbf{Bord}^{\operatorname{nc}}_{\operatorname{oc}}}
\newcommand{\bdf}[1]{{\bd^{\operatorname{fr}}\hspace{-1pt}{#1}}}
\newcommand{\bdg}[1]{{\bd^{\operatorname{gl}}\hspace{-1pt}{#1}}}
\def\vphi{\varphi}
\def\span{\operatorname{span}}
\newcommand{\pair}[1]{\langle #1 \rangle}
\def\ourV{\cV}
\def\tbv{\widetilde{\ourV}}
\def\sk{\cS}
\def\Int{\operatorname{Int}}
\def\MCG{\operatorname{Mod}}

\newcommand{\bbSigma}{%
  \mathord{
    \tikz[baseline=(char.base)]{
      \node[inner sep=0pt] (char) {$\Sigma$}; 
\begin{scope}[xshift=0.5]
      \draw[line width=0.3pt] (-0.05,0.115)--(0.03,0.002);
      \draw[line width=0.3pt] (-0.06,-0.115)--(0.035,0.004);
\end{scope}
    }
  }
}
\newcommand\scalemath[2]{\scalebox{#1}{\mbox{\ensuremath{\displaystyle #2}}}}
\newcommand{\smbbSigma}{\scalemath{0.7}{\bbSigma}}

\def\intr{{\operatorname{Int}}}

\newcommand\void[1]{}

\newcommand{\threeCell}[1]{{#1}^{(3)}}
\newcommand{\twoCell}[1]{{#1}^{(2)}}
\newcommand{\oneCell}[1]{{#1}^{(1)}}
\newcommand{\zeroCell}[1]{{#1}^{(0)}}
\newcommand{\ourGraph}{\fG}

\newcommand{\spnTree}{\fT}
\newcommand{\loopEdge}{{\de_{2}\setminus\spnTree}}
\newcommand{\belt}{\operatorname{blt}}
\newcommand{\clsr}{\operatorname{cl}}
\newcommand\arxiv[2]      {\href{http://arXiv.org/abs/#1}{arXiv:#2}}
\newcommand\doi[2]        {\href{http://dx.doi.org/#1}{#2}}
\newcommand{\etalchar}[1] {$^{#1}$}

\def\cC{\mathcal{C}}

\def\cK{\mathcal{K}}

\def\cN{\mathcal{N}}

\def\cS{\mathcal{S}}

\def\cV{\mathcal{V}}

\def\cZ{\mathcal{Z}}

\def\bB{{\mathbb{B}}}

\def\bD{{\mathbb{D}}}

\def\bM{{\mathbb{M}}}
\def\bN{{\mathbb{N}}}

\def\bR{{\mathbb{R}}}
\def\bS{{\mathbb{S}}}

\def\kk{{\bm{k}}}

\def\Ld{\Lambda}

\def\de{\delta}

\def\fA{\mathfrak{A}}

\def\fG{\mathfrak{G}}

\def\fH{\mathfrak{H}}

\def\fT{{\mathfrak{T}}}

\def\1{{\mathds{1}}}
\def\to{\rightarrow}

\def\Hom{{\operatorname{Hom}}}

\def\End{{\operatorname{End}}}

\def\Rad{{\operatorname{rad}}}

\def\rad{\Rad}

\newcommand\Irr{{\operatorname{Irr}}}
\def\irr{{\operatorname{Irr}}}
\def\id{{\operatorname{id}}}

\def\Vect{{\mathbf{Vect}}}

\def\ev{{\operatorname{ev}}}
\def\coev{{\operatorname{coev}}}

\usepackage{diagbox}

\newcommand{\ot}{\otimes}

\title{Non-semisimple open-closed 3d TFT}
\author{Ingo Runkel}
\address{Ingo Runkel, Fachbereich Mathematik, Universität Hamburg, Bundesstraße 55, 20146 Hamburg, Germany}
\email{ingo.runkel@uni-hamburg.de}

\author{Yilong Wang}
\address{Yilong Wang, Beijing Institute of Mathematical Sciences and Applications (BIMSA), Beijing, China}
\email{wyl@bimsa.cn}

\begin{document}

\begin{abstract}
Given a spherical finite tensor category $\cC$, not necessarily semisimple, and a two-sided modified trace on its projective ideal, we define an open-closed three-dimensional topological field theory with values in vector spaces. 

The bordism category has as morphisms three-dimensional bordisms with corners, whose boundary is partitioned into the gluing boundary, pa\-ra\-me\-trised by source and target surface, and the unparametrised free boundary. The free boundary is equipped with an embedded $\cC$-coloured graph satisfying an admissibility condition.

Our construction starts from a new three-manifold invariant based on the multi-han\-dle\-bo\-dy invariant of \cite{CGPT20-Kup} and the chromatic maps of \cite{CGPV23}.
The open-closed topological field theory is then obtained via the universal construction.
\end{abstract}

\maketitle

\tableofcontents

\thispagestyle{empty}

\newpage

\section{Introduction}
State-sum constructions of three-dimensional topological field theories (TFTs) were first given by Turaev-Viro \cite{TV92} and Barrett-Westbury \cite{BW95}. State sums are naturally understood in terms of open-closed TFTs which assign vector spaces to 2-manifolds with boundary and linear maps to bordisms with corners, see for example the alterfold theory \cite{atfd1, atfd2, atfd3}, as well as \cite{SS25}. The constructions mentioned so far all require semisimple algebraic input. 

Recently, there has been a vast development in the construction of TFTs based on non-semisimple algebraic inputs, with the first major step forward being the invariants in \cite{CGP14} and the corresponding TFT in \cite{BCGP16}. Further examples of non-semisimple 3d TFTs include the one constructed from modular tensor categories based on the surgery presentation of 3-manifolds \cite{DeRenzi2022}, as well as the one defined in \cite{CGPV23} from chromatic categories (generalization of spherical fusion categories) via admissible skein modules and the generators-and-relations presentation of bordism categories in \cite{Juh18}. 

\vspace{0.8em}

In this paper, we construct a non-semisimple open-closed TFT, whose algebraic input is a triple $(\cC, \mtr, \pi_{\1})$ where $\cC$ is a spherical finite tensor category, $\mtr$ is a non-zero modified trace on $\Proj_{\cC}$ and $\pi_{\1}: P_{\1} \to \1$ is a surjection which just serves as normalisation. 
Our construction crucially uses chromatic maps and their properties from \cite{CGPV23}, as well as the corresponding invariant $\cgpt$ of admissible bichrome graphs in multi-handlebodies defined in \cite{CGPT20-Kup}. 
The basic idea is as follows: We first extend the multi-handlebody invariant $\cgpt$ to an invariant $\tau_{\cC}$ of general 3-manifolds with boundary, together with a $\cC$-coloured bulk-admissible graph embedded in the boundary. Bulk-admissibility means that each connected component of the 3-manifold has a projectively labelled edge somewhere in its boundary (Definition \ref{def:adm-1}). In particular, a connected 3-manifold is not allowed to have empty boundary.
We then use the universal construction \cite{BHMV2} to define a TFT. 

We briefly outline the definition of the invariant $\tau_{\cC}(M, \Gamma)$, where $M$ is a 3-manifold with boundary and $\Gamma$ is a bulk-admissible $\cC$-graph in $\bd M$. We choose a PLCW decomposition of $M$ and push $\Gamma$ to the boundary of a tubular neighbourhood of the 1-skeleton, resulting in a multi-handlebody $\fH$ with an admissible $\cC$-graph on its boundary. Next, we construct a graph using the 2- and 3-cells (Definition \ref{def:23-graph}), and choose a spanning tree of the graph. Then, we add red loops to $\fH$ to record how 2-cells are attached to the 1-skeleton of $M$, where we only add red loops for 2-cells that do not correspond to the edges in the spanning tree. In this way, 
$\fH$ is decorated with an admissible bichrome (red and blue) graph $\Gamma'$ on its boundary surface, and we define $\tau_{\cC}(M, \Gamma)$ to be the multi-handlebody
invariant $\cgpt(\fH, \Gamma')$. 

The first main result of this paper is that $\tau_{\cC}(M, \Gamma)$ does not depend on the choices of the PLCW decomposition, the spanning tree, and how $\Gamma$ is pushed to the boundary of the handlebody $\fH$. Namely, we have

\begin{thm}[Theorem \ref{thm:tau-well-def}]
$\tau_{\cC}(M, \Gamma)$ is well-defined and only depends on the diffeomorphism class of $(M, \Gamma)$.
\end{thm}

The key novelty of our construction is the use of spanning trees when adding red loops. Based on existing state-sum constructions of 3-manifold invariants and TFTs, it might be tempting to add red loops for every 2-cell: Indeed, in the semisimple case, this is typical (see for example \cite{atfd1} and \cite{SS25}); 
and in the non-semisimple case, this is closely related what the authors of \cite{CGPT20-Kup} did to define an invariant of closed 3-manifolds. However, in the presence of boundaries, such a construction would often create red loops bounding disks after handle-slides/sliding moves, and such red loops evaluate to zero in many non-semisimple cases  (Remark~\ref{rem:red-loop-zero}), 
but also in some semisimple examples (Remark~\ref{rem:ssi-chromatic}). 
In our construction we avoid such obviously unwanted zeros by using spanning trees, resulting in not only fewer red loops, but also less room for the red loops to slide (cf.~the proof of Lemma \ref{lem:T-indep}).

One can define an invariant of closed connected 3-manifolds $M$ by evaluating $\tau_\cC$ on $M$ minus a 3-ball with a suitable admissible graph $I$ on the boundary $\bS^2$ (see Section~\ref{sec:CGPT-compare} for details).
In Proposition~\ref{prop:compare-closed} we show that $\tau_\cC(M \setminus \bB^3,I)$ agrees with the invariant $\mathcal{K}_\cC(M)$ of closed connected 3-manifolds in \cite{CGPT20-Kup}.
    
\medskip
After defining the invariant $\tau_{\cC}$, we use the universal construction to construct a non-compact open-closed TFT on the bordism category of surfaces whose boundaries contain $\cC$-marked points (called $\bd$-marked surfaces), and 3-dimensional bordisms between them: 

\begin{thm}[Theorem \ref{thm:ocTFT}]
Via the universal construction, the invariant $\tau_{\cC}$ extends to an open-closed topological field theory $\ourV: \ourBord(\cC) \to \Vect_{\kk}$.     
\end{thm}

The key to the proof of Theorem \ref{thm:ocTFT} is Proposition \ref{prop:collar}, which says that the state space of any $\bd$-marked surface $(\Sigma, L)$ is spanned by what we call ``collar elements''. 
Such an element has underlying 3-manifold $\bbSigma = \Sigma \times [0,1]$ viewed as a bordism from $\emptyset$ to $(\Sigma, L)$. 
Using collar elements, we show that there exists a surjection 
\[E(\Sigma, L): \sk_{\adm}(\Sigma, L) \to \ourV(\Sigma, L)\]
from the admissible skein module of $(\Sigma, L)$ to the corresponding state space of $\ourV$ (Proposition \ref{prop:our-V}). In particular, $\ourV(\Sigma, L)$ is finite-dimensional.
In addition to the finite-dimensionality of the state spaces, Proposition \ref{prop:collar} also helps us to establish monoidality of $\ourV$. 

Both the TFT $\ourV$ in this paper and the TFT $\EuScript{S}$ of admissible skein modules in \cite{CGPV23} depend heavily on the properties of the multi-handlebody
invariant $\cgpt$, and the source bordism category of $\EuScript{S}$ can be identified with the full subcategory of $\ourBord(\cC)$ consisting of closed surfaces with no marked points. 
However, the two TFTs are defined by different methods: $\ourV$ is defined via the universal construction, while $\EuScript{S}$ is defined by Juh\'asz's presentation of bordism categories \cite{Juh18}. 
We expect that $\ourV$ agrees with $\EuScript{S}$ for closed surfaces and admissible bordisms with empty free boundary (Remark \ref{rmk:compare}). 

When the input category $\cC$ is semisimple, we argue in Remark \ref{rmk:atfd-1} that $\tau_{\cC}$ generalises the Turaev-Viro invariant by comparing both of these invariants with the alterfold invariant $Z_{\cC}$ in \cite{atfd1}. 
Namely, $\tau_{\cC}$ is a non-zero scalar multiple of $Z_{\cC}$ and, since both the open-closed TFT $\ourV$ and the alterfold TFT are built from the universal construction, the state space of any 2-alterfold under the alterfold TFT is isomorphic to the state space of its $B$-coloured region (which is a $\bd$-marked surface) under $\ourV$, see Remark \ref{rmk:atfd-2}. 

Finally, it would be interesting to explore the relation between the TFT $\ourV$ in this paper and the TFT constructed in \cite{DeRenzi2022}. We expect that for any spherical finite tensor category $\cC$, our TFT $\ourV$ coincide with the TFT in \cite{DeRenzi2022} associated to the Drinfeld center $\cZ(\cC)$ of $\cC$. This is inspired by the general ``RT=TV'' philosophy in the semisimple case \cite{Turaev:2017uxl, atfd2}, which is supported in particular by the alterfold theory. We plan to return to this question in the future.

\medskip

The paper is organised as follows. In Section 2, we review the definition and basic properties of the algebraic data that is used to define the 3-manifold invariant $\tau_{\cC}$ and the topological $\ourV$. In Section 3, we review admissible skein modules and the handlebody invariant $\cgpt$. In Section 4, we give the definition of the invariant $\tau_{\cC}$ and prove its topological invariance. Finally, in Section 5, we apply the universal construction to construct $\ourV$, and verify that it is indeed a topological field theory. 

\medskip

\noindent
\textbf{Acknowledgements.}
We would like to thank Francesco Costantino, Azat Gainutdinov and Shuang Ming for fruitful discussions.
I.R.\ thanks BIMSA for hospitality during a visit in September 2024 where this project was started.
Y.W. thanks University of Hamburg for hospitality during his visit in February 2026.
I.R.\ acknowledges support by the Deutsche Forschungsgemeinschaft (DFG, German Research Foundation) under Germany's Excellence Strategy - EXC 2121 ``Quantum Universe'' - 390833306 and the Collaborative Research Center - SFB 1624 ``Higher structures, moduli spaces and integrability'' - 506632645.
Y.W. is supported by NSFC No.~12301045 and No.~12571041. 

\bigskip

\noindent
\textbf{Conventions.}
Throughout this paper, $\kk$ will denote an algebraically closed field of arbitrary characteristic, and $\cC$ is a spherical finite tensor category over $\kk$.
Here, ``spherical'' means that $\cC$ is unimodular and pivotal and admits a two-sided modified trace on its projective ideal.

\section{Algebraic input}

In this section we describe the algebraic data we fix to define the three-manifold invariants and topological field theory in the later sections, and we collect the algebraic properties we will need.

\subsection{Spherical finite tensor categories}\label{sec:spherical}

As stated in the conventions, $\cC$ is a spherical finite tensor category over $\kk$. 
For \textit{finite tensor categories} we follow the conventions in \cite{EGNO}, see Definitions 1.8.6 and 4.1.1, as well as Section 6 there. In particular $\cC$, is rigid and has a simple tensor unit. We will take $\cC$ to be strictly monoidal for notational simplicity.

A tensor category is \textit{spherical} if it is pivotal, unimodular, and admits a two-sided modified trace on its projective ideal. Let us recall these notions in turn.

That $\cC$ is \textit{pivotal} means that there is a natural monoidal isomorphism $\delta : \mathrm{Id} \to (-)^{**}$, and we hence do not need to distinguish between left and right duals. For the various evaluation and coevaluation maps we employ the standard string diagram expressions, read from bottom to top,
\begin{equation}
\begin{split}
\ev_X = 
\includegraphics[valign=c,scale=1]{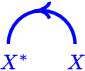} : X^* \otimes X \to \1
\quad & \quad 
\coev_X = 
\includegraphics[valign=c,scale=1]{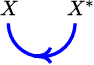} : \1 \to X \ot X^*\\ 
\widetilde{\ev}_X = 
\includegraphics[valign=c,scale=1]{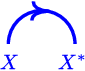} : X \otimes X^* \to \1 
\quad & \quad 
\widetilde{\coev}_X = 
\includegraphics[valign=c,scale=1]{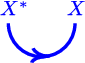} : \1 \to X^* \ot X
\end{split}
\end{equation}

We will denote by $\Proj_{\cC}$ the full subcategory of projective objects in $\cC$. These form a tensor ideal in $\cC$. Namely, tensor products of a projective object with any object are again projective, and direct sums and summands of projective objects are projective. If $P$ is projective, then so is $P^*$ \cite[Prop.\,6.1.3]{EGNO}, or, equivalently, every projective object in $\cC$ is also injective.

Every object $X \in \cC$ has a unique projective cover $P_X$. The projective cover $P_\1$ of the tensor unit $\1$ will be particularly important. The Hom-space $\cC(P_\1,\1)$ is one-dimensional by construction. That $\cC$ is \textit{unimodular} means that also $\cC(\1,P_\1)$ is non-zero, and hence necessarily one-dimensional. Equivalently, $\cC$ is unimodular iff $P_\1 \cong P_\1^*$. We will fix non-zero morphisms
\begin{equation}\label{eq:pi-and-iota}
    \pi_\1 : P_\1 \to \1
    \qquad \text{and} \qquad
    \iota_\1 : \1 \to P_\1 \ .
\end{equation}
These are unique up to scalars. The following conditions are equivalent: 
\begin{itemize}
    \item $\cC$ is semisimple,
    \item $\1$ is projective (then $X = X \otimes \1$ is projective for all $X \in \cC$),
    \item $\pi_\1 \circ \iota_\1 \neq 0$ (up to a rescaling, $\pi_\1 \circ \iota_\1 = \id_\1$ and so $\1$ is a direct summand of $P_\1$, hence projective).
\end{itemize}

A (left, right, or two-sided) \textit{modified trace} $\mtr$ on $\Proj_{\cC}$ is a family of linear maps
\begin{equation}
    ( \mtr_P )_{P \in \Proj_{\cC}}
    \quad \text{where} ~~
    \mtr_P : \End(P) \to \kk \ ,
\end{equation}
such that the $\mtr_P$ are cyclic and satisfy partial trace conditions which depend on whether $\mtr$ is left, right, or two-sided, see \cite{geer2010generalizedtracemodifieddimension} and \cite[Sec.\,2.2]{GPV13} for details. A useful way to think of a two-sided modified trace is as a linear form on the admissible skein module of the 2-sphere 
\cite{CGPV23}, see Proposition~\ref{prop:mod-tr-and-skein} below.
The following theorem is instrumental \cite[Cor.\,5.6]{GKP22}.

\begin{thm}
A pivotal unimodular finite tensor category $\mathcal{D}$ admits a non-zero right modified trace $\mtr$ on $\Proj_{\mathcal{D}}$. The right modified trace $\mtr$ is unique up to an overall scalar and induces a non-degenerate pairing for each $X \in \mathcal{D}$, $P \in \Proj_{\mathcal{D}}$,
$$
\mathcal{D}(X,P) \times \mathcal{D}(P,X) \to \kk
~~,\quad
(f,g) \mapsto \mtr_P(f \circ g) \ .
$$
\end{thm}

A pivotal unimodular finite tensor category is called \textit{spherical} if the unique-up-to-scalar right modified trace is two-sided.
By \cite[Thm.\,1.3]{Shibata_2021}, this is equivalent to the definition of sphericality in \cite[Def.\,3.5.2]{douglas2018dualizabletensorcategories}. If $\cC$ is semisimple, it agrees with the earlier notion of sphericality in \cite{Barrett:1993zf}, namely that that left and right categorical trace coincide.

\medskip

To summarise, the algebraic input for the construction of manifold invariants and the topological field theory below is a triple
\begin{equation}\label{eq:oc-invariant-alg-input}
    (\cC,\mtr,\pi_\1) \ ,
\end{equation}
where
\begin{itemize}
    \item $\cC$ is a spherical finite tensor category,
    \item $\mtr$ a non-zero two-sided modified trace on $\Proj_{\cC}$, and
    \item $\pi_\1 : P_\1 \to \1$ is a surjection.
\end{itemize}
This determines the embedding $\iota_\1:\1 \to P_\1$ via the non-degeneracy of the modified trace through the normalisation condition
\begin{equation}\label{eq:t-i-pi-norm}
    \mtr_{P_\1}(\iota_\1 \circ \pi_\1) = 1 \ .
\end{equation}
Below we will only deal with two-sided modified traces and so from now on, by ``modified trace'' we will always mean ``two-sided modified trace''.

\subsection{The central monad}

The construction of chromatic maps below relies on the central monad $Z : \cC \to \cC$ \cite{BV12}. For each $X \in \cC$, $Z(X)$ is defined to be the coend 
\begin{equation}
   Z(X) = \int^{V \in \cC} 
   V^* \otimes X \otimes V \ ,
\end{equation}
with dinatural maps $\iota_{X,V} : V^*\otimes X \otimes V \to Z(X)$. By construction of the functor $Z$, $\iota_{X,V}$ is natural in $X$. In fact, $Z$ is even a quasi-triangular Hopf monad, and the name ``central monad'' derives from the fact that the category of $Z$-modules is braided-equivalent to the Drinfeld centre of $\cC$ \cite{BV12, DS07}.

\begin{remark}
The Drinfeld centre $\mathcal{Z}(\mathcal{D})$ of a pivotal finite tensor category $\mathcal{D}$ has a ribbon structure if and only if $\mathcal{D}$ is spherical, see \cite[Thm.\,5.11]{shimizu2021ribbonstructuresdrinfeldcenter} and \cite[Cor.\,2.13]{muller2024distinguishedinvertibleobjectribbon}. By \cite[Prop.\,4.4]{etingof2004analogueradfordss4formula} and \cite[Thm.\,1.1]{shimizu2016nondegeneracyconditionsbraidedfinite}, the braiding on $\mathcal{Z}(\mathcal{D})$ is non-degenerate, and so in our setting $\mathcal{Z}(\cC)$ is a modular tensor category, that is, a non-degenerately braided finite ribbon tensor category.
\end{remark}

One way to construct $Z(X)$ is as a cokernel in $\cC$ \cite[Prop.\,5.1.7]{KL01} (as $\cC$ is rigid, the tensor product functor is biexact):

\begin{prop}\label{prop:monad-coker}
    Let $G$ be a projective generator of $\cC$ and let $\{ f_i \}_{i=1,\dots,n}$ be a basis of $\End(G)$. The following sequence is exact,
$$
\bigoplus_{i=1}^n G^* \otimes X \otimes G
\xrightarrow{\bigoplus_i f_i^* \otimes \id - \id \otimes f_i}
G^* \otimes X \otimes G 
\xrightarrow{\iota_{X,G}} Z(X) 
\to 0 \ .
$$
\end{prop}

In the above expression it is understood that $f_i^* \otimes \id_{X \otimes G} - \id_{G^* \otimes X} \otimes f_i$ starts in the $i$'th copy of $G^* \otimes X \otimes G$ of the direct sum.

As in \cite[Sec.\,4.5]{CGPV23}, we define 
\begin{equation}\label{eq:del-via-iota}
\partial_{X,V} = 
\includegraphics[valign=c,scale=0.8]{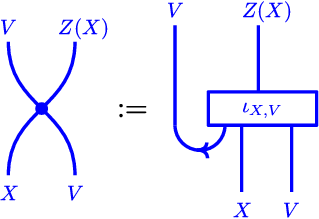}
 : X \otimes V \to V \otimes Z(X) \ .   
\end{equation}
The maps $\partial_{X,V}$ are natural in both $X$ and $V$. The universal property of $Z(X)$ now takes the following form: For each natural transformation $\xi : X \otimes (-) \to (-) \otimes Y$ there is a unique $\tilde\xi : Z(X) \to Y$ such that for all $V$,
\begin{equation}\label{eq:ZX-universal-via-delta}
\includegraphics[valign=c,scale=0.8]{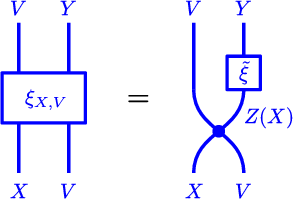}
\end{equation}

The Hopf monad $Z$ admits a unique-up-to-scalar integral and cointegral \cite[Thm.\ 4.8 and Cor.\,4.10]{shimizu2017integralsfinitetensorcategories}\footnote{\label{fn:comonad-vs-monad}
That reference uses Hopf comonads, see \cite[Sec.\,2.5\,\& Prop.\,2.10]{Berger_2021} for the translation to Hopf monads.}
\begin{equation}\label{eq:int-coint-normalisation}
\Lambda : Z(\1) \to \1~~,
\quad
\lambda : \1 \to Z(\1)~~,
\quad \text{such that} ~~
\Lambda \circ \lambda = 1 \ .
\end{equation}
That the tensor unit appears above rather than a more general invertible object is due to the assumption that $\cC$ is unimodular.

The integral $\Lambda$ is determined by the choice of modified trace $\mtr$ on $\cC$ as follows. For $P$ projective let $\{ \alpha_i \} \subset \cC(P,\1)$ be a basis, and let $\{ \alpha^i \} \subset \cC(\1,P)$ be its dual basis in the sense that $\mtr_P(\alpha^i \circ \alpha_j) = \delta_{i,j}$. Define
\begin{equation}\label{eq:Lam-t-P-def}
    \Lambda^{\mtr}_P = \sum_i 
\includegraphics[valign=c,scale=0.8]{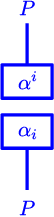}
    : P \to P \ . 
\end{equation}
The morphism $\Lambda^{\mtr}_P$ is independent of the choice of basis $\{ \alpha^i \}$ and hence uniquely determined by $\mtr$. The dual basis relation also implies that
\begin{equation}\label{eq:trace-Lambda}
    \mtr_P( \Lambda^{\mtr}_P ) = \dim \cC(P,\1) \ .
\end{equation}
We have \cite[Lem.\,4.1\,\&\,4.3]{CGPV23}:

\begin{lemma}\label{lem:Lam-natural}
There is a unique natural endomorphism $\Lambda^\mtr_\bullet$ of the identity functor such that for each $P \in \Proj_{\cC}$, $\Lambda^\mtr_P$ is as above.
\end{lemma}

By \eqref{eq:ZX-universal-via-delta} with $X=Y=\1$, this implies the existence of a unique $\Lambda : Z(\1)\to \1$ such that
\begin{equation}\label{eq:Lambda-from-Lambda-t-X}
    \includegraphics[valign=c,scale=0.8]{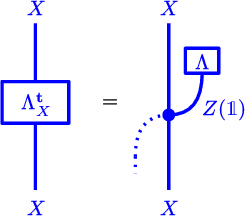}
\ ,
\end{equation}
which is indeed an integral for $Z$ \cite[Lem.\,4.8]{CGPV23}.

\begin{remark}\label{rem:ssi-from-integral}
The integral $\Lambda$ gives us another way to characterise semisimplicity of $\cC$ \cite[Prop.\,5.6]{shimizu2016monoidalcentercharacteralgebra}.${}^{
\ref{fn:comonad-vs-monad}}$ Namely, let $u := \iota_{\1,\1} : \1 \to Z(\1)$ be the unit of the monad $Z$. Then
$$
\cC \text{ semisimple}
\quad \Leftrightarrow \quad
\Lambda \circ u \neq 0 \ .
$$
Thanks to the explicit expression \eqref{eq:Lam-t-P-def}, we can give a simpler argument as compared to \cite{shimizu2016monoidalcentercharacteralgebra}. The direction `$\Rightarrow$' follows form explicit computation, which we omit.
For `$\Leftarrow$' first note that by \eqref{eq:Lambda-from-Lambda-t-X},
$\Lambda \circ u \neq 0$ is equivalent to $\Lambda^\mtr_\1 = \alpha \, \id_\1$ for some $\alpha \neq 0$. On the other hand, from \eqref{eq:t-i-pi-norm} and \eqref{eq:Lam-t-P-def} we see that $\Lambda^\mtr_{P_\1} = \iota_\1 \circ \pi_\1$.
For the surjection $\pi_\1 : P_\1 \to \1$, naturality of $\Lambda^\mtr_\bullet$ (Lemma~\ref{lem:Lam-natural}) gives
\begin{equation}
    \alpha^{-1} \, \pi_\1 
    = 
     \Lambda^\mtr_\1  \circ \pi_\1
    = 
     \pi_\1 \circ \Lambda^\mtr_{P_\1} 
    = 
    \pi_\1 \circ \iota_\1 \circ \pi_\1 \ .
\end{equation}
Thus we must have $\pi_\1 \circ \iota_\1 \neq 0$, which is one of the equivalent characterisations of semisimplicity of $\cC$.
\end{remark}

\subsection{Chromatic maps}

A \textit{chromatic map (based on a projective $P$ for a projective generator $G$)} is a map $c_{P,G} : P \otimes G \to P \otimes G$ such that for all $X \in \cC$,
\begin{equation}\label{eq:chromatic-def}
\includegraphics[valign=c,scale=0.8]{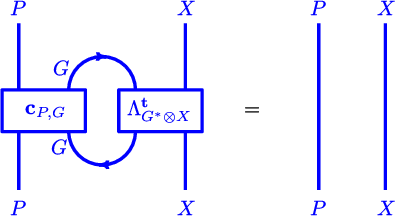}
    \ ,
\end{equation}
see \cite[Sec.\,1.8]{CGPV23}. Importantly, chromatic maps always exist, and we will quickly review their construction in
\cite{CGPV23} as we will need it later. 
Using that duals of projective covers are again projective covers, one checks that $G^*$ is also a projective generator, and from Proposition~\ref{prop:monad-coker} we see that $\iota_{\1,G^*}$ is surjective. 
Hence there exists a (typically non-unique) map $d_P : P \to P \otimes G \otimes G^*$ such that
\begin{equation}\label{eq:dP-definition}
    \begin{tikzcd}
        &&&& P \ar[d,"\id_P \otimes \lambda"]
        \arrow[dllll, dashed, "\exists d_P"']        
        \\
        P \otimes G \otimes G^* 
        \arrow[rr, "\id_P \otimes \delta_G \otimes \id_{G^*}"'] 
        &&
        P \otimes G^{**} \otimes G^* 
        \arrow[rr, "\id_P \otimes \iota_{\1,G^*}"'] 
        &&
        P \otimes Z(\1)
    \end{tikzcd}
    \quad .
\end{equation}
Define
\begin{equation}
c_{P,G} := \big[ P\otimes G \xrightarrow{d_P \otimes \id_G} P\otimes G\otimes G^*\otimes G \xrightarrow{\id_{P\otimes G} \otimes \mathrm{ev}_G} P\otimes G \big]
\end{equation}
Expressing $d_P$ in terms of $c_{P,G}$ and combining with \eqref{eq:del-via-iota}, the defining condition \eqref{eq:dP-definition} for $d_P$ is equivalent to
\begin{equation}\label{eq:chromatic-via-cointegral}
\includegraphics[valign=c,scale=0.8]{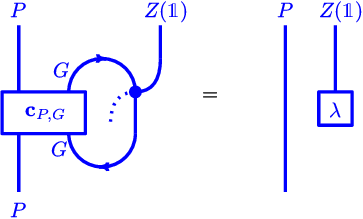}
\end{equation}
So far we know that there exists a $c_{P,G}$ which satisfies \eqref{eq:chromatic-via-cointegral}. Next we sketch the argument why $c_{P,G}$ is a chromatic map, i.e.\ why \eqref{eq:chromatic-via-cointegral} implies \eqref{eq:chromatic-def}. Writing $L$ and $R$ for the left and right hand side of \eqref{eq:chromatic-def}, we have
\begin{equation}
L \stackrel{(1)}{=} 
\includegraphics[valign=c,scale=0.7]{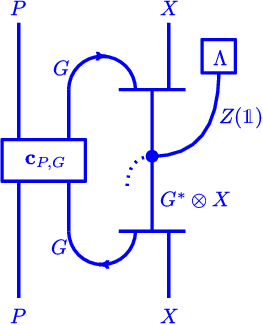} 
\stackrel{(2)}{=}
\includegraphics[valign=c,scale=0.7]{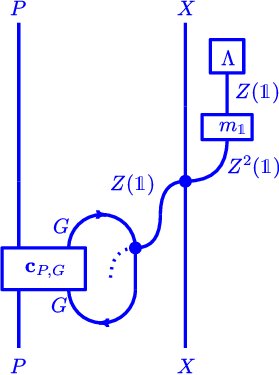} 
\stackrel{(3)}{=}
\includegraphics[valign=c,scale=0.7]{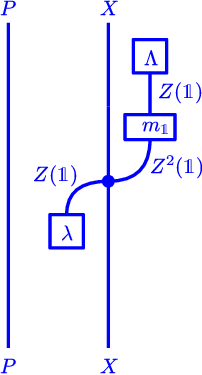} 
\stackrel{(4)}{=} R
\end{equation}
In step (1), the horizontal bars just mean that we combine the parallel strands labelled $G^*$ and $X$ into a single strand labelled $G^* \otimes X$. Formally, it stands for a coupon labelled $\id_{G^* \otimes X}$. We then substitute the definition of $\Lambda$ in \eqref{eq:Lambda-from-Lambda-t-X}. In step (2) we use the definition of the product $m_X : Z(Z(X)) \to Z(X)$ of the monad $X$, which we did not review, and for which we refer to \cite{bruguières2006hopfmonads} and \cite[Sec.\,4.5]{CGPV23}. Step (3) is \eqref{eq:chromatic-via-cointegral}, and step (4) follows from the defining property of a cointegral (see again the previous two references) and from the normalisation $\Lambda \circ \lambda = 1$.

We denote the affine linear space of solutions to  \eqref{eq:chromatic-via-cointegral} by 
\begin{equation}\label{eq:CPG-def}
    C_{P,G} = \big\{ c_{P,G} \in \End(P \otimes G) \,\big|\, c_{P,G} \text{ solves \eqref{eq:chromatic-via-cointegral}} \big\} \ ,
\end{equation}
and refer to $C_{P,G}$ as the \textit{chromatic space (based on the projective $P$ for the projective generator $G$)}. 
We summarise the discussion so far by the following theorem \cite[Thm.\,4.2]{CGPV23}:

\begin{thm}
For each $P \in \Proj_{\cC}$ and projective generator $G$, the chromatic space $C_{P,G}$ is non-empty, and each $c_{P,G} \in C_{P,G}$ is a chromatic map.
\end{thm}

\begin{remark}\label{rem:red-loop-zero}
Write $Z_0 : Z(\1) \to \1$ for the counit of the Hopf monad $Z$, characterised by $\id_X = (\id_X \otimes Z_0) \circ \partial_{\1,X}$. Composing \eqref{eq:chromatic-via-cointegral} with $\id_P \otimes Z_0$, we get
\begin{equation}\label{eq:G-loop-Z0lam}
  \includegraphics[valign=c,scale=1]{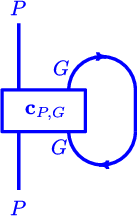}
  = (Z_0 \circ \lambda) \, \id_P \ .
\end{equation}
This particular morphism will be relevant when evaluating a red loop bounding a disc, as we will explain later in Remark~\ref{rem:red-loop-via-C_PG}.
Note that $Z_0 \circ \lambda$ is precisely opposite to the combination used in Remark~\ref{rem:ssi-from-integral} to detect semisimplicity. We are not aware of a proof that, for characteristic zero, $Z_0 \circ \lambda \neq 0$ iff $\Lambda \circ u \neq 0$ (for finite characteristic this is false in general), but it seems analogous to the Larsen-Radford theorem relating semisimplicity and cosemisimplicity for Hopf algebras, see \cite[Sec.\,4.7]{shimizu2017integralsfinitetensorcategories} for a related discussion. However, 
there are two special cases where we can already conclude that $Z_0 \circ \lambda \neq 0$ is equivalent to semisimplicity: 
\begin{enumerate}
    \item 
If $\cC$ is in addition ribbon and modular, then $Z = L \otimes (-)$ for $L = Z(\1)$ the canonical coend of $\cC$. As $\cC$ is ribbon, $L$ is a braided Hopf algebra, and as $\cC$ is modular, $L$ is self-dual.        
    Indeed, by one of the equivalent characterisations of modularity, there is a non-degenerate Hopf pairing for $L$, which exchanges integrals and cointegrals (see \cite{Lyu95M}, \cite[Def.\,5.2.7, Prop.\,4.2.14]{KL01} and \cite[Thm.\,1.1]{shimizu2016nondegeneracyconditionsbraidedfinite}). 
    \item  If $\cC = \mathrm{Rep}(H)$ for $H$ a finite-dimensional unimodular and unibalanced Hopf algebra such that $\dim(H) \neq 0$ in $\kk$, see \cite[Ex.\,1.5]{CGPV23}.
\end{enumerate}  
Much of the complication in the construction of the three-manifold invariants below is due to the fact that we allow for the case $Z_0 \circ \lambda = 0$.
\end{remark}

\begin{remark}\label{rem:ssi-chromatic}
Let us consider the case that $\cC$ is a spherical fusion category. Then $\cC = \Proj_{\cC}$ and for the modified trace $\mtr$ we take the categorical trace $\mathrm{tr}$.
The projective cover of $\1$ is $P_\1=\1$ and we choose $\pi_\1 = \id_\1$, which implies $\iota_\1 = \id_\1$.

The central Hopf monad is described for example in \cite[Ch.\,9.4]{Turaev:2017uxl}. $Z$ acts on objects as $Z(X) = \bigoplus_{U \in \Irr(\cC)} U^* \otimes X \otimes U$. We have
\begin{equation}
\begin{split}
     u = e_\1 : \1 \to Z(\1) 
     \ ,
     \quad & 
     Z_0 = \sum_{U \in \Irr(\cC)} \ev_U : Z(\1) \to \1  
     \ ,
     \\
     \Lambda = p_\1 : Z(\1) \to \1 \ ,
     \quad &
     \lambda = \sum_{U \in \Irr(\cC)} \dim(U) \,\widetilde{\coev}_U : \1 \to Z(\1) \ ,
\end{split}
\end{equation}
where $\dim(U)$ is the categorical dimension of $U$, and $e_\1 : \1 = \1^* \otimes \1 \hookrightarrow Z(\1)$ and $p_\1 : Z(1) \twoheadrightarrow \1^* \otimes \1 = \1$ are the embedding and projection of the direct summand $\1^* \otimes \1 \subset Z(\1)$. Note that with this normalisation indeed $\Lambda \circ \lambda = 1$ as required in \eqref{eq:int-coint-normalisation}. We further have
\begin{equation}
    \Lambda \circ u = 1
    ~~,\quad
    Z_0 \circ \lambda = \mathrm{Dim}(\cC) \ ,
\end{equation}
where $\mathrm{Dim}(\cC) = \sum_{U \in \Irr(\cC)} \dim(U)^2$ is the global dimension of $\cC$. In particular, $\Lambda \circ u \neq 0$ in accordance with Remark~\ref{rem:ssi-from-integral}. We also have the evident conclusion
\begin{equation}
    Z_0 \circ \lambda = 0 
    \quad \Leftrightarrow \quad 
    \mathrm{Dim}(\cC) = 0 \text{ in } \kk \ .
\end{equation}
This provides an example that $Z_0 \circ \lambda = 0$ can happen in semisimple categories over fields in finite characteristic. (In characteristic zero one necessarily has $\mathrm{Dim}(\cC) \neq 0$, see \cite[Thm.\,7.21.12]{EGNO}.) If $\cC = \mathrm{Rep}(H)$ is the category of representations of a semisimple spherical Hopf algebra, then $\mathrm{Dim}(\cC) = \dim_\kk(H)$, and so $Z_0 \circ \lambda = 0$ iff $\dim_\kk(H) = 0$ in $\kk$, cf.\ \cite[Ex.\,1.5]{CGPV23}.

To obtain the chromatic maps, 
as projective generator of $\cC$ we take $G := \bigoplus_{U \in \Irr(\cC)} U$. 
For any object $P \in \cC$, the chromatic map is then given by 
\begin{equation}
c_{P, G} = \bigoplus_{U \in \Irr(\cC)} \dim(U) \cdot \id_{P} \otimes \id_{U} \ ,
\end{equation}
see \cite[Ex.\,1.4]{CGPV23}.
Indeed, it is not hard to check that $c_{P,G}$ solves \eqref{eq:chromatic-via-cointegral}.
\end{remark}

\section{The handlebody invariant}

\subsection{Admissible skein modules}\label{sec:sk-mod}

In this section, we recall the definition of admissible skein modules, introduced in \cite{CGPV23}, and in \cite{Reutter:2020} under the name non-unital skein modules. For a general discussion of admissible skein modules, see also \cite{RST24}. Roughly speaking, the skein module of a surface is the vector space freely spanned by $\cC$-coloured graphs modulo relations induced by $\cC$ which hold inside embedded disks. If we require each graph to have at least one projectively coloured edge, we speak of admissible graphs and admissible skein modules.

\medskip

Let us go through the definition in more detail. We first need the notion of a $\partial$-marked surface.

\begin{definition}\label{def:d-marked-surface}
A \emph{$\bd$-marked surface} $(\Sigma,L)$ (over $\cC$) is a compact\footnote{
One can allow non-compact surfaces with a finite number of connected components, see \cite{RST24}, but we restrict ourselves to the compact case.}
oriented surface $\Sigma$ with boundary $\bd\Sigma$ (possibly empty), and a finite set $L$ of \textit{marked points}. A marked point is a triple $(p,\nu,X)$ where $p \in \bd\Sigma$, $\nu \in \{ \pm \}$, and $X \in \cC$. The points $p$ in $L$ are required to be mutually distinct, i.e., the map $L \to \bd\Sigma$ is injective.
\end{definition}

Let $(\Sigma,L)$ be a $\bd$-marked surface over $\cC$.
A \textit{blue graph} (over $\cC$), or a (blue) $\cC$-graph, in $(\Sigma,L)$, is a union of a finite collection of oriented arcs (embedded $[0,1]$), circles (embedded $\bS^1$) and vertices (embedded points). We refer to the arcs and circles as edges of the blue graph. The vertices, circles and the interior of the arcs (images of $(0,1)$) are mutually disjoint, and are contained in the interior of $\Sigma$. The endpoints of the arcs either lie on vertices of $\Gamma$ or on marked points on $\bd\Sigma$. In the latter case, the arc meets $\bd\Sigma$ transversally.
Each arc and each circle is labelled by an object of $\cC$. The endpoints of the arcs meeting a vertex $v$ are cyclically ordered by the orientation of $\Sigma$, and are partitioned into two consecutive sets of ingoing and outgoing endpoints, both equipped with a total order compatible with the cyclic order. The vertex is labelled by  a morphism in $\cC$, compatible with the in-- and outgoing endpoints of edges, their orientation, and their labels.
It will be convenient to represent the vertices as coupons, as in the usual string diagram notation:
\begin{equation}
\includegraphics[valign=c,scale=1]{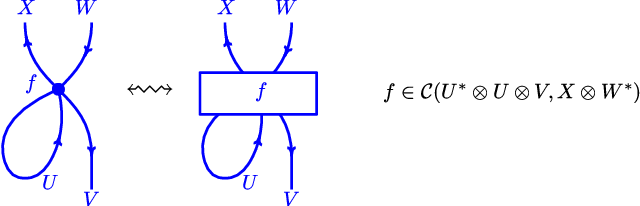}
\end{equation}
Each marked point $(p,\nu,X)$ on $\bd\Sigma$ is met by precisely one endpoint of one arc. This arc is labelled by $X$, and if $\nu=-$, the arc points away from $\bd\Sigma$ (i.e., the endpoint $0$ of $[0,1]$ is mapped to the marked point $p$), and if $\nu=+$ the arc points towards $\bd\Sigma$. 

\begin{definition}
    A blue graph $\Gamma$ on a $\bd$-marked surface $(\Sigma,L)$ is called \textit{admissible} if each connected component of $\Sigma$ contains at least one edge of $\Gamma$ labelled by an object in $\Proj_\cC$.
\end{definition}

Next we turn to skein relations. 
Let us write $\tilde{\sk}(\Sigma,L)$ for the $\kk$-linear span of all blue graphs on $\Sigma$, and $\tilde{\sk}_{\adm}(\Sigma,L)$ for the span of all admissible blue graphs.

Let $K \subset \Sigma$ be an embedded square $[0,1] \times [0,1]$ in the interior of $\Sigma$. 
We call $[0,1] \times \{0\}$ and $[0,1] \times \{1\}$ the in-- and outgoing base of $K$, respectively. Let $\Gamma$ be a blue graph which intersects the boundary of $K$ transversally and only at its bases. 
By evaluating the resulting graph in $K$ as a string diagram in $\cC$, we obtain a morphism
\begin{equation}
    F(\Gamma \cap K)
\end{equation}
in $\cC$ from the tensor product of objects labelling the edges intersecting the ingoing base to that for the outgoing base.
For example
\begin{equation}\label{eq:skein-K-example}
\includegraphics[valign=c,scale=0.8]{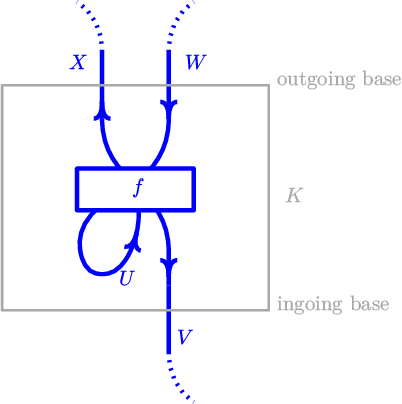}
F(\Gamma\cap K) = f \circ (\widetilde{\operatorname{coev}_U} \otimes \operatorname{id}_{V^*}): V \to X \otimes W^*
\end{equation}
This is well-defined because $\cC$ is assumed to be pivotal, see e.g., \cite[Ch.\,2]{Turaev:2017uxl} for more on the diagrammatic calculus for pivotal monoidal categories. 

A \textit{skein relation for $K$} is an element $\sum_{i=1}^{n} a_i\Gamma_i \in \tilde{\sk}(\Sigma,L)$ such that
\begin{enumerate}
\item $\Gamma_i$ only intersects the boundary of $K$ transversely and only at the bases of $K$;
\item the blue graphs agree outside of $K$: $\Gamma_i\cap (\Sigma\setminus K) = \Gamma_j\cap (\Sigma\setminus K)$ for $i, j = 1, ..., n$;
\item $\sum_{i=1}^{n} a_i F(\Gamma_i \cap K) = 0$ as a morphism in $\cC$.
\end{enumerate}
A skein relation for $K$ is called \textit{admissible}, if one (hence every) $\Gamma_i$ is admissible and the connected component of $\Sigma$ containing $K$ has a $\Proj_\cC$-coloured edge not completely contained in $K$. Equivalently, $\Gamma_i$ is admissible in $\Sigma \setminus K$. 

Let $\mathcal{N}(\Sigma,L) \subset \tilde{\sk}(\Sigma,L)$ be the span of all skein relations for $K$ as we vary over all choices of $K$. 
Analogously $\mathcal{N}_{\adm}(\Sigma,L) \subset \tilde{\sk}_{\adm}(\Sigma,L)$ is the span of all admissible skein relations. 

\begin{definition}\label{def:sk-adm}
Let $(\Sigma,L)$ be a $\bd$-marked surface.
The \textit{admissible skein module} (over $\cC$) for $(\Sigma,L)$ is the quotient space
$$
    \sk_{\adm}(\Sigma,L) = \tilde{\sk}_{\adm}(\Sigma,L) / \mathcal{N}_{\adm}(\Sigma,L) \ .
$$
We write $[\Gamma] \in \sk_{\adm}(\Sigma,L)$ for the class of a blue graph $\Gamma$. 
\end{definition}

Admissible skein modules can be generalised to other subcategories than $\Proj_{\cC}$ \cite{CGP23,RST24}, but we will not need that here.
It follows from \cite[Prop.\,5.6]{RST24} (see also \cite[Thm.\,2.3]{CGP23} for the case $L=\emptyset$) that: 

\begin{prop}\label{prop:finite-1}
The skein module $\sk_{\adm}(\Sigma, L)$ is finite dimensional over $\kk$ for all $\bd$-marked surfaces $(\Sigma, L)$.
\end{prop}

\begin{remark}
Of course one can also define skein modules without the admissibility condition in the same way as $\sk(\Sigma,L) = \tilde{\sk}(\Sigma,L) / \mathcal{N}(\Sigma,L)$. If $\cC$ is semisimple, $\sk(\Sigma,L)= \sk_{\adm}(\Sigma,L)$. For non-semisimple $\cC$, the $\sk(\Sigma,L)$ are less well behaved than $\sk_{\adm}(\Sigma,L)$. For example, for $\Sigma$ an annulus, $\sk(\Sigma,\emptyset)$ can be infinite dimensional \cite[Rem.\,5.10]{RST24}, while $\sk_{\adm}(\Sigma,L)$ is always finite-dimensional, as we saw in the previous proposition. 
\end{remark}

Admissible skein modules on the 2-sphere are dual to (two-sided) modified traces. Namely, for $P \in \Proj_\cC$ and $f : P \to P$, let $\Lambda_f$ be the blue graph in $\bS^2$ formed of a single vertex $v$ labelled $f$ and a single edge from $v$ to $v$ labelled $P$. We have (see \cite[Thm.\,2.4]{CGPV23}, and \cite[Prop.\,5.5]{RST24} for a proof based on excision):

\begin{prop}\label{prop:mod-tr-and-skein}
Let $\psi \in \sk_{\adm}(\bS^2, \emptyset)^*$. For $P \in \Proj_{\cC}$ write $\psi_P : \End_{\cC}(P) \to \kk$ for the map $f \mapsto \psi([\Lambda_f])$. Then $(\psi_P)_{P \in \Proj_{\cC}}$ is a two-sided modified trace on $\Proj_{\cC}$ and all such traces are of this form.
\end{prop}

Conversely, a linear form on $\sk_{\adm}(\bS^2, \emptyset)$ can be obtained from a modified trace $\mtr$ as follows. 
Let $\Gamma$ be an admissible blue graph on $\bS^2$. Pick a point $p$ on an edge $e$ labelled by a projective object $P$. Pick a small enough open disc $D$ around $p$, so that it intersects $e$ in a single arc. The complement of $D$ is topologically a square $K$, whose in- and outgoing base are intersected once by the edge $e$, so that $F(\Gamma \cap K)$ is a morphism $P \to P$, which we can then evaluate with $\mtr_P$.
It is proved in \cite[Thm.\ 5]{GPV13} that this is a well-defined isotopy invariant for admissible blue graphs on $\bS^2$, and in \cite[Thm.\,2.4]{CGPV23} that it respects admissible skein relations. Altogether, we get a linear form
\begin{equation}\label{eq:sph-inv}
\cgpt': \sk_{\adm}(\bS^2, \emptyset) \to \kk\,,\ \Gamma \mapsto \cgpt'(\Gamma) := \mtr_P(F(\Gamma \cap K)) \ .
\end{equation}
From here on, $\cgpt'$ will be the linear form corresponding to the choice of modified trace $\mtr$ we made in Section~\ref{sec:spherical}.

\begin{lemma}\label{lem:vanish-S2}
Suppose $\cC$ is non-semisimple.
If $[\Gamma] \in \sk_{\adm}(\bS^2, \emptyset)$ is the union of two disjoint graphs $\Gamma = \Gamma_1 \sqcup \Gamma_2$ such that each of $\Gamma_{1,2}$ contains a projectively coloured edge, then $\cgpt'([\Gamma])=0$.
\end{lemma}

\begin{proof}
As $\Gamma_1$ contains a projectively coloured edge, we can apply a skein relation to a square $K$ which entirely contains $\Gamma_2$. Since $\Gamma_2$ contains a projectively labelled edge, the evaluation $F(\Gamma_2) : \1 \to \1$ then can be written as a composition $\1 \to P_\1 \to \1$, which is always zero for non-semisimple  $\cC$.
\end{proof}

\subsection{Handlebody invariants}\label{sec:handlebody-inv}

The invariant $\cgpt'$ from \eqref{eq:sph-inv} can be extended to higher genus surfaces by identifying the surface with the boundary of a handlebody. To characterise the extension, we first need to recall the \textit{cutting move} from \cite{CGPT20-Kup}. 

By a multi-handlebody $H$ we mean a finite disjoint union of 3-dimensional handlebodies.
Let $\Gamma$ be an admissible blue graph on $\bd H$, that is, $\Gamma \in \tilde{\sk}_{\adm}(\bd H)$, where we abbreviate $\tilde{\sk}_{\adm}(\Sigma) = \tilde{\sk}_{\adm}(\Sigma,\emptyset)$ as $\bd H$ is closed and so there are no marked boundary points.

Let $D \subset H$ be an oriented properly embedded disk whose boundary $\bd D\subset \bd H$ does not meet the vertices of $\Gamma$ and intersects the strands of $\Gamma$ transversely in a non-empty set containing at least one $\Proj_\cC$-coloured strand. We call $D$ a \textit{cutting disc in $H$}.
Denote by $\cut_D(H)$ the multi-handlebody obtained by cutting $H$ along $D$. 
Cut along $\bd D$ the edges of $\Gamma$ which intersect $\bd D$, say, labelled by $X_1, ..., X_n$. Then, attach the cut points to two new coupons in $\bd\cut_D(H)$, one on each side of the cut. 
By assumption, $P := X_1 \ot \cdots \ot X_n \in \Proj_\cC$, which means we can colour the pair of coupons joining the two sets of the cut points by a dual basis pair as in the definition of $\Ld^{\mtr}_{P}$ in \eqref{eq:Lam-t-P-def}, see Figure~\ref{fig:cut-1}.
In this way, we obtain an admissible graph
\begin{equation}
    \cut_D(\Gamma) \in \tilde{\sk}_{\adm}(\bd \cut_D(H)) \ .
\end{equation}

\begin{figure}[tb]
\begin{center}
$\begin{tikzcd}[column sep = 5em]
\includegraphics[valign=c, scale=0.6]{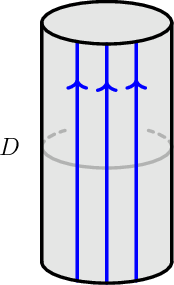}
\ar[r, "\mathrm{Cutting}"]&
\sum\limits_{i}\ \includegraphics[valign=c, scale=0.6]{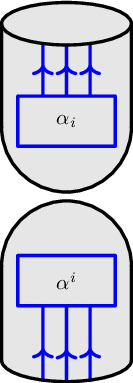}
\end{tikzcd}$
\caption{The cutting move along the disk $D$.}
\label{fig:cut-1}
\end{center}
\end{figure}

The next theorem is shown in \cite[Thm.\,4.1\,\&\,4.3]{CGP23} and is the key ingredient on which the construction of the open-closed TFT below will be based.

\begin{thm}\label{thm:CGP23}
Let $(\cC,\mtr,\pi_\1)$ be as in \eqref{eq:oc-invariant-alg-input}.
There is a unique collection of maps
\begin{equation}\label{eq:FH-def}
\cgpt := \{\cgpt(H,-): \sk_{\adm}(\bd H) \to \kk\}     \ ,
\end{equation}
where $H$ ranges over multi-handlebodies 
such that 
\begin{enumerate}
\item 
the values of {$\cgpt(H, [\Gamma])$} only depend on the orientation preserving diffeomorphism classes of elements $(H,[\Gamma])$; 
\item 
$\cgpt({\bD^3}, -) = \cgpt'$ from \eqref{eq:sph-inv} as linear maps $\sk_{\adm}(\bS^2, \emptyset) \to \kk$; 
\item $\cgpt$ is multiplicative with respect to disjoint unions; 
\item $\cgpt$ is invariant under the cutting move, that is, for $\Gamma \in \tilde{\sk}_{\adm}(\bd H)$ and $D$ a cutting disc, we have
\begin{equation*}
\cgpt(H, [\Gamma]) =
\cgpt( \cut_D(H), [\cut_D(\Gamma)] )\ .
\end{equation*}
\end{enumerate}
\end{thm}

\begin{exmp}
Let $\cC$ be non-semisimple, and let $P \in \Proj_\cC$. 
Let $H$ be a solid torus and let $\Gamma_c$ and $\Gamma_n$ be $P$-coloured circles embedded in $\partial H$ such that $\Gamma_c$ is contractible in $H$ and $\Gamma_n$ is not:
\begin{equation}
\includegraphics[valign=c,scale=0.8]{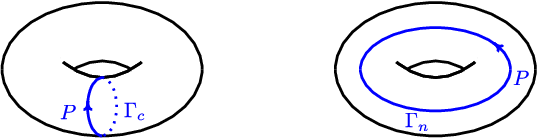}
\end{equation}
Then 
\begin{equation}
\cgpt(H, \Gamma_c) = 
\cgpt(H, \includegraphics[valign=c,scale=0.8]{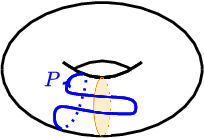})
=
\sum_{j}\cgpt(\bD^3, \includegraphics[valign=c,scale=0.4]{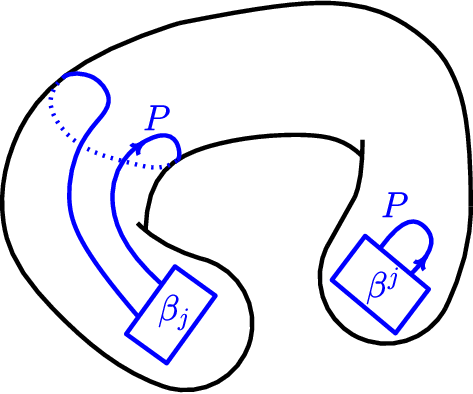})
= 0
\end{equation}
where the second step is the cutting move, and the third step is Lemma~\ref{lem:vanish-S2}. On the other hand,
\begin{equation}
\begin{split}
&\cgpt(H, \Gamma_n) = \cgpt(H, \includegraphics[valign=c,scale=0.8]{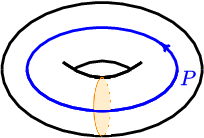}) = \sum_{i}\cgpt(\bD^3, \includegraphics[valign=c,scale=0.4]{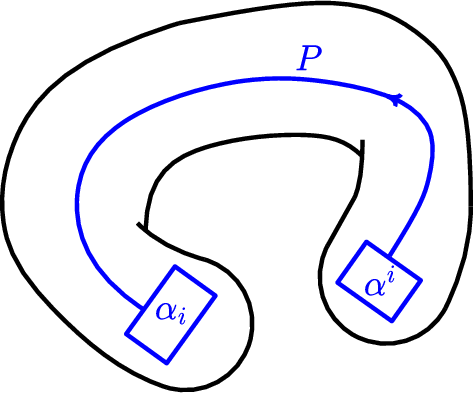})\\
=\ & \mtr(\Lambda_{P}^{\mtr}) = \dim_{\kk}\cC(P, \1)
\end{split}    
\label{eq:solid-torus-blue-non-contractible}
\end{equation}
where the last equality is \eqref{eq:trace-Lambda}.
\end{exmp}

\subsection{Skein modules and handlebody invariants for bichrome graphs}
A \emph{bichrome graph} in a closed surface $\Sigma$ is the disjoint union of a blue graph in $\Sigma$ (the ``blue part'') and finitely many pairwise disjoint unoriented embedded curves in $\Sigma$ (the ``red part'').

\begin{definition}\label{def:adm-bichrome}
A bichrome graph $\Gamma$ on a closed surface $\Sigma$ is \textit{admissible} if its blue part is admissible (each connected component of $\Sigma$ contains an edge of the blue graph labelled by an object in $\Proj_\cC$).
We write $\tilde{\sk}^{\mathrm{bi}}_{\adm}(\Sigma)$
for the linear span of admissible bichrome graphs on $\Sigma$.
\end{definition}

Next we define a surjective linear map $\mathrm{RB} : \tilde{\sk}^{\mathrm{bi}}_{\adm}(\Sigma) \to \sk_{\adm}(\Sigma)$ (``red-to-blue map'').
Recall the definition of the chromatic space $C_{P,G}$ in \eqref{eq:CPG-def}.
A \textit{red-blue modification} of an admissible bichrome graph is a modification in an annular neighbourhood of a red curve $\omega$ of the following form 
\begin{equation}\label{eq:red-to-blue-mod}
\includegraphics[valign=c,scale=0.9]{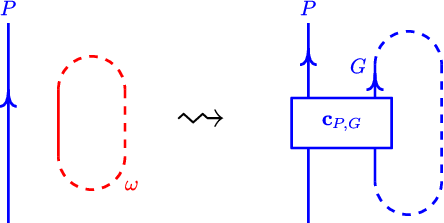}
\end{equation}
where $P \in \Proj_\cC$, $G$ is a projective generator, and $c_{P,G} \in C_{P,G}$.
Here $P$ can stand for a tensor product of several edge labels if the path connecting $\omega$ to the projectively labelled edge crosses other blue edges. In more detail, a red-blue modification depends on the following choices: 
1) a point $p$ on a red curve $\omega$ and a point $q$ on a blue edge labelled by a projective object $P$; 
2) an embedded path $\gamma : [0,1] \to \Sigma$ with $\gamma(0)=p$ and $\gamma(1)=q$ used to isotope the $P$-labelled edge (and the other edges it crosses) close to $p \in \omega$ ($\gamma$ is allowed to cross blue edges but not red curves); 
3) a choice of projective generator $G$ and a choice of chromatic map $c_{G,Q} \in C_{G,Q}$, where $Q = X_1 \otimes \cdots \otimes X_n \otimes P$ with $X_1,\dots,X_n$ labels of other blue strands crossed by $\gamma$.

We can use red-blue modifications to iteratively convert a bichrome graph $\Gamma$ into a purely blue graph  as follows: Define the \textit{depth} of a red curve $\omega \subset \Gamma$ as the minimal number of red curves a path from a point on $\omega$ to a point on a projectively labelled blue edge of $\Gamma$ needs to cross. Then apply red-blue modifications to all depth-0 red curves and repeat the process.
We denote the resulting admissible blue graph by $\Gamma_{\mathrm{rb}}$.
Of course, $\Gamma_{\mathrm{rb}}$ depends on many arbitrary choices, but it is shown in \cite[Lem.\,2.6]{CGPV23} and \cite[Prop.\,5.16]{costantino2026gradedsphericalskein21ghqft} that its image $[\Gamma_\mathrm{rb}] \in \sk_{\adm}(\Sigma)$ does not:

\begin{lemma}
$[\Gamma_\mathrm{rb}] \in \sk_{\adm}(\Sigma)$ only depends on $\Gamma$ and not on the order and choices made in the red-blue modifications during the above iterative procedure.    
\end{lemma}

Using this lemma, we can define
\begin{equation}
   \mathrm{RB} : \tilde{\sk}^{\mathrm{bi}}_{\adm}(\Sigma) \to \sk_{\adm}(\Sigma)
   ~~ , \quad
   \Gamma \mapsto [\Gamma_\mathrm{rb}] \ .
\end{equation}
Since admissible bichrome graphs without red part are allowed, the map $\mathrm{RB}$ is evidently surjective, and we define $\sk^{\mathrm{bi}}_{\adm}(\Sigma)$ to be the quotient
\begin{equation}
    \sk^{\mathrm{bi}}_{\adm}(\Sigma) := 
    \tilde{\sk}^{\mathrm{bi}}_{\adm}(\Sigma) / \ker( \mathrm{RB} ) \ .
\end{equation}
By abuse of notation, we also denote the resulting isomorphism by $\mathrm{RB}$:
\begin{equation}
   \mathrm{RB} : \sk^{\mathrm{bi}}_{\adm}(\Sigma) \xrightarrow{~\sim~} \sk_{\adm}(\Sigma) \ .
\end{equation}
In particular, we can now extend the invariant $\cgpt$ in \eqref{eq:FH-def} to admissible bichrome graphs:
\begin{equation}
    \cgpt^\mathrm{bi} := \{
    \cgpt^\mathrm{bi}(H, -):  \, \sk^{\mathrm{bi}}_{\adm}(\bd H) \xrightarrow{\mathrm{RB}} \sk_{\adm}(\bd H)
    \xrightarrow{\cgpt(H, -)}
    \kk \, \}
\end{equation}
where $H$ ranges over multi-handlebodies.
Accordingly, $\cgpt^\mathrm{bi}$ inherits properties from $\cgpt$ as listed in Theorem~\ref{thm:CGP23}. We state these as a corollary:

\begin{cor}\label{cor:F-H-bi-properties}
\ 
\begin{enumerate}
\item If $[\Gamma] \in \sk^{\mathrm{bi}}_{\adm}(\bd H)$ has no red components, then $\cgpt^\mathrm{bi}(H, [\Gamma]) = \cgpt(H, [\Gamma])$.

\item The values of $\cgpt^\mathrm{bi}$ only depends on the orientation preserving diffeomorphism classes of elements $(H,[\Gamma])$.

\item $\cgpt^\mathrm{bi}$ is multiplicative with respect to disjoint unions. 

\item Let $(H,\Gamma)$ be a handlebody with admissible bichrome graph $\Gamma$ and $D \subset H$ a cutting disc which does not meet the red graph, then 
\begin{equation*}
\cgpt^\mathrm{bi}(H, [\Gamma] ) =
\cgpt^\mathrm{bi}(\cut_D(H), [\cut_D(\Gamma)] )\ .
\end{equation*}
\end{enumerate}
\end{cor}

\begin{proof}
(1) follows as $\mathrm{RB}$ is the identity on blue graphs. (2) and (3) are clear, and for (4) note that one can apply the red-blue transformation in such a way that the graph is not changed in a neighbourhood of $\bd D$, and such that $\Gamma_\mathrm{rb}$ and $(\cut_D(\Gamma))_\mathrm{rb}$ agree outside a neighbourhood of $\bd D$.
\end{proof}

\begin{remark}\label{rem:red-loop-via-C_PG}
By construction, the replacement \eqref{eq:red-to-blue-mod} is an equality after applying $\cgpt^\mathrm{bi}(H,-)$. Consider a bichrome graph $\Gamma \cup \omega$, where the red loop $\omega$ bounds a disc in $\partial H$ which does not intersect the bichrome graph $\Gamma$. We can then use the admissible skein relation obtained from \eqref{eq:G-loop-Z0lam} to get the equality
\begin{equation}
\cgpt^\mathrm{bi}(H, [\Gamma \cup \omega] ) =
Z_0 \circ \lambda \cdot \cgpt^\mathrm{bi}(H, [\Gamma] ) \ .
\end{equation}
Thus if $Z_0 \circ \lambda = 0$, red loops bounding an otherwise empty disc lead to $\cgpt^\mathrm{bi}(H,-)$ being zero.
\end{remark}

The key relation that holds in $\sk^{\mathrm{bi}}_{\adm}(\Sigma)$ is that a projective blue strand can slide over a nearby red curve:

\begin{lemma}[{\cite[Lem.\,2.8]{CGPV23}}] \label{lem:projective-over-red}
Let $\Gamma_1$ and $\Gamma_2$ be two admissible bichrome graphs on a closed surface $\Sigma$ which agree outside a neighbourhood of a red curve $\omega$ and in the neighbourhood look like 
\begin{equation}
\Gamma_1 = \includegraphics[valign=c,scale=1]{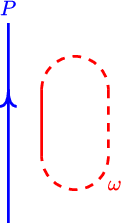}\,,\quad 
\Gamma_2 = 
\includegraphics[valign=c,scale=1]{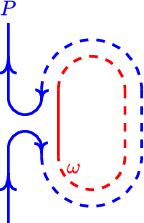}
\end{equation}
where $P \in \Proj_\cC$.
Then in $\sk^{\mathrm{bi}}_{\adm}(\Sigma)$ we have $[\Gamma_1] = [\Gamma_2]$.
\end{lemma}

There are two useful corollaries of this lemma. The first one is that also a non-projective edge can slide over a red curve:

\begin{cor}\label{cor:slide-X}
Suppose $\Sigma \setminus \omega$ is connected. Then the statement of Lemma~\ref{lem:projective-over-red} also holds if we replace $P$ by an arbitrary object $X \in \cC$.
\end{cor}

\begin{proof}
As $\Sigma \setminus \omega$ is connected, we can find a path $\gamma$ which connects $\omega$ to a $P$-labelled edge, with $P$ projective, and which passes through the $X$-labelled edge before crossing any other edges. If the path crosses more edges before reaching $P$ (possibly including the original edge $X$), we replace $P$ by the corresponding tensor product.
We can drag the $P$-labelled edge close to the $X$-labelled edge along $\gamma$.  Then the following equalities hold in $\sk^{\mathrm{bi}}_{\adm}(\Sigma)$: 
\begin{equation}
\includegraphics[valign=c,scale=1]{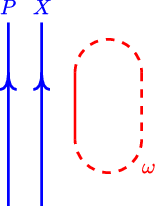}
=
\includegraphics[valign=c,scale=1]{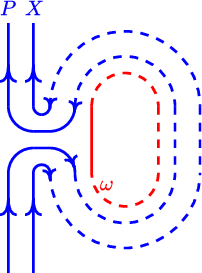}
=
\includegraphics[valign=c,scale=1]{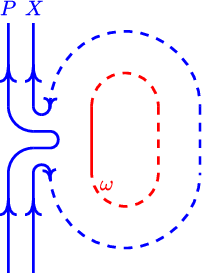}
\end{equation}
In the first step we apply Lemma~\ref{lem:projective-over-red} to the projective object $X \otimes P$, and in the second step we apply the same lemma in the opposite direction to the projective object $P$ only.
\end{proof}

The second corollary is that we can also slide a red curve over another red curve under suitable connectedness assumptions.

\begin{cor}\label{cor:red-over-red}
Let $\Gamma_1$ and $\Gamma_2$ be two admissible bichrome graphs on a closed surface $\Sigma$ which agree outside a neighbourhood of a red curve $\omega$ and in the neighbourhood look like:
\begin{equation}
\Gamma_1 = \includegraphics[valign=c,scale=1]{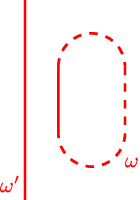}\,,\quad 
\Gamma_2 = 
\includegraphics[valign=c,scale=1]{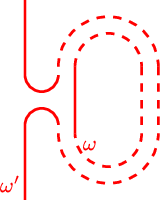}
\end{equation}
Suppose $\Sigma \setminus \omega$ is connected. 
Then in $\sk^{\mathrm{bi}}_{\adm}(\Sigma)$ we have $[\Gamma_1] = [\Gamma_2]$.
\end{cor}

\begin{proof}
As $\Sigma \setminus \omega$ is connected, we may pick a representative of $[\Gamma_1]$ where enough red curves have been changed into blue curves so that $\omega$ and $\omega'$ have depth 0. Again by connectedness of $\Sigma \setminus \omega$, we can then isotope a projectively labelled edge (or possibly a tensor product of several edge labels, including a projective one) either close to the left side of $\omega'$ or to the right side. 
We give the computation for the blue curve approaching from the left, the other computation is similar:
\begin{equation}
\includegraphics[valign=c,scale=0.9]{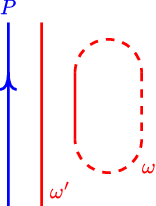}
=
\includegraphics[valign=c,scale=0.9]{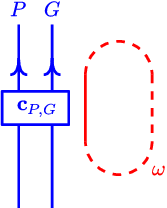}
=
\includegraphics[valign=c,scale=0.9]{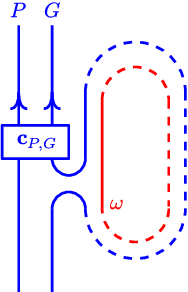}
=
\includegraphics[valign=c,scale=0.9]{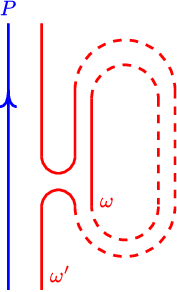}
\end{equation}
The first and third equality are both given by the red-blue modification, and the second equality is Lemma~\ref{lem:projective-over-red}.
\end{proof}

\begin{definition}\label{def:slide}
The change of bichrome graphs in Corollaries \ref{cor:slide-X} and \ref{cor:red-over-red} is called a \emph{sliding move} along a red curve. 
\end{definition}

Let $(H, \Gamma)$ and $(H', \Gamma')$ be multi-handlebodies with admissible bichrome graphs on the boundary. 
Following \cite[Sec.\,5]{CGPT20-Kup}, we say that $(H, \Gamma)$ is obtained from $(H', \Gamma')$ by a \emph{red capping move} along a red loop $\omega \subset \Gamma'$ and write $(H, \Gamma) = \Cap_{\omega}(H', \Gamma')$ if 
\begin{itemize}
\item there is a properly embedded disk $D \subset H'$ such that $\bd D \subset \bd H'$ intersects $\omega$ at one point;
\item $H$ is obtained from $H'$ by attaching a 2-handle along $\omega$; and
\item $\Gamma = \Gamma' \setminus \omega$.
\end{itemize}
In the converse direction we say that $(H', \Gamma')$ is obtained from $(H, \Gamma)$ by a \emph{red digging move}:
\begin{equation}\label{eq:cap-dig-1}
\includegraphics[valign=c,scale=0.8]{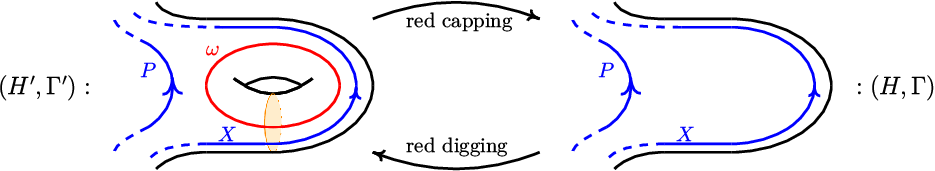}
\end{equation}
The next lemma shows that $\cgpt^\mathrm{bi}$ is invariant under red digging/capping moves. This is shown in \cite[Lem.\,5.7]{CGPT20-Kup} in the case that every edge of the blue graph is coloured in $\Proj_\cC$. The same proof works for admissible bichrome graphs, and we review it below for the reader's convenience.

\begin{lemma}\label{lem:hb-cap}
If $(H,\Gamma)$ is obtained from $(H',\Gamma')$ by a red capping move, then we have 
$\cgpt^\mathrm{bi}(H, [\Gamma]) = \cgpt^\mathrm{bi}(H', [\Gamma'])$.
\end{lemma}

\begin{proof}
Let $D \subset H'$ be a disc as in the conditions on the red capping move above. We may assume that a projectively labelled edge (or a tensor product involving one such edge) is near $\omega$. We get the following identities, for $P \in \Proj_\cC$ and $X \in \cC$, which hold after applying $\cgpt^\mathrm{bi}$,
\begin{equation}
\begin{split}
&\includegraphics[valign=c,scale=0.8]{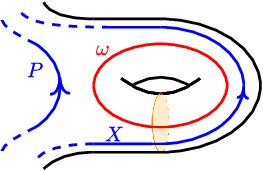}
\stackrel{(1)}{=}
\includegraphics[valign=c,scale=0.8]{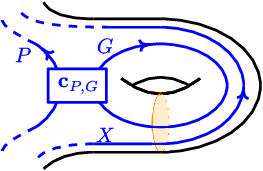}
\stackrel{(2)}{=}
\sum_i \includegraphics[valign=c,scale=0.8]{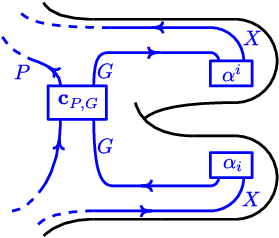}\\
&\stackrel{(3)}{=}
\includegraphics[valign=c,scale=0.8]{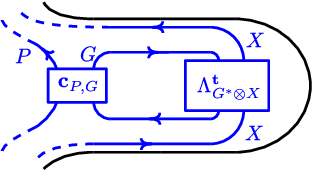}
\stackrel{(4)}{=}
\includegraphics[valign=c,scale=0.8]{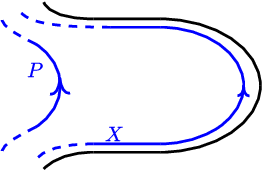}
\end{split}
\end{equation}
Step (1) is the red-blue modification, step (2) is invariance under the cutting move (Corollary~\ref{cor:F-H-bi-properties}), step (3) is just the definition of $\Lambda_{G^* \otimes X}^{\mtr}$ from \eqref{eq:Lam-t-P-def}, and step (4) is \eqref{eq:chromatic-def}. Note that step (3) and (4) are admissible skein relations as there is still the $P$-labelled edge in the complement of the disc where we apply identities of morphisms in $\cC$.
\end{proof}

\section{Three-manifold invariant}\label{sec:def-of-tau}

\subsection{Piecewise linear cell decompositions of manifolds}
In this section, we recall the notion of PLCW decompositions of manifolds following \cite{Kir12}. 
From now on, all maps are assumed to be piecewise linear, and for a subset $A$ inside a manifold $M$, we write $\nb(A)$ for a small neighbourhood of $A$ in $M$.

A \textit{generalised $n$-cell} $A^{(n)}$ is the image of a map $\chi: \bD^n \to \bR^N$ such that $\chi(\intr(\bD^n)) \cap \chi(\bd \bD^n) = \emptyset$ and $\chi$ is injective on $\intr(\bD^n)$. 
In this case, we call $\chi$ a \textit{characteristic map} for $A^{(n)}$, and denote $\intr(A^{(n)}) := \chi(\intr(\bD^n))$, $\bd A^{(n)} := \chi(\bd \bD^n)$. 
A \textit{generalised cell complex} is a finite collection $K$ of generalised cells such that for distinct $A$ and $B$ in $K$, $\intr(A) \cap \intr(B) = \emptyset$, and for any cell $A \subset K$, $\bd A$ is a union of cells. 
The support $|K|$ of $K$ is the union of all its cells, i.e., $|K| = \bigcup_{A \in K} A$. 
A \textit{regular cellular map} $f: L \to K$ between two generalised cell complexes is a map $f: |L| \to |K|$ such that for every cell $A = \chi(\bD^n) \in L$, there exists a cell $A' \in K$ such that $A' = f(A)$ and $f \circ \chi: \bD^n \to A'$ is a characteristic map for $A'$. 
The dimension and $n$-skeleton $K^{(n)}$ of a generalised cell complex are defined in the usual way.

A \textit{PLCW complex} is a generalised cell complex $K$ such that $\dim(K)=0$, or $\dim(K)=n>0$ and the following conditions hold:
\begin{itemize}
\item $K^{(n-1)}$ is a PLCW complex;
\item For any $n$-cell $A \in K$, $A = \chi(\bD^n)$, there is a PLCW complex $L$ such that $|L| \cong \bd\bD^n$, and the restriction $\chi|_{\bd\bD^n}: L \to K^{(n-1)}$ is a regular cellular map.
\end{itemize}

A \emph{PLCW decomposition} of a compact manifold $M$ is a PLCW complex $\de$ such that $M$ is homeomorphic to $|\de|$. If $\dim(M)=n$, then for $0 \le i \le n$ we denote the set of $i$-cells by $\de_i(M)$, or simply by $\de_i$. To indicate the relation with $M$, we denote by $M^{(i)}_{\de} = |\de_i|$ the $i$-skeleton of $M$ with respect to $\de$, and we write $\delta = (M_\delta^{(0)} \subset \cdots \subset M_\delta^{(n)} = M)$ when $\de$ is a PLCW decomposition of $M$.

Different PLCW decompositions of a compact manifold are related by a sequence of moves called elementary subdivisions introduced in \cite[Sec.\,7]{Kir12}. Here we briefly recall relevant concepts for later use. 

Let $n \ge 1$ and let $R_0 \subset \bR^n$ be the hyperplane determined by $x_n=0$, which divides $\bR^n$ into two subspaces $R_+$ (resp.~$R_-$) satisfying $x_n \ge 0$ (resp.~$x_n\le 0$). 
Let $K$ be a PLCW complex and $A^{(n)} = \chi(\bD^n)$ an $n$-cell. 
By definition, there is a PLCW decomposition $L$ of $\bd\bD^n$ such that $\chi|_{\bd\bD^n}: L \to K^{(n-1)}$ is a regular cellular map. 
If the image of the equator $\bd \bD^n \cap R_0$ in $L$ is a union of cells of $L$, then we can obtain a PLCW complex $K'$ from $K$ by replacing $A^{(n)}$ by the collection of the cells $A_+ = \chi(\bD^n \cap R_+)$, $A_- = \chi(\bD^n \cap R_-)$ and $A_0 = \chi(\bD^n \cap R_0)$. In this case, we say that $K'$ is obtained from $K$ by an elementary subdivision of the cell $A^{(n)}$. In summary, under suitable conditions, an elementary subdivision of $A^{(n)}$ divides $A^{(n)}$ into two $n$-cells that are separated by an $(n-1)$-cell. 

According to \cite[Thm.\,8]{Kir12}, if $\de$ and $\de'$ are PLCW decompositions of the same compact manifold $M$, then they are related by a finite sequences of elementary subdivisions and their inverses. We will need this fact to show that our invariant $\tau$ in Section \ref{sec:def-of-tau} is well-defined.

\begin{definition}\label{def:23-graph}
Let $M$ be a connected 3-manifold with boundary $\bd M$, and $\de$ a PLCW decomposition of $M$. 
The \emph{(2,3)-graph} of $M$ associated to $\de$, denoted by $\ourGraph_{\de}(M)$, is the graph constructed as follows:
\begin{itemize}
\item 
The set of vertices of $\ourGraph_{\de}(M)$, denoted by $\ourGraph_{\de}(M)_0$, consists of a distinguished vertex $v_0$, together with one vertex for each 3-cell of $M$. 
In other words, we have 
$\ourGraph_{\de}(M)_0 =\{v_0\} \sqcup \delta_3$, 
and in particular, $|\ourGraph_{\de}(M)_0| = |\de_3|+1$.

\item 
For each 2-cell $\twoCell{W}$ between two 3-cells $\threeCell{X}$ and $\threeCell{Y}$, i.e., $\twoCell{W} \subset \threeCell{X}\cap\threeCell{Y}$, we add an edge connecting $\threeCell{X}$ and $\threeCell{Y}$. For each 2-cell $\twoCell{W}$ between $\bd M$ and a 3-cell $\threeCell{X}$, i.e., $\twoCell{W} \subset \threeCell{X} \cap \bd M$, we add an edge connecting $v_0$ and $\threeCell{X}$. 
\end{itemize}
\end{definition}

\begin{remark}\label{rmk:our-graph}
The following observations follow immediately from the definition.
\begin{enumerate}
\item Since $M$ is assumed to be connected, $\ourGraph_{\de}(M)$ is a connected graph.
\item The distinguished vertex $v_0$ can be thought of as corresponding to the boundary $\bd M$ of $M$.
\item The part $\ourGraph_{\de}(M) \setminus \{v_0\}$ of the graph $\ourGraph_{\de}(M)$ with half-edges around $v_0$ can be embedded in $M$. Moreover, $\nb(M_\delta^{(1)})$, the neighbourhood of the 1-skeleton of $M$, is nothing but (the closure of) the complement of $\ourGraph_{\de}(M) \setminus \{v_0\}$ in $M$. 
\item  By construction, we have a bijection between the set $\de_2$ of 2-cell of $M$ and the set of the edges of $\ourGraph_{\de}(M)$. In the following, we will identify these two sets. 
\end{enumerate}
\end{remark}

We now consider substructures of a connected graph which will be used later.
A \textit{spanning tree} of a connected graph $\fG$ is a connected subgraph $\fT$ of $\fG$ which contains all the vertices of $\fG$ and has no cycles. 
Consequently, removing any edge from $\fT$ results in a disconnected graph.

\begin{remark}\label{rmk:cycle-switch}
By abuse of notation, we will denote the set of edges of a graph $\fG$ also by $\fG$. We will need the following properties of spanning trees of a graph $\fG$.
\begin{enumerate}
\item Suppose $\fG$ has $n$ vertices and let $\mathfrak{S} \subset \fG$ be a subgraph with the same set of vertices. Then $\mathfrak{S}$ is a spanning tree of $\fG$ if and only if $\mathfrak{S}$ is connected and has exactly $n-1$ edges, see e.g.\ \cite[Thm.\,2.1.4]{West-graphtheory}.

\item Let $\fT \subset \fG$ be a spanning tree and let $f$ be an edge not in $\fT$. Then there is a unique cycle $C$ formed by $f$ and edges from $\fT$. Removing any edge $e \neq f$ from $C$ results again in a spanning tree $(\fT \setminus \{e\}) \cup \{f\}$. Let us write this as $\fT(e \leadsto f)$. 

\item Let $\fT$ and $\fT'$ be two different spanning trees and let $f$ in $\fT'\setminus \fT$. Then there is an edge $e$ in $\fT\setminus \fT'$ such that $\fT(e\leadsto f)$ is again a spanning tree, see e.g.~\cite[Prop.\,2.1.7]{West-graphtheory} or~\cite[Ex.\,5.8]{Grinberg-2023}. Note that $\fT(e\leadsto f)$ has one edge more in common with $\fT'$ than $\fT$ had.

To find $e$ as above, consider the cycle $C$ formed by $f$ and edges of $\fT$. As $\fT'$ has no cycles, there must be at least one edge $e$ of $C$ not contained in $\fT'$. Apply part 1 to the pair $e$ and $f$.

\item The \textit{tree graph} $T(\fG)$ has as vertices spanning trees $\fT \subset \fG$ and has an edge between $\fT$ and $\fT'$ if and only if $\fT'=\fT(e \leadsto f)$ for some $e,f$ in the notation of part 1.
By part 2, $T(\fG)$ is connected, cf.\ \cite[Ex.\,5.9]{Grinberg-2023}.

\item Let $T_0(\fG)$ be the subgraph with the same vertices, but with fewer edges. Namely, $\fT$ and $\fT'$ are adjacent in $T_0(\fG)$ if and only if $\fT'=\fT(e \leadsto f)$ for some $e,f$ which have a common vertex $v$ in $\fG$. Then $T_0(\fG)$ is still connected.

To see this, note that by part 3 it is enough to show that $\fT$ and $\fT(e\leadsto f)$ are connected in $T_0(\fG)$ for an arbitrary choice of $e$ in $\fT$ and $f$ not in $\fT$. By part 1 there is a unique cycle $C$ containing $f$ and edges of $\fT$. This cycle must contain $e$. Let $f= h_0,h_1,\dots,h_n=e$ be consecutive edges on $C$. By construction, $h_1,\dots,h_n$ are in $\fT$ while $h_0$ is not. Let $\fT_{i+1} = \fT_i(h_{i+1} \leadsto h_i)$ with $\fT_0=\fT$. Then $\fT_1 = \fT_0(h_1 \leadsto h_0) = \fT(h_1 \leadsto f)$ and $\fT_2 = \fT_1(h_2 \leadsto h_1) = \fT(h_2 \leadsto f)$, etc., giving $\fT_i = \fT(h_i \leadsto f)$. By construction, $\fT_i$ and $\fT_{i+1}$ are adjacent in $T_0(\fG)$, and so $\fT_0 = \fT$ and $\fT_n = \fT(e \leadsto f)$ are connected in $T_0(\fG)$.
\end{enumerate}
\end{remark}

\subsection{Bulk-admissible graphs and the invariant \texorpdfstring{$\tau$}{}}
In this section we fix a triple $(\cC,\mtr,\pi_\1)$ as in \eqref{eq:oc-invariant-alg-input}. One can compare the following definition with \cite[Def.\,2]{RST24}.

\begin{definition}\label{def:adm-1}
A \emph{bulk-admissible graph over $\cC$} (or \textit{bulk-admissible $\cC$-graph}) in a compact oriented 3-manifold $M$ with boundary is a blue graph $\Gamma \subset \partial M$ over $\cC$ with no free ends such that for each connected component $J$ of $M$, at least one connected component of $\partial J$ contains a strand of $\Gamma$ coloured with an object in $\Proj_\cC$.
\end{definition}

\begin{definition}
Let $M$ be a 3-manifold with boundary $\Sigma := \partial M$, and $\Gamma$ a bulk-admissible $\cC$-graph in $M$. 
A \emph{generic} PLCW decomposition of $(M, \Gamma)$ is a PLCW decomposition $\delta = (M_\delta^{(0)} \subset M_\delta^{(1)} \subset M_\delta^{(2)} \subset M_\delta^{(3)} = M)$ of $M$, that satisfies the following conditions:
\begin{itemize}
\item 
$\delta$ restricts to a PLCW decomposition $(\Sigma_\delta^{(0)} \subset \Sigma_\delta^{(1)} \subset \Sigma_\delta^{(2)})$ of $\Sigma$;
\item 
$\Gamma \cap \delta_{0}(\Sigma) = \emptyset$;
\item 
$\Gamma$ intersects $\delta_{1}(\Sigma)$ only on the edges of $\Gamma$ and these intersections are transversal.
\end{itemize}
\end{definition}

Let $M$ be a connected 3-manifold and $\Gamma$ a bulk-admissible graph on $\bd M$. Choose the following data:
\begin{itemize}
\item $\delta$, a generic PLCW decomposition of $(M,\Gamma)$,
\item $\spnTree$, a spanning tree of $\ourGraph_{\de}(M)$, and
\item $Q := \{q \in \intr(\twoCell{W}) \mid \twoCell{W} \subset \bd M\}$, a set of points in the interior of the 2-cells in $\bd M$ (one point for each boundary 2-cell) that avoids $\Gamma$, i.e., $q \notin \Gamma$ for all $q \in Q$.
\end{itemize}
Under the identification of the set of 2-cells of $M$ and that of the edges in $\ourGraph_{\de}(M)$, we write $\loopEdge$ for the subset of 2-cells that do not correspond to edges in $\spnTree$. 

Let $\nb(M_\delta^{(1)})$ be a tubular neighbourhood of $M_\delta^{(1)}$ (the 1-skeleton of $M$). For each $q\in Q$, let $\twoCell{W}_q$ be the 2-cell containing $q$, and let $U_q$ be a small disk around $q$ in $\twoCell{W}_q$ that is disjoint from $\Gamma$. 
For all $q \in Q$, enlarge $U_q$ in $\twoCell{W}_q$ and push $\Gamma$ to align along the $\bd U_q$'s in this enlarging process, 
so that $\Gamma$, as well as all of the $\bd U_q$'s, are contained in $\nb(\oneCell{M}_{\de})$:
\begin{equation}\label{eq:push-red}
\includegraphics[valign=c, scale=0.7]{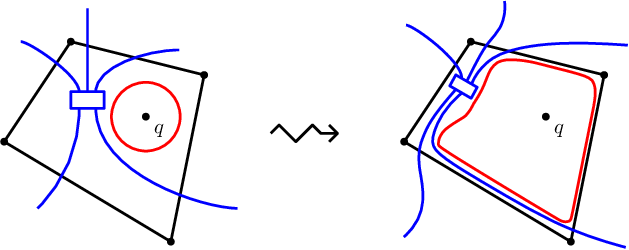}
\end{equation}
For every 2-cell $\twoCell{W} \in \de_2$, we define its belt corresponding to $Q$, denoted by $\belt_Q(\twoCell{W})$, as follows:
\begin{equation}\label{eq:belt}
\belt_{Q}(\twoCell{W}) := 
\begin{cases}
\bd U_q & \text{if $\twoCell{W} \subset \bd M$ and $\twoCell{W} \cap Q = q$;}\\
\twoCell{W} \cap \bd(\nb(\oneCell{M}_{\de})) & \text{if $\Int(\twoCell{W}) \subset \Int(M)$.}
\end{cases}
\end{equation}
Here, the $\bd U_q$ on the right hand side is understood as the boundary of the enlarged neighbourhood of $q$.

For each 2-cell $\twoCell{W} \in \loopEdge$, add a red loop along $\belt_Q(\twoCell{W})$. Pushing $\Gamma$ to the surface $\partial \nb(M_\delta^{(1)})$, and taking its union with the red belt loops (which are already on the surface), we obtain a handlebody $\nb(M^{(1)}_{\de})$ with an admissible bichrome graph $\Gamma_{Q}(\spnTree)$ on its boundary:
\begin{equation}\label{eq:push-to-hdle}
\includegraphics[valign=c, scale=0.7]{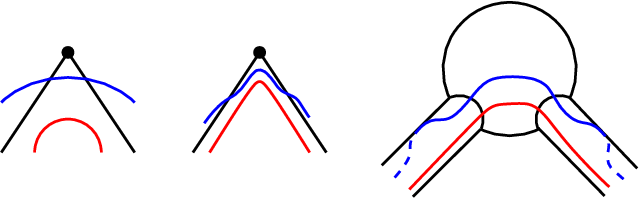}
\end{equation}

Using the above notations, we give the following definition.

\begin{definition}\label{def:hb-of-M}
Let $M$ be a connected compact oriented 3-manifold and $\Gamma$ a bulk-admissible $\cC$-graph on $\bd M$. 
For a choice of $\de$, $\spnTree$ and $Q$ as in the above discussions, we define the associated \emph{handlebody bichrome graph} of $(M, \Gamma)$ to be 
\begin{equation}\label{eq:def-hb}
\fH_{\de}(M, \Gamma, Q, \spnTree) := (\nb(M^{(1)}_{\de}), \Gamma_{Q}(\spnTree))\,.
\end{equation}
\end{definition}

We can now define an invariant of bulk-admissible $\cC$-graphs in arbitrary compact oriented 3-manifolds.

\begin{definition}\label{def:tau}
Let $(\cC,\mtr,\pi_\1)$ be as in \eqref{eq:oc-invariant-alg-input}, $M$ a compact oriented 3-manifold with a bulk-admissible $\cC$-graph $\Gamma$ in $\bd M$. 
Assume that $(M, \Gamma) = \bigsqcup_{i=1}^{m}(M_i, \Gamma_i)$ is the decomposition of $(M, \Gamma)$ into connected components. We define 
\[\tau_{\cC}(M, \Gamma) := \cgpt^\mathrm{bi}(\bigsqcup_{i=1}^{m} \fH_{\de_{i}}(M_i, \Gamma_i, Q_i, \spnTree_i)) = \prod_{i=1}^{m} \cgpt^\mathrm{bi}(\fH_{\de_{i}}(M_i, \Gamma_i, Q_i, \spnTree_i))\]
where for each $1 \le i \le m$, $\de_{i}$ is a generic PLCW decomposition of $(M_{i}, \Gamma_{i})$, $\spnTree_{i}$ is a spanning tree of $\ourGraph_{\de_{i}}(M_{i})$, and $Q_{i}$ is a set of interior points in 2-cells on $\partial M$ that avoids $\Gamma_{i}$.
\end{definition}

\begin{thm}\label{thm:tau-well-def}
Let $(\cC,\mtr,\pi_\1)$ be as in \eqref{eq:oc-invariant-alg-input}, $M$ a compact oriented 3-manifold with a bulk-admissible $\cC$-graph $\Gamma$ in $\bd M$. Then the invariant $\tau_{\cC}(M, \Gamma)$ in Definition \ref{def:tau} does not depend on the choice of $\de_{i}$, $\spnTree_{i}$ and $Q_{i}$. Moreover, $\tau_{\cC}(M, \Gamma)$ depends only on the orientation preserving diffeomorphism class of $(M, \Gamma)$.
\end{thm}

The proof of the theorem will be carried out in the next subsection. Before that, let us comment on how $\tau_{\cC}$ can be seen as an extension of the handlebody invariant $\cgpt$.

\medskip

Recall the notion of admissible skein modules from Definition~\ref{def:sk-adm}. For a $\bd$-marked surface $(\Sigma,L)$ we define the \textit{1-admissible skein module} 
\begin{equation}
    \sk_{1\text{-adm}}(\Sigma,L)     
\end{equation}
in the same way as $\sk_{\adm}(\Sigma,L)$, except that we only require the graphs in $\sk_{1\text{-adm}}(\Sigma,L)$ to have a projective edge in \textit{at least one} connected component of $\Sigma$, rather than in all of them. Similarly, 1-admissible skein relations are relations for discs such that there is a projective edge in the complement somewhere on $\Sigma$, not necessarily in the component the disc is placed in. Of course, if $\Sigma$ is connected, $\sk_{1\text{-adm}}(\Sigma,L)=\sk_{\adm}(\Sigma,L)$. 

One difference between $\sk_{\adm}(\Sigma,L)$ and $\sk_{1\text{-adm}}(\Sigma,L)$ is that for $\sk_{\adm}(\Sigma,L)$, taking the disjoint union of graphs gives an isomorphism (see e.g.\ \cite[Lem.\,2.9]{RST24})
\begin{equation}
    \sk_{\adm}(\Sigma,L) \otimes \sk_{\adm}(\Sigma',L') \to \sk_{\adm}(\Sigma \sqcup \Sigma',L \sqcup L') \ ,
\end{equation}
while for $\sk_{1\text{-adm}}(\Sigma,L)$ it does not. Indeed, just take two copies of $\bS^2$. Then $\sk_{1\text{-adm}}(\bS^2,\emptyset)$ is 1-dimensional with basis given by the graph $\Gamma$ with one edge labelled $P_\1$ and morphisms $\pi_\1$ and $\iota_\1$ on the vertices. However, the disjoint union $[\Gamma \sqcup \Gamma]$ is zero in $\sk_{1\text{-adm}}(\bS^2 \sqcup \bS^2,\emptyset)$.

The reason to introduce $\sk_{1\text{-adm}}$ is that it describes the skein-invariance of $\tau_\cC$ for connected manifolds:

\begin{prop}\label{prop:skein}
Let $M$ be a connected 3-manifold. Then $\tau_\cC$ gives a well-defined linear map
\[
    \tau_\cC(M,-) : \sk_{\operatorname{1-adm}}(\bd M,\emptyset) \to \kk \ .
\]
\end{prop}

\begin{proof}
Let $\Gamma$ be a 1-admissible graph on $\bd M$. Since $M$ is connected this implies (and is in fact equivalent to) $\Gamma$ being bulk-admissible. We need to show that the assignment $\Gamma \mapsto \tau_\cC(M,\Gamma)$ factors through $\sk_{1\text{-adm}}(\bd M,\emptyset)$, i.e., that it respects skein relations. Let $K \subset \bd M$ be an embedded disk as in \eqref{eq:skein-K-example}. Pick a PLCW decomposition $\delta'$ of $\bd M$ such that $K$ is a neighbourhood of a vertex of $\delta'$ which does not contain a full 1-cell or 2-cell of $\delta'$. Then extend $\delta'$ to a PLCW-decomposition of $M$ (cf.\ Proposition~\ref{prop:ext} below).
Pick the set of points $Q$ such that it avoids $K$. Pick an arbitrary spanning tree $\spnTree$ and let $\fH_{\de}(M, \Gamma, Q, \spnTree)$ be the associated handlebody bichrome graph. By 
Theorem \ref{thm:CGP23}, the assignment
\[
\tau_{\cC}(M, \Gamma) = \cgpt^\mathrm{bi}(\fH_{\de}(M, \Gamma, Q, \spnTree)) 
\]
is independent of these choices. By construction,
$\fH_{\de}(M, \Gamma, Q, \spnTree)$ inherits the embedding of $K$ and satisfies $K \cap \Gamma = K \cap \Gamma_{Q}(\spnTree)$. As the boundary of $\nb(M^{(1)}_{\de})$ is connected, the 1-admissible skein relation in $\sk_{1\text{-adm}}(\bd M,\emptyset)$ becomes an admissible skein relation in $\sk_{\adm}^\mathrm{bi}(\bd \nb(M^{(1)}_{\de}),\emptyset)$ which holds by construction of $\cgpt^\mathrm{bi}$.
\end{proof}

The next proposition shows that $\tau_\cC$ extends the handlebody invariant $\cgpt$ in Theorem~\ref{thm:CGP23} in the sense that for handlebodies $H$, both define the same linear form
\begin{equation}
    \tau_\cC(H,-) = \cgpt(H,-) : \sk_{\adm}(\bd H,\emptyset) \to \kk \ .
\end{equation}

\begin{prop}\label{prop:hdby}
Let $H$ be a handlebody. Then for any admissible $\cC$-graph $\Gamma$ in $H$, we have 
$\tau_{\cC}(H, \Gamma) = \cgpt(H, [\Gamma])$.
\end{prop}

\begin{figure}[tb]
\raisebox{4em}{a)} 
 \hspace{1em} \includegraphics[valign=c,scale=0.6]{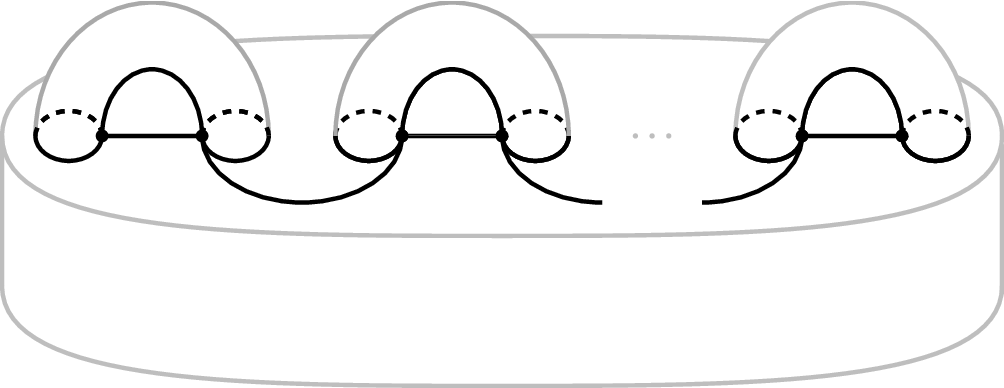}
\\[1em]
\raisebox{4em}{b)} 
\includegraphics[valign=c,scale=1]{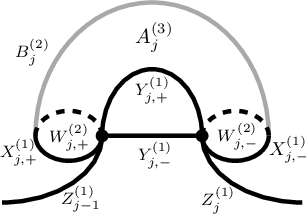} 
\hspace{1em}
\raisebox{4em}{c)} 
\hspace{1em}
\includegraphics[valign=c,scale=0.4]{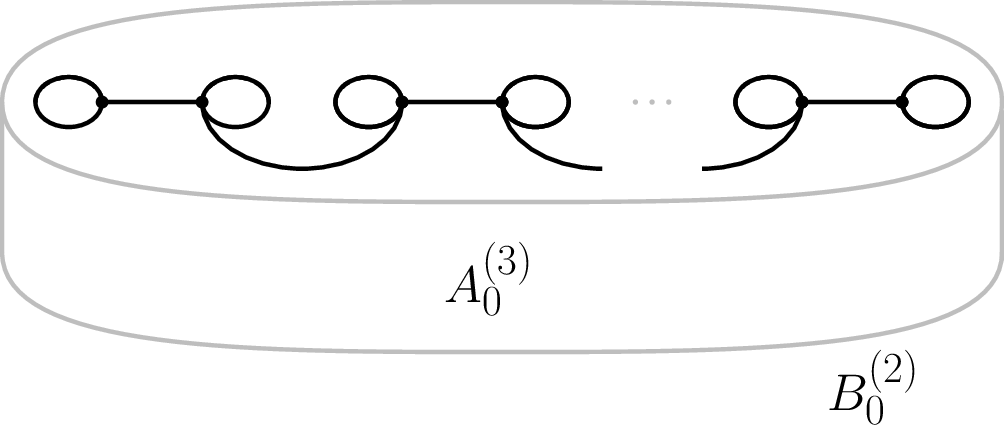} 
\caption{
a) PLCW decomposition of a genus-$g$ handlebody represented as $g$ handles attached to a 3-ball. The black lines represent the 1-cells, and the gray lines are only used to indicate the shape of the handlebody. \\
b) For $j=1, ..., g$, the $j$-th handle has two 0-cells, four 1-cells $\oneCell{X}_{j,\pm}$, $\oneCell{Y}_{j, \pm}$, three 2-cells $\twoCell{W}_{j,\pm}$ and $\twoCell{B}_{j}$ and one 3-cell $\threeCell{A}_j$. 
All the handles are attached to a ball, which has additional 1-cells $\oneCell{Z}_{j}$ connecting the 0-cells of nearby handles. \\
c) The ball itself has one 2-cell $\twoCell{B}_{0}$ and one 3-cell $\threeCell{A}_0$.}
\label{fig:genus-g-PLCW}
\end{figure}

\begin{proof}
Assume $H$ is of genus $g$. Consider the PLCW decomposition $\de$ of $(H, \Gamma)$ in Figure~\ref{fig:genus-g-PLCW}.
Perturb $\Gamma$ so that the above PLCW decomposition is generic. It is easy to see that the (2,3)-graph $\ourGraph_{\de}(H, \Gamma)$ of $(H, \Gamma)$ with respect to $\delta$ is of the following form:
\begin{equation}\label{eq:4-7} \includegraphics[valign=c,scale=0.9]{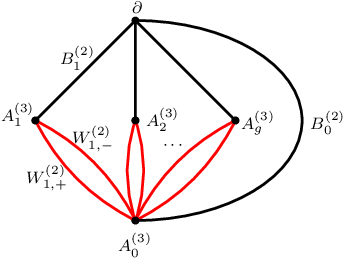}
\end{equation}
Moreover, one immediately sees that a spanning tree $\spnTree$ of $\ourGraph_{\de}(H, \Gamma)$ can be taken to contain only edges $\twoCell{B}_{j}$ for $j=0, 1, ..., g$. As an illustration, in \eqref{eq:4-7}, the black edges are edges in $\spnTree$, and the red edges correspond to 2-cells whose belts are to be decorated by red loops in $\nb(H_{\de}^{(1)})$ (see \eqref{eq:4-8}). 

Choose interior points $Q$ of the boundary 2-cells $B^{(2)}_j$ arbitrarily, then $\nb(H_{\de}^{(1)})$ is of the following form: 
\begin{equation}\label{eq:4-8} 
\includegraphics[valign=c,scale=0.5]{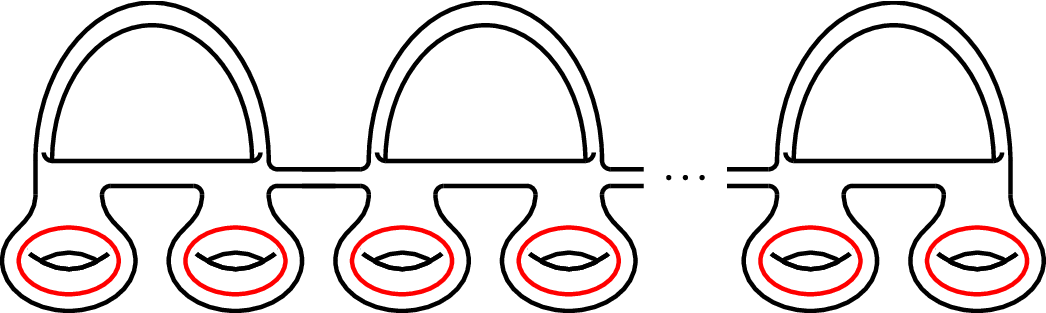}
\end{equation}
where the red loops are added to the handles corresponding to the boundary of $\twoCell{W}_{j,\pm}$, which are not in the spanning tree $\spnTree$ by construction. 
In other words, we have 
\[\fH_{\de}(H, \Gamma, Q, \spnTree) = (\nb(H_{\de}^{(1)}), \Gamma_Q(\spnTree)) = (\nb(H_{\de}^{(1)}), \Gamma\cup\bigcup_{j=1}^{g} \gamma_{j,\pm})\]
where $\gamma_{j,\pm}$ stands for the red loops along $\belt(\twoCell{W}_{j,\pm})$.
Note also that the holes which are not decorated by red loops are in one-to-one correspondence with the 2-cells $\twoCell{B}_{j}$ for $j=1, ..., g$. 
By applying the capping move \eqref{eq:cap-dig-1}, we have 
\[\Cap_{\{\gamma_{j,\pm}\mid 1\le j \le g\}}(\fH_{\de}(H, \Gamma, Q, \spnTree)) \cong (H, \Gamma)\]
which implies that 
\[
\tau_{\cC}(H, \Gamma) = \cgpt^\mathrm{bi}(\fH_{\de}(H, \Gamma, Q, \spnTree)) = \cgpt^\mathrm{bi}(\Cap_{\{\gamma_{j,\pm}\mid 1\le j \le g\}}(\fH_{\de}(H, \Gamma, Q, \spnTree))) = \cgpt(H, \Gamma) \ ,
\]
where the second equality is invariance of $\cgpt^\mathrm{bi}$ under the capping move (Lemma \ref{lem:hb-cap}), and the third equality is the observation that after the capping move no red curves remain, and so $\cgpt^\mathrm{bi}$ is just $\cgpt$ (Corollary~\ref{cor:F-H-bi-properties}\,(1)).
\end{proof}

\subsection{Comparison to the closed 3-manifold invariant in \texorpdfstring{\cite{CGPT20-Kup}}{}}\label{sec:CGPT-compare}

The invariant $\tau_{\cC}$ can be used to define an invariant of closed 3-manifolds in the following way: Let $M$ be a closed connected oriented 3-manifold, and let $\bB \subset M$ be an embedded 3-ball. 
Consider the $\cC$-graph $I$ formed of a single edge labelled $P_\1$ and two vertices labelled $\iota_\1$ and $\pi_\1$ respectively. Embed $I$ in $\bd (M \setminus \bB)$. Then 
\begin{equation}\label{eq:inv-of-M}
\mathscr{I}_{\cC}(M) := \tau_{\cC}(M\setminus \bB, I) \end{equation}
is an invariant of closed connected 3-manifolds. 
For example, if $M=\bS^3$ then $M\setminus \bB = \bD^3$, the closed 3-disc, and we find
\begin{equation}
  \mathscr{I}_{\cC}(\bS^3) 
  \overset{(1)}=  
  \tau_{\cC}(\bD^3, I) 
  \overset{(2)}= 
  \cgpt(\bD^3, I) 
  \overset{(3)}= 
  \mtr_{P_\1}(\iota_\1 \circ \pi_\1) 
  \overset{(4)}= 1 \ .
\end{equation}
Step 1 is the definition of $\mathscr{I}_{\cC}$, step 2 uses that $\bD^3$ is a handlebody and so by Proposition~\ref{prop:hdby}, $\tau_\cC$ agrees with $\cgpt$. Step 3 uses the relation between $\cgpt$ and the modified trace in Theorem~\ref{thm:CGP23}\,(2), and finally step 4 is the normalisation condition \eqref{eq:t-i-pi-norm}.

In \cite[Thm.\,2.4]{CGPT20-Kup}, the authors define an invariant $\mathcal{K}_{\cC}(M)$ of closed connected 3-manifolds $M$ using Heegaard splittings. Namely, let $M = H_\alpha \cup_\Sigma H_\beta$, where $H_\alpha$ and $H_\beta$ are handlebodies and $\Sigma$ is their common boundary. Let $\beta_1,\dots,\beta_g$ be a reducing set of bounding circles for $H_\beta$ (each $\beta_i$ bounds an embedded disc in $H_\beta$, one per handle, see \cite{CGPT20-Kup} for details). Build an admissible bichrome graph $\Gamma_\beta$ on $\Sigma = \partial H_\alpha$ by declaring the $\beta_i$ to be red loops and adding the graph $I$ from above anywhere away from the red loops. Then $\mathcal{K}_{\cC}(M) = \cgpt^\mathrm{bi}(H_\alpha,\Gamma_\beta)$.
We claim that this agrees with the invariant from the previous paragraph: 

\begin{proposition}\label{prop:compare-closed}
For every closed connected 3-manifold $M$, we have
$\mathcal{K}_{\cC}(M) = \mathscr{I}_{\cC}(M)$. 
\end{proposition}

\begin{proof}
Start by picking a PLCW-decomposition $\delta$ of $M$ and consider its (2,3)-graph $\ourGraph_{\de}(M)$. Its genus is $g = |\delta_2| - |\delta_3| + 1$, and this is also the genus of the two complementary handlebodies $H_\alpha = \nb(M_\delta^{(1)})$ and $H_\beta$ the closure of the complement $M \setminus H_\alpha$. In the construction of $\mathcal{K}_{\cC}(M)$ we add $g$ red curves to $H_\alpha$, one per handle of $H_\beta$. Let us see how that comes out of the the spanning tree construction. $\ourGraph_{\de}(M)$ has $|\delta_3|$ vertices and $|\delta_2|$ edges. By Remark~\ref{rmk:cycle-switch}\,(1), every spanning tree $\fT$ of $\ourGraph_{\de}(M)$ has $|\delta_3|-1$ edges. This would correspond to $|\ourGraph_{\de}(M)| - |\fT| = |\delta_2| - |\delta_3|+1 = g$ edges, as required. Indeed, pick a basis $C_1,\dots,C_g$ of the cycles in $\ourGraph_{\de}(M)$, and construct a spanning tree $\fT$ by omitting one edge per cycle. Then, conversely, in the construction of $\tau_\cC$ we would add one red loop per omitted edge, which precisely gives the red loops on $H_\alpha$ in the construction of $\mathcal{K}_{\cC}(M)$.

Of course, $M$ has empty boundary, so this is not yet the construction of $\tau_\cC$. Instead we should take a PLCW decomposition of $M \setminus \bB$ with graph $I$ on the boundary. We modify the PLCW-decomposition $\delta$ by arbitrarily picking a 2-cell $A^{(2)}$ and a 1-cell $L^{(1)}$ in the boundary of $A^{(2)}$. We may assume that $A^{(2)}$ sits on the boundary of two distinct 3-cells $D^{(3)}_{1,2}$.  
Modify $\delta$ to $\delta'$ as follows:
\begin{equation}
\includegraphics[valign=c, scale=0.6]{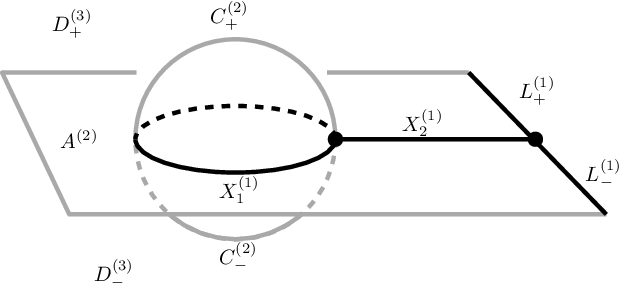}    
\end{equation}
Here we have split the 1-cell $L^{(1)}$ into $L^{(1)}_\pm$ by adding a new vertex, added two 1-cells $X^{(1)}_{1,2}$ and 2-cells $C^{(2)}_{\pm}$, as well as one 0-cell (in addition to the vertex on $L^{(1)}$). Note that the region between $C^{(2)}_{+}$ and $C^{(2)}_{-}$ is empty as this is where we have removed the 3-ball $\bB$.
The (2,3)-graph changes locally as follows:
\begin{equation}
\includegraphics[valign=c, scale=0.6]{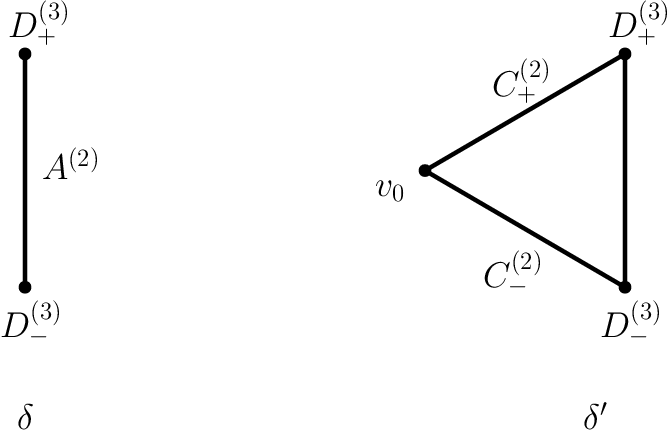}    
\end{equation}
For the new spanning tree $\fT'$ we add one edge, say $X^{(1)}_1$, to $\fT$. Then locally, near the neighbourhood of $L^{(1)}$, $H_\alpha$ changes as follows
\begin{equation}
\includegraphics[valign=c, scale=0.4]{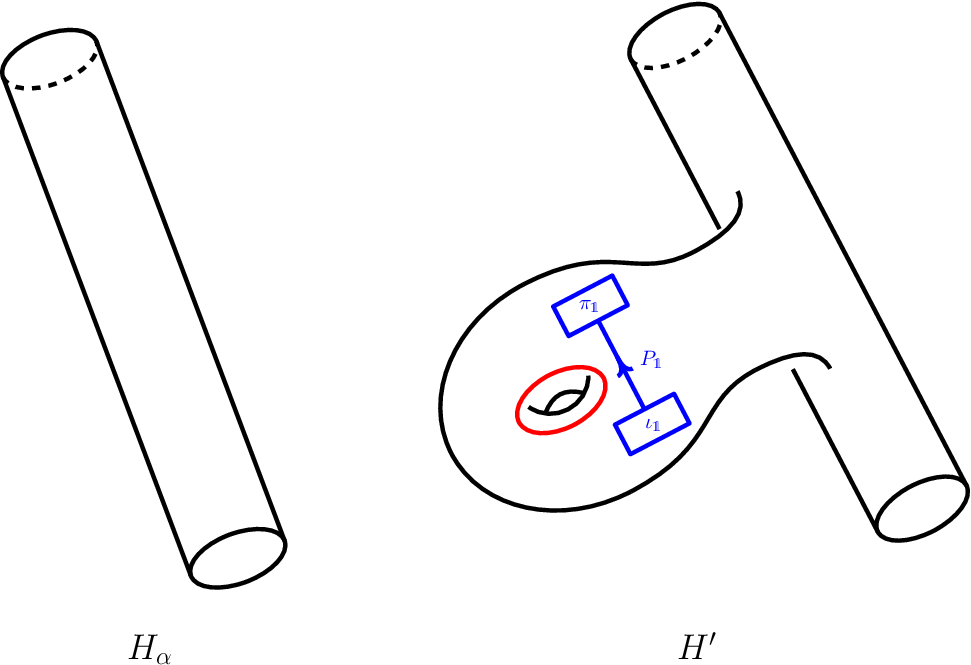}    
\end{equation}
Outside of the displayed region, $H_\alpha$ and $H'$ agree. Executing a red capping move (Lemma~\ref{lem:hb-cap}), we obtain $H_\alpha$ with graph $\Gamma$. 
\end{proof}

\begin{remark}\label{rmk:atfd-1}
Consider the case that $\cC$ is in addition semisimple with non-zero global dimension, and recall from Remark~\ref{rem:ssi-chromatic} the chromatic map $c_{P,G}$.
Applying the red-to-blue map to a red loop is equivalent to replace the red loop by $G$ with a coupon $\sum_{U\in\irr(\cC)} \dim(U) \id_U$ inserted on it. Note that since $\1$ is projective and has non-zero dimension, every $\cC$-graph (including the empty graph) on the boundary of a 3-manifold is automatically bulk-admissible, because we can add contractible loops coloured by $\1$ wherever needed. 

By discussions in \cite[Sec.\ 3.2]{atfd1} and \cite[Sec.\ 2]{CGPT20-Kup}, for any graph $\Gamma$ on a multi-handlebody $H$, the invariant $\cgpt(H, \Gamma)$ is a non-zero scalar multiple of the alterfold invariant $Z_{\cC}(H \subset \bS^3, \bd H, \Gamma)$. Here, $H$ is embedded in $\bS^3$ as the $B$-coloured region, and its complement is the $A$-coloured region. Consequently, by Proposition \ref{prop:hdby}, the alterfold invariant $Z_{\cC}(H \subset \bS^3, \bd H, \Gamma)$ is a scalar multiple of $\tau_{\cC}(H, \Gamma)$. 
Moreover, all the invariants above generalise the Turaev-Viro invariant in the following sense: Let $M$ be a closed 3-manifold, then by \cite[Thm.\ 2.7]{CGPT20-Kup} and \cite[Prop.\ 3.14]{atfd1} and Proposition \ref{prop:compare-closed}, we have 
$\operatorname{TV}_{\cC}(M) = Z_{\cC}(M_{B}, \emptyset, \emptyset) = \cK_{\cC}(M)/\mathrm{Dim}(\cC) = \mathscr{I}_{\cC}(M)/\mathrm{Dim}(\cC)$, 
where $M_B$ is the $B$-coloured alterfold $M$ (with empty $A$-coloured region).

More generally, consider a 3-manifold $M$ with a $\cC$-graph $\Gamma$ on its boundary. Construct a 3-alterfold as follows: for each connected component $J_i$ of $\bd M$, let $H_i$ be a handlebody of the same genus as $J_i$, and define $M' = M \cup \bigcup_i H_i$. Let $M$ be the $B$-coloured region and the $H_i$'s are the $A$-coloured region, we get an alterfold $M'$ with a $\cC$-graph $\Gamma$ on its separating surface. Then $Z_{\cC}(M', \bd M, \Gamma)$ is a scalar multiple of $\tau_{\cC}(M, \Gamma)$. We mention the following key ingredients in the verification of this statement, leaving the details to the interested reader: (1) the red capping move of $\cgpt$ corresponds to alterfold Move 1, and the cutting move corresponds to alterfold Move 2; (2) the difference in the scalar multiple of the two invariants is caused by the use of spanning trees, see Remark \ref{rmk:cycle-switch} (1).
\end{remark}

\subsection{Proof of Theorem \ref{thm:tau-well-def}}
Throughout this subsection, let $(\cC,\mtr,\pi_\1)$ be as in \eqref{eq:oc-invariant-alg-input}, and let $(M, \Gamma)$ be a connected compact oriented 3-manifold with a bulk-admissible $\cC$-graph $\Gamma$ on the boundary $\bd M$. 
We prove Theorem \ref{thm:tau-well-def} by showing the independence of $\tau_{\cC}$ of the choices of the spanning trees, the interior points of boundary 2-cells and the PLCW decompositions. We will continue using the conventions in the previous subsection. 

We will use sliding moves (Definition \ref{def:slide}) repeatedly in the arguments below, and we have to make sure that the condition in Corollaries \ref{cor:slide-X} and \ref{cor:red-over-red} is satisfied, namely that the complement of the red loop across which we are sliding a blue or red strand is connected. We note here once that this is indeed the case for all red loops on $\fH_{\de}(M, \Gamma, Q, \spnTree)$: the complement of each individual attaching belt of a 2-cell in $\nb(M^{(1)}_{\de})$ is connected.

\begin{lemma}\label{lem:T-indep}
Let $\spnTree$ and $\spnTree'$ be two spanning trees of $\ourGraph_{\de}(M)$, then $\fH_{\de}(M, \Gamma, Q, \spnTree')$ can be obtained from $\fH_{\de}(M, \Gamma, Q, \spnTree)$ by sliding moves among the red loops on the boundary surface $\Sigma := \bd\nb(M^{(1)}_{\de})$.
\end{lemma}
\begin{proof}
By construction, the only difference between $\fH_{\de}(M, \Gamma, Q, \spnTree)$ and $\fH_{\de}(M, \Gamma, Q, \spnTree')$ is that the red loops in $\Gamma(\spnTree)$ and $\Gamma(\spnTree')$ are at different positions on the boundary surface $\Sigma$, which correspond to the different subsets of 2-cells $\loopEdge$ and $\loopEdge'$ respectively. 

By Remark \ref{rmk:cycle-switch}\,(5), we only have to consider the case when $\spnTree$ and $\spnTree'$ are different in 2 edges $\{e\} = \spnTree\setminus\spnTree'$ and $\{e'\} =\spnTree'\setminus\spnTree$ that share a common vertex $v \in e \cap e'$. 
Let $L_e$ and $L_{e'}$ denote belts of the 2-cells corresponding to $e$ and $e'$ respectively. 

In this case, the only difference between $\fH_{\de}(M, \Gamma, Q, \spnTree)$ and $\fH_{\de}(M, \Gamma, Q, \spnTree')$ is that in $\fH_{\de}(M, \Gamma, Q, \spnTree)$, $L_{e'}$ is decorated by a red loop (as $e' \in \loopEdge$), but $L_{e}$ is not. 
In $\fH_{\de}(M, \Gamma, Q, \spnTree')$, we have the opposite: $L_{e}$ is decorated by a red loop, but $L_{e'}$ is not. 
Hence, in order to show that $\fH_{\de}(M, \Gamma, Q, \spnTree')$ can be obtained from $\fH_{\de}(M, \Gamma, Q, \spnTree)$ by sliding moves, we only have to show that $L_{e'}$ can slide along the other red loops of $\fH_{\de}(M, \Gamma, Q, \spnTree)$ on $\Sigma$ to result in $L_e$. 

The following observation enables us to further simplify the question: Sliding a strand of an admissible bichrome graph in a surface over a red loop is effectively the same as isotoping the strand over one side of the boundary of the 2-handle attached to the surface along the red loop. 
Let $\fA$ be the 3-manifold with boundary obtained as follows: Attach a 2-handle to $\fH_{\de}(M, \Gamma, Q, \spnTree)$ along each of its red loops except for the red loop at $L_{e'}$, see below: 
\begin{equation}
\includegraphics[valign=c,scale=0.5]{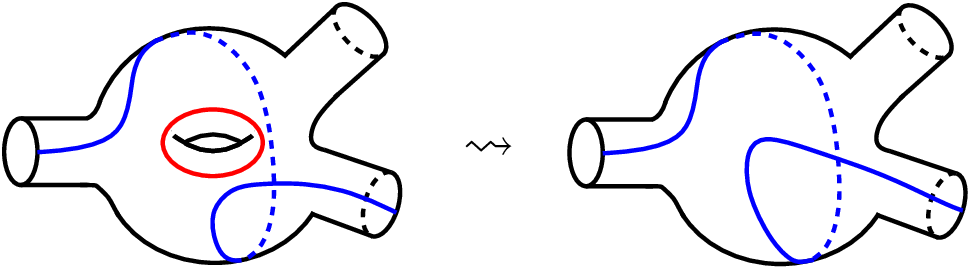}
\end{equation}

Then to finish the proof of the lemma, it suffices to show that $L_{e'}$ can be isotoped to $L_{e}$ in $\fA$. In view of Remark \ref{rmk:our-graph}, $\fA$ is obtained from (the closure of) the complement of $\spnTree\setminus\{v_0\}$ in $M$ by removing the 2-handle attached along $L_{e'}$.

Let $f(v, v_0)$ be the edge at $v$ that is contained in the unique path in $\spnTree$ connecting $v$ and $v_0$ (if $v=v_0$, we set $f(v, v_0)$ to be empty). We have the following two cases: 

\smallskip

\noindent
\textbf{Case 1.}
$f(v, v_0) = e$. In this case, denote the connected component of $\spnTree\setminus\{e\}$ containing $v$ by $B$, then $B$ as a subtree of $\spnTree$ is not connected to $v_0$, and the boundary the tubular neighbourhood of $B$ in $\fA$ (see Remark \ref{rmk:our-graph}) is homeomorphic to an annulus $\bS^1 \times [0, 1]$ in $\fA$ having $L_{e}$ and $L_{e'}$ as its two boundary circles. Hence, $L_{e'}$ can be isotoped to $L_{e}$ in $\fA$.
\begin{equation}
\includegraphics[valign=c,scale=0.5]{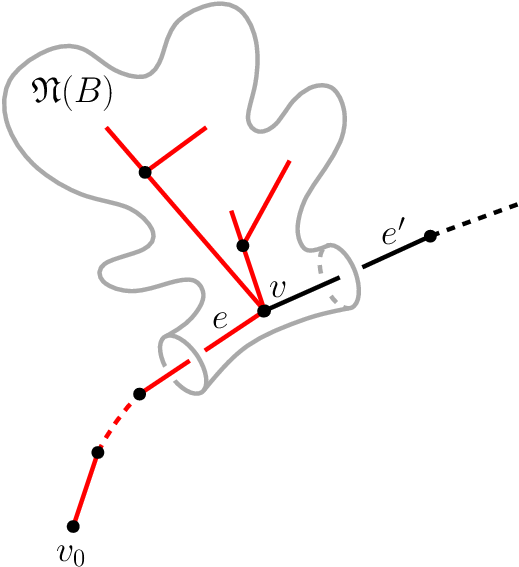}
\end{equation}

\smallskip

\noindent
\textbf{Case 2.}
$f(v, v_0) \neq e$. Let $E(v, \spnTree)$ be the set of edges of $\spnTree$ that has $v$ as a vertex. By construction, $e$, $e'$ and possibly some other edges in $\spnTree$ form a cycle $C$ in $\ourGraph_{\de}(M)$, and the connected component $B$ of $\spnTree\setminus(E(v, \spnTree)\setminus \{e\})$ containing $v$ has $C \setminus \{e'\}$ as a subtree. By assumption, $B$ is not connected to $v_0$, and by the same argument as above, the boundary of the tubular neighbourhood of $B$ in $\fA$ is isomorphic to an annulus $\bS^1 \times [0, 1]$ in $\fA$ having $L_{e}$ and $L_{e'}$ as its two boundary circles.
\begin{equation}
\includegraphics[valign=c,scale=0.5]{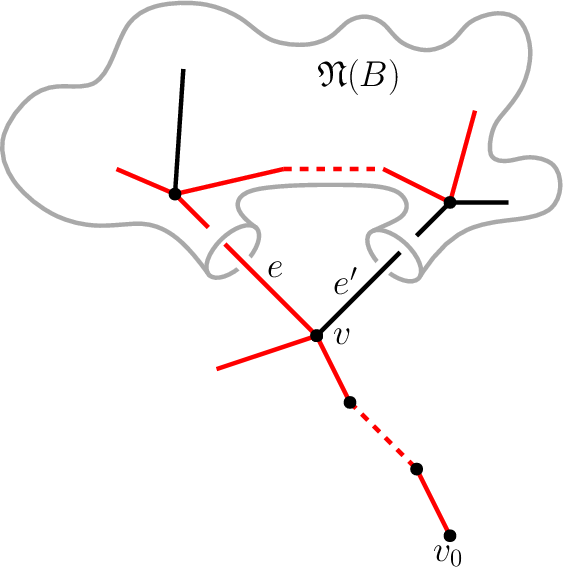}
\end{equation}
Hence, $L_{e'}$ can be isotoped to $L_{e}$ in $\fA$, and this completes the proof. 
\end{proof}

Now fix $\de$ and $\spnTree$. Let $Q$ and $Q'$ be two sets of interior points in 2-cells in $\bd M$ that avoid $\Gamma$, then $\fH_{\de}(M, \Gamma, Q, \spnTree)$ and $\fH_{\de}(M, \Gamma, Q', \spnTree)$ have the same underlying handlebody $\nb(M^{(1)}_{\de})$. By construction, since $\fH_{\de}(M, \Gamma, Q, \spnTree)$ and $\fH_{\de}(M, \Gamma, Q', \spnTree)$ are associated to the spanning tree $\spnTree$ of $\ourGraph_{\de}(M)$, so the bichrome graphs $\Gamma_{Q}(\spnTree)$ and $\Gamma_{Q'}(\spnTree)$ are identical except for the blue part on the boundary of the solid torus $\nb(\bd \twoCell{W})$ in $\nb(M^{(1)}_{\de})$. In particular, they have identical red parts.

\begin{lemma}\label{lem:Q-indep}
The bichrome graphs $\Gamma_Q(\spnTree)$ and $\Gamma_{Q'}(\spnTree)$ are related by a sequence of sliding moves along the red loops on $\Sigma := \bd\nb(M^{(1)}_{\de})$.
\end{lemma}
\begin{proof}
It suffices to prove the statement for the case when $Q$ and $Q'$ are only different in the choice of one point in one 2-cell. 
More precisely, in this situation, $Q\setminus Q' = \{q\}$ and $Q' \setminus Q = \{q'\}$ for two points $q$ and $q'$ contained in the interior of a boundary
2-cell $\twoCell{W} \subset \bd M$. 
Let $e$ be the edge of $\ourGraph_{\de}(M)$ corresponding to $\twoCell{W}$, then $e$ connects to $v_0$.
    
Now we have the following 2 cases.

\smallskip

\noindent
\textbf{Case 1.}
$e \notin \spnTree$. In this case, there is a red loop $C$ along the belt of $\twoCell{W}$ corresponding to $Q$ (resp.\ $Q'$) in $\Gamma_{Q}(\spnTree)$ (resp.~$\Gamma_{Q'}(\spnTree)$). By Corollaries \ref{cor:slide-X} and \ref{cor:red-over-red}, $\Gamma_{Q}(\spnTree)$ can be slided to $\Gamma_{Q'}(\spnTree)$ along the red loop $C$: 
\begin{equation}
\includegraphics[valign=c,scale=0.5]{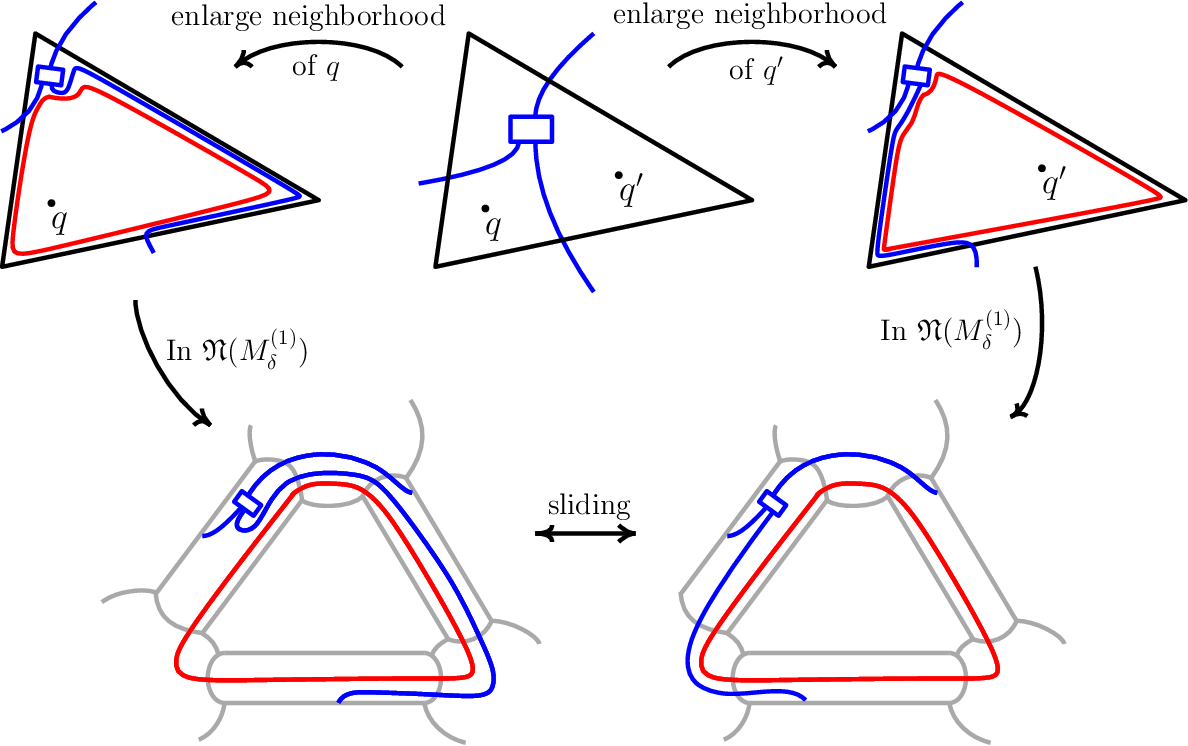}
\end{equation}

\noindent
\textbf{Case 2.}
$e \in \spnTree$. 
By the property of spanning trees, $\spnTree \setminus \{e\}$ is disconnected (see above Remark~\ref{rmk:cycle-switch}).
Consider the connected component of $\spnTree \setminus \{e\}$ that does not contain $v_0$, 
and take the union of this connected component with the half-edge of $e$ to which it connects. Denote the result by $B$. 

For each red loop of $\Gamma_{Q}(\spnTree)$ (or $\Gamma_{Q'}(\spnTree)$, recall that they have the same red parts), attach a 2-handle to $\nb(M^{(1)}_{\de})$, and denote the resulting 3-manifold by $\fA$. 
Then the boundary of the complement of $B$ in $M$ is homeomorphic to a disk $\bD^2$ in $\fA$ whose boundary circle is the longitude of the solid torus $\nb(\bd \twoCell{W})$. 
\begin{equation}
\includegraphics[valign=c,scale=0.5]{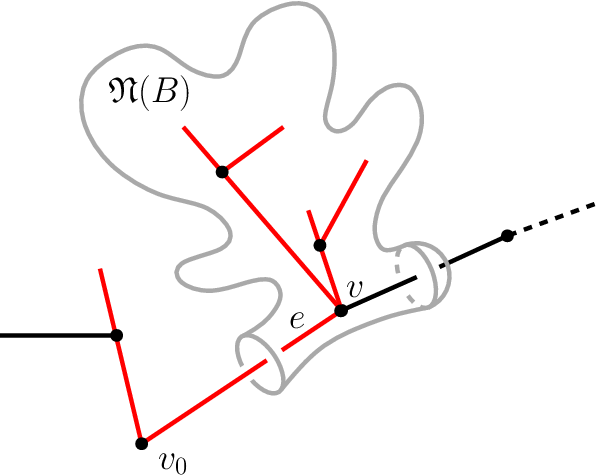}
\end{equation}

It is then easy to see that the blue part of $\Gamma_{Q}(\spnTree)$ that is contained in the boundary of $\nb(\bd \twoCell{W})$ can be isotoped in $\fA$ through the complement of $B$ to the corresponding blue part of $\Gamma_{Q'}(\spnTree)$. This is illustrated as follows:
\begin{equation}
\includegraphics[valign=c,scale=0.5]{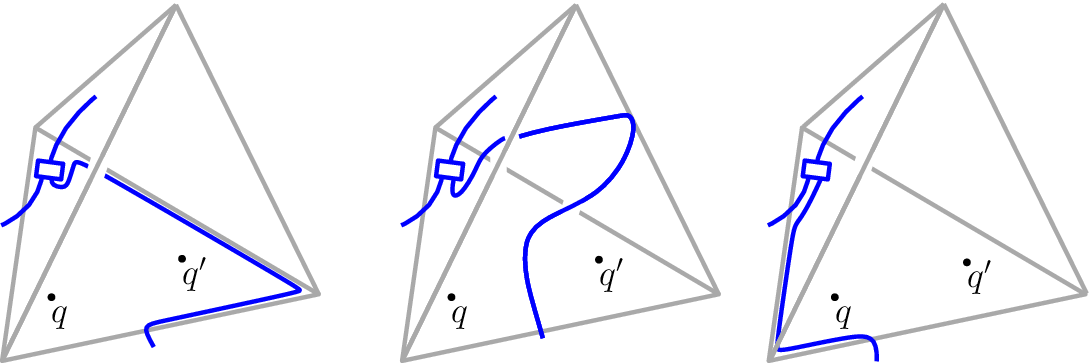}
\end{equation}

As is noted in the proof of Lemma \ref{lem:T-indep}, this implies that $\Gamma_Q(\spnTree)$ and $\Gamma_{Q'}(\spnTree)$ are related by a sequence of sliding moves along the red loops, and this completes the proof. 
\end{proof}

\begin{cor}\label{cor:T-Q-indep}
The value $\cgpt^\mathrm{bi}(\fH_{\de}(M, \Gamma, Q, \spnTree))$ does not depend on the choice of the spanning tree $\spnTree$ of $\ourGraph_{\de}(M)$ and the set of interior points $Q$ of 2-cells.  \qed
\end{cor}

\begin{lemma}\label{lem:de-indep}
Let $\de$ and $\de'$ be two generic PLCW decompositions of $(M, \Gamma)$. Then for any choice of spanning trees of $\ourGraph_{\de}(M)$ (resp.~$\ourGraph_{\de'}(M)$) and any choice of interior points $Q$ (resp.~$Q'$) of boundary 2-cells of $\de$ (resp.~$\de'$), we have \begin{equation}\label{eq:de-inv}
\cgpt^\mathrm{bi}(\fH_{\de}(M, \Gamma, Q, \spnTree)) = \cgpt^\mathrm{bi}(\fH_{\de'}(M, \Gamma, Q', \spnTree'))\,.
\end{equation}
\end{lemma}

\begin{proof}
According to \cite[Thm.\,8]{Kir12}, $\de$ and $\de'$ are related by a finite sequences of elementary subdivisions and their inverses. 
Therefore, it suffices to prove \eqref{eq:de-inv} when $\de'$ is obtained from $\de$ by an elementary subdivision of an $n$-cell in $\de$ for $n = 1, 2, 3$. 
Our proof of the lemma is hence divided into 3 cases, where in each one of them, we will make some assumptions on $\spnTree$, $\spnTree'$, $Q$ and $Q'$. 
Thanks to Corollary \ref{cor:T-Q-indep}, such assumptions can be made without loss of generality, and for simplicity, we will not refer to Corollary \ref{cor:T-Q-indep} every time we use it.

\smallskip

\noindent
\textbf{Case 1.} 
When $n=1$, $\de'$ is obtained from $\de$ by adding a 0-cell to a 1-cell $\oneCell{A}$. In this case, $\de$ and $\de'$ have identical sets of 2-cells and 3-cells, and in particular, $\ourGraph_{\de}(M) = \ourGraph_{\de'}(M)$. Hence, we can assume that $\spnTree = \spnTree'$ and $Q=Q'$. 
Moreover, it is easy to see that the subdivision of $\oneCell{A}$ only changes $\nb(M_\delta^{(1)})$ by fattening the 1-handle corresponding to $\oneCell{A}$, and change the local shape of $\Gamma$ on it accordingly:
\begin{equation}\label{eq:1-sub}
\includegraphics[valign=c,scale=0.7]{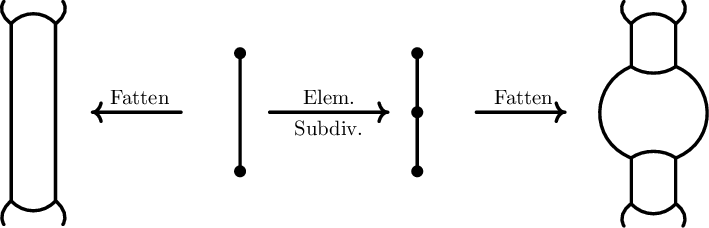}    
\end{equation}
Therefore, $\fH_{\de}(M, \Gamma, Q, \spnTree)$ is diffeomorphic to $\fH_{\de'}(M, \Gamma, Q', \spnTree')$, which implies that \eqref{eq:de-inv} holds.

\smallskip

\noindent
\textbf{Case 2.}
When $n = 2$, $\de'$ is obtained from $\de$ by adding a 1-cell $\oneCell{W}$ inside a 2-cell $\twoCell{A}$ such that $\oneCell{W}$ connects two 0-cells $\zeroCell{V}_{1}$ and $\zeroCell{V}_{2}$ in $\bd\twoCell{A}$, and $\oneCell{W}$ separates $\twoCell{A}$ into two 2-cells $\twoCell{A}_{1}$ and $\twoCell{A}_{2}$. A local picture is given as follows:
\begin{equation}
\includegraphics[valign=c,scale=0.8]{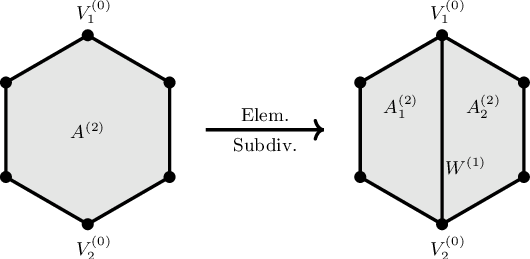}
\end{equation}

Now we compare the graph $\ourGraph_{\de}(M)$ with $\ourGraph_{\de'}(M)$. Let $e$ be the edge in $\ourGraph_{\de}(M)$ corresponding to $\twoCell{A}$, and let $v$, $v'$ be the endpoints of $e$. 
We allow $v$ and $v'$ to be equal when $e$ is a loop. 
Then by construction, $\ourGraph_{\de'}(M)$ can be obtained from $\ourGraph_{\de}(M)$ by adding an edge $e'$ that connects $v$ and $v'$:
\begin{equation}
\includegraphics[valign=c,scale=1]{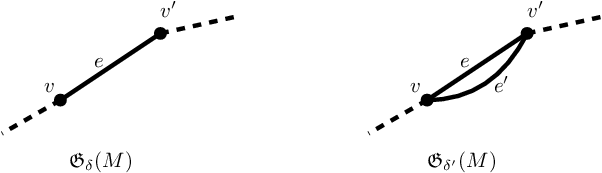}
\end{equation}
We take $e \in \ourGraph_{\de'}(M)$ to corresponds to the 2-cell $\twoCell{A}_{1}$ in $\de'$, and $e'$ to $\twoCell{A}_{2}$. 

By Corollary~\ref{cor:T-Q-indep}, the choice of spanning tree $\fT'$ for $\ourGraph_{\de'}(M)$ does not affect the value of $\cgpt^\mathrm{bi}$.
Since $e$ and $e'$ form a loop in $\ourGraph_{\de'}(M)$, we can assume that $\spnTree'$ does not contain $e'$. 
In addition, as $\ourGraph_{\de}(M)$ and $\ourGraph_{\de'}(M)$ have the same set of vertices and are only different in the edge $e'$, it is easy to see that $\spnTree'$ is also a spanning tree of $\ourGraph_{\de}(M)$. So we can assume that $\spnTree=\spnTree'$.

Again by Corollary~\ref{cor:T-Q-indep}, the choice of interior points $Q'$ in boundary 2-cells of $\de'_2$ does not affect the value of $\cgpt^\mathrm{bi}$.
Since the only difference between the 2-cells of $\de$ and those of $\de'$ are $\twoCell{A}$, $\twoCell{A}_{1}$ and $\twoCell{A}_{2}$, we can pick the same interior points for all the boundary 2-cells in $\de_2 \setminus \{\twoCell{A}\}=\de_2'\setminus \{\twoCell{A}_{1}, \twoCell{A}_{2}\}$.
In addition, since $\twoCell{A}$ is homeomorphic to $\twoCell{A}_{1} \cup \oneCell{W} \cup \twoCell{A}_{2}$, we can pick an interior point $x_{1} \in \twoCell{A}_{1}$, and assume that it is also the interior point we choose for $\twoCell{A}$, where we view $\twoCell{A}_{1}$ as a subset of $\twoCell{A}$. 
This implies the only difference in the choice of interior points of 2-cells in $\de$ and $\de'$ lies in the choice of the interior point $x_2$ in $\twoCell{A}_{2}$, i.e., $Q' = Q \cup \{x_2\}$.

By the above discussions, we can see that there is a red loop $\gamma$ in $\fH_{\de'}(M, \Gamma, Q', \spnTree')$ parallel to $\bd\twoCell{A}_{2}$:
\begin{equation}\label{eq:2-sub}
\includegraphics[valign=c,scale=0.5]{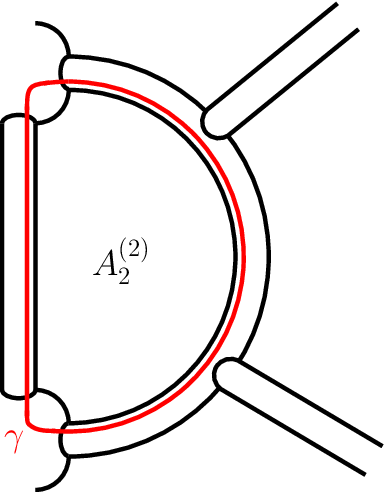}
\end{equation}
Thus, $\fH_{\de}(M, \Gamma, Q, \spnTree)$ can be obtained from $\fH_{\de'}(M, \Gamma, Q', \spnTree') = (\nb(M_{\de'}^{(1)}), \Gamma_{Q'}(\spnTree'))$ by a red capping move along $\gamma$. In other words, we have 
\[
\fH_{\de}(M, \Gamma, Q, \spnTree) 
= 
\Cap_{\gamma}(\fH_{\de'}(M, \Gamma, Q', \spnTree'))\,,
\]
which implies that \eqref{eq:de-inv} holds in this case due to Lemma \ref{lem:hb-cap}. 

Note that even if $\oneCell{W} \cap \Gamma \neq \emptyset$ (then $\twoCell{A}$ is necessarily contained in $\bd M$), the above argument is still valid, because in the tubular neighbourhood of $\oneCell{W}$ in $\nb(M_{\delta'}^{(1)})$, we can pull all other parts of the graph away from $\gamma$ so we can still perform the red capping move. 

\smallskip

\noindent
\textbf{Case 3.}
When $n = 3$, $\de'$ is obtained from $\de$ by adding a 2-cell $\twoCell{W}$ inside a 3-cell $\threeCell{A}$ such that the equator $\bd\twoCell{W}$ is a union of 1-cells in $\bd\threeCell{A}$, and $\twoCell{W}$ separates $\threeCell{A}$ into two 2-cells $\threeCell{A}_{1}$ and $\threeCell{A}_{2}$. 
In particular, $\de$ and $\de'$ have the same 1-skeleton, so we have $\nb := \nb(M_{\de}^{(1)}) = \nb(M_{\de'}^{(1)})$. Moreover, $\de_2 = \de'_2 \setminus \{\twoCell{W}\}$.

Next, we compare $\Gamma_{Q}(\spnTree)$ with $\Gamma_{Q'}(\spnTree')$. By construction, $\Gamma$ (after being pushed to $\bd\nb$) is a subgraph of $\Gamma_{Q}(\spnTree)$ and $\Gamma_{Q'}(\spnTree')$. Moreover, since $\de_2 = \de'_2 \setminus \{\twoCell{W}\}$, and $\intr(\twoCell{W})\subset \intr(M)$, the set of 2-cells in $\bd M$ in $\de$ and $\de'$ are the same, so by Corollary \ref{cor:T-Q-indep}, we can assume without loss of generality that $Q = Q'$. Consequently, the images of $\Gamma$ as subgraphs of $\Gamma_{Q}(\spnTree)$ and $\Gamma_{Q'}(\spnTree')$ respectively are identical. 

Then we investigate the red loops in $\Gamma_{Q}(\spnTree)$ and $\Gamma_{Q'}(\spnTree')$ that are added to $\bd\nb$ according to the choice of $\spnTree$ (resp.~$\spnTree'$).
Let $v_1$ be the vertex in $\ourGraph_{\de}(M)$ corresponding to $\threeCell{A}$, $v'_{1}$ (resp.~$v'_{2}$) the vertex in $\ourGraph_{\de'}(M)$ corresponding to $\threeCell{A}_{1}$ (resp.~$\threeCell{A}_{2}$), and $e'$ the edge in $\ourGraph_{\de'}(M)$ corresponding to $\twoCell{W}$. 
By construction, $\ourGraph_{\de'}(M)$ can be obtained from $\ourGraph_{\de}(M)$ by splitting the vertex $v_1$ into two distinct vertices $v'_1$ and $v'_2$ connected by one edge $e'$:
\begin{equation}
\begin{tikzcd}
\includegraphics[valign=c,scale=0.7]{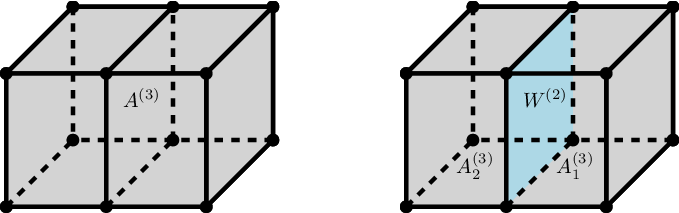}  \\
\includegraphics[valign=c,scale=0.7]{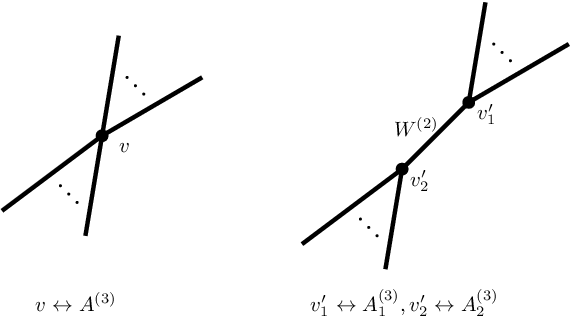}   
\end{tikzcd}
\end{equation}
In other words, $\ourGraph_{\de'}(M)$ has one vertex and one more edge than $\ourGraph_{\de}(M)$, and shrinking $\ourGraph_{\de'}(M)$ along $e'$ (so that $v'_{1}$ and $v'_{2}$ becomes a single vertex) will produce $\ourGraph_{\de}(M)$. 

Since $v'_1 \ne v'_2$, there must be a spanning tree $\spnTree'$ of $\ourGraph_{\de'}(M)$ containing $e'$:
Start with any spanning tree of $\ourGraph_{\de'}(M)$, if it does not contain $e'$, then we can carry out an edge exchange as in Remark~\ref{rmk:cycle-switch}\,(2)
to obtain a new spanning tree that does contain $e'$.
By shrinking $\spnTree'$ along $e'$, we get a spanning tree $\spnTree$ of $\ourGraph_{\de}(M)$. Indeed, shrinking along $e'$ removes one vertex and one edge, and so the condition in Remark~\ref{rmk:cycle-switch}\,(1) is still satisfied.
Thus, there is a 1-to-1 correspondence between the set of edges in $\ourGraph_{\de}(M) \setminus \spnTree$ and those in $\ourGraph_{\de'}(M) \setminus \spnTree'$. 
These sets of edges in turn correspond to the same subset of 2-cells in $\de_2 = \de'_2\setminus\{\twoCell{W}\}$, to whose boundary we add red loops to obtain $\Gamma_{Q}(\spnTree)$ and $\Gamma_{Q'}(\spnTree')$. 

Combining the above discussions, we conclude that $\Gamma_{Q}(\spnTree)$ and $\Gamma_{Q'}(\spnTree')$ have the same $\Gamma$ component and the same red loops corresponding to 2-cells, i.e., $\Gamma_{Q}(\spnTree) = \Gamma_{Q'}(\spnTree')$. 
Therefore, we have 
\[\fH_{\de}(M, \Gamma, Q, \spnTree)) = (\nb(M_{\de}^{(1)}), \Gamma_{Q}(\spnTree)) = (\nb(M_{\de}^{(1)}), \Gamma_{Q'}(\spnTree'))  = \fH_{\de'}(M, \Gamma, Q', \spnTree')\,,\]
which immediately implies \eqref{eq:de-inv}. 
This completes the proof.
\end{proof}

\begin{proof}[Proof of Theorem \ref{thm:tau-well-def}]
The statement follows directly from Lemma \ref{lem:de-indep} and Corollary \ref{cor:F-H-bi-properties}.
\end{proof}

\subsection{The cutting move}\label{subsec:cut}
The cutting move of  bulk-admissible $\cC$-graphs in general 3-manifolds is defined in the same way as that of bichrome graphs in multi-handlebodies in Section~\ref{sec:handlebody-inv}. 
More precisely, consider $(M, \Gamma)$ and let $D \subset M$ be a cutting disc, that is, an oriented properly embedded disk whose boundary $\bd D\subset \bd M$ does not meet the coupons of $\Gamma$ and intersects the strands of $\Gamma$ transversely in a non-empty set which contains at least one strand coloured by an object in $\Proj_{\cC}$. 
Denote by $\cut_D(M)$ the 3-manifold obtained by cutting $M$ along $D$ and denote by $\cut_D(\Gamma)$ the bichrome graph in $\cut_D(M)$ obtained by joining the cut points of $\Gamma$ to two new coupons in $\bd\cut_D(M)$, and colour the pair of coupons by the dual basis \eqref{eq:Lam-t-P-def}. Locally, we have 
\begin{equation}\label{eq:tau-cut-0}
\begin{tikzcd}[column sep = 5em]
\includegraphics[valign=c, scale=0.6]{pic/cut-pic-1.eps}
\ar[r, "\mathrm{Cutting}"]&
\sum\limits_{i}\ \includegraphics[valign=c, scale=0.6]{pic/cut-pic-2.eps}
\end{tikzcd}
\end{equation}

In this case, we say that $(\cut_D(M), \cut_D(\Gamma))$ is obtained from \((M, \Gamma)\) via the cutting move (along $D$). Note that by construction, if $\Gamma$ is bulk-admissible in $M$, then \(\cut_D(\Gamma)\) is bulk-admissible in \(\cut_D(M)\). 
 
\begin{prop}\label{prop:cut}
Let $M$ be a compact oriented 3-manifold and $\Gamma$ a bulk-admissible $\cC$-graph $\Gamma$ in $M$. For any cutting disk $D$ of $(M, \Gamma)$, we have 
\[\tau_\cC(M, \Gamma) = \tau_\cC(\cut_D(M), \cut_D(\Gamma))\,.\]
\end{prop}
\begin{proof}
Without loss of generality, assume that $M$ is connected. By Theorem \ref{thm:tau-well-def}, we can choose a generic PLCW decomposition $\de$ of $(M, \Gamma)$ such that a tubular neighbourhood of $D$ consists of the following cells: 
\begin{equation}\label{eq:tau-cut-1}
\includegraphics[valign = c, scale=0.6]{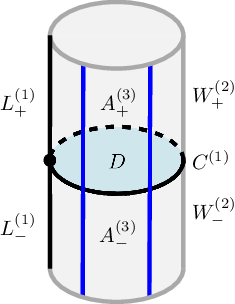}    
\end{equation}
Here, $D$ is a 2-cell that is bounded by one 1-cell $C^{(1)}$ and one 0-cell; the cylinder above and below $\bd D$ are parts of two 2-cells denoted by $\twoCell{W}_+$ and $\twoCell{W}_-$ respectively, whose boundaries also contain 1-cells $L_\pm^{(1)}$ above and below the 0-cell on $\bd D$. Finally, there is one 3-cell $\threeCell{A}_+$ above $D$, and one below $D$, denoted by $\threeCell{A}_-$.

Locally, the (2,3)-graph $\ourGraph_{\de}(M)$ is of the following form near $D$:
\begin{equation}
\includegraphics[valign=c,scale=0.7]{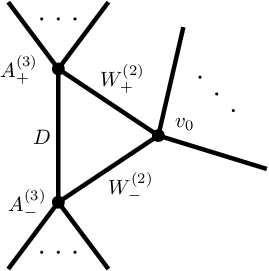}    
\end{equation}
Since the edges of $\ourGraph_{\de}$ corresponding to $D$, $\twoCell{W}_{+}$ and $\twoCell{W}_{-}$ form a cycle, there exists a spanning tree $\spnTree$ of $\ourGraph_{\de}$ such that $D \notin \spnTree$ (see Remark \ref{rmk:cycle-switch}\,(2)).
After picking a set $Q$ of interior points in 2-cells contained in $\bd M$, we obtain the handlebody $\fH_{\de}(M, \Gamma, Q, \spnTree)$ with bichrome diagrams on its boundary, and by definition,
\begin{equation*}
\tau_{\cC}(M, \Gamma) = \cgpt^\mathrm{bi}(\fH_{\de}(M, \Gamma, Q, \spnTree))\,.
\end{equation*}
Locally, the part of the handlebody $\fH_{\de}(M, \Gamma, Q, \spnTree)$ corresponding to the cells near $D$ is illustrated by the left hand side of \eqref{eq:H-Phi}
\begin{equation}\label{eq:H-Phi}
\includegraphics[valign=c,scale=0.6]{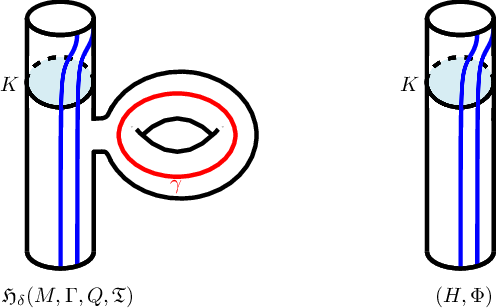}    
\end{equation}
where the genus is created by the 1-cell bounding $D$ and the red loop is present in $\fH_{\de}(M, \Gamma, Q, \spnTree)$ because $D \notin \spnTree$. Note that on the left hand side of \eqref{eq:H-Phi}, the vertical tube is part of the thickened 1-skeleton, and let $K$ be a horizontal section of this tube. We will later use the disc $K$ to perform cutting moves on multi-handlebodies with bichrome graphs.

Let $(H, \Phi) := \Cap_{\gamma}(\fH_{\de}(M, \Gamma, Q, \spnTree))$ be the handlebody with bichrome graph that is obtained from $\fH_{\de}(M, \Gamma, Q, \spnTree)$ by applying the red capping move along the red loop $\gamma$ (see \eqref{eq:cap-dig-1}). In particular, $\Phi = \Gamma \setminus \{ \gamma\}$, and by Lemma \ref{lem:hb-cap}, we have 
\begin{equation*}
\tau_{\cC}(M, \Gamma) = \cgpt^\mathrm{bi}(\fH_{\de}(M, \Gamma, Q, \spnTree)) = \cgpt^\mathrm{bi}(H, \Phi)\,.
\end{equation*}

Now consider $(\cut_D(M), \cut_D(\Gamma))$. We may assume that $\cut_D(M)$ admits the generic PLCW decomposition illustrated as follows:
\begin{equation}
\sum_{i}\quad
\includegraphics[valign=c,scale=0.6]{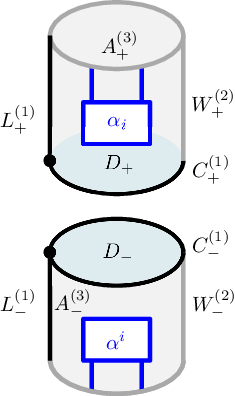}
\end{equation}
where on each side of the cut, we have a disk consisting of a 2-cell bounded by one 1-cell and one 0-cell, and other cells of $\cut_D(M)$ are identical to those of $M$. We denote the resulting PLCW decomposition of $(\cut_D(M), \cut_D(\Gamma))$ by $\de'$. 

If $\cut_D(M)$ is connected, then near the cut, the (2,3)-graph $\ourGraph_{\de'}(\cut_D(M))$ takes the form 
\begin{equation}\label{eq:tau-cut-5}
\includegraphics[valign=c,scale=0.7]{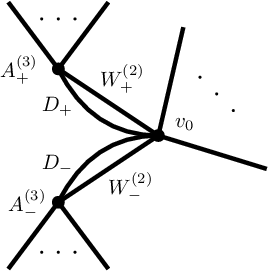}    
\end{equation}
As before, since the edges corresponding to $D_+$ (resp.~$D_-$) and $\twoCell{W}_+$ (resp.~$\twoCell{W}_-$) form a cycle, we can see that there exists a spanning tree $\spnTree'$ of $\ourGraph_{\de'}(\cut_D(M))$ that does not contain the edges corresponding to $D_+$ and $D_-$. 

If $\cut_D(M)$ is separated into 2 connected components $M_+$ and $M_-$, then $\de'$ restricts to PLCW decompositions on each of them. Let $\ourGraph_{\de'}(M_+)$ be the (2,3)-graph of $M_+$ associated to $\de'$ (or more precisely, the restriction of $\de'$ to $M_+$), similarly let $\ourGraph_{\de'}(M_-)$ be the (2,3)-graph of $M_-$. Then near the cut $\ourGraph_{\de'}(M_+)$ and $\ourGraph_{\de'}(M_-)$ are of the form 
\begin{equation}\label{eq:tau-cut-6}
\includegraphics[valign=c,scale=0.7]{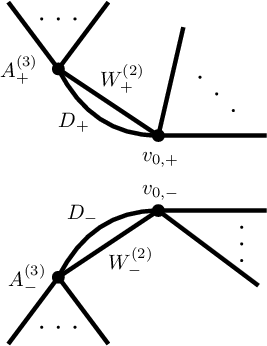}    
\end{equation}
where $v_{0, +}$ and $v_{0, -}$ are the corresponding distinguished vertices. 
The same argument as in the above implies that 
there exist spanning trees $\spnTree'_+$ and $\spnTree'_-$ of $\ourGraph_{\de'}(M_+)$ and $\ourGraph_{\de'}(M_-)$, respectively, such that $\spnTree'_+$ does not contain $D_+$ and $\spnTree'_-$ does not contain $D_-$.
In this case, denote $\cut_D(\Gamma)_+ := \cut_D(\Gamma) \cap M_+$ and $\cut_D(\Gamma)_- := \cut_D(\Gamma) \cap M_-$
    so that
$(\cut_D(M), \cut_D(\Gamma)) = (M_+, \cut_D(M)_+) \sqcup (M_-, \cut_D(M)_-)$.

Pick arbitrarily a set $Q'$ of interior points of boundary 2-cells of $\cut_D(M)$. If $\cut_D(M)$ is connected, denote 
\begin{equation*}
(H', \Phi') := \fH_{\de'}(\cut_D(M), \cut_D(\Gamma), Q', \spnTree') \ ,
\end{equation*}
while for $\cut_D(M)$ not connected set
\begin{equation*}
(H', \Phi') := \fH_{\de'}(M_+, \cut_D(\Gamma)_+ , Q'_+, \spnTree'_+) \sqcup \fH_{\de'}(M_-, \cut_D(\Gamma)_-, Q'_-, \spnTree'_-) \ ,
\end{equation*}
where $Q'_+ = Q' \cap M_+$ and $Q'_- = Q' \cap M_-$. 
Then by definition, in both cases, we have
\begin{equation}
\tau_{\cC}(\cut_D(M), \cut_D(\Gamma)) = \cgpt^\mathrm{bi}(H', \Phi')\,.
\end{equation}

By the above discussions, the local picture of $(H', \Phi')$ near the cut, regardless of the connectedness of $\cut_D(M)$, is of the form 
\begin{equation}\label{eq:H-prime}
\sum_{i}\quad 
\includegraphics[valign=c,scale=0.7]{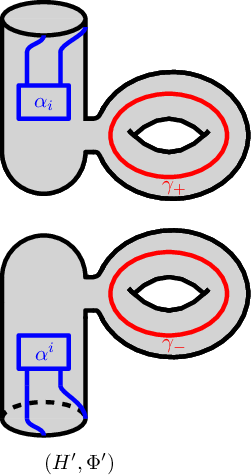}    
\end{equation}
where the red loops $\gamma_+$ and $\gamma_-$, corresponding to $D_+$ and $D_-$ respectively, are present in $(H', \Phi')$ because of our choice of spanning trees. Note also that outside of the neighbourhood of the cut depicted in \eqref{eq:H-prime}, $(H', \Phi')$ is identical to $(H, \Phi)$ (see \eqref{eq:H-Phi}). 

Applying the red capping move along $\gamma_+$ and $\gamma_-$ to $(H', \Phi')$, we obtain a new multi-handlebody $(H'', \Phi'')$ with bichrome graphs on it, i.e., 
\begin{equation}\label{eq:H-dbl-prime-cap}
(H'', \Phi'') = \Cap_{\gamma_+, \gamma_-}(H', \Phi')\,.    
\end{equation}
The local configuration of $(H'', \Phi'')$ near the cut is illustrated below:
\begin{equation}\label{eq:after-cut}
\sum_{i}\quad
\includegraphics[valign=c,scale=0.6]{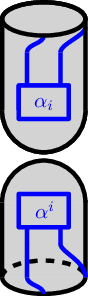}   
\end{equation}
By construction, outside of the neighbourhood of the cut depicted in \eqref{eq:after-cut}, $(H'', \Phi'')$ is identical to $(H', \Phi')$ and $(H, \Phi)$. Moreover, by comparing \eqref{eq:after-cut} and \eqref{eq:H-Phi}, we can see that
\begin{equation}\label{eq:H-dbl-prime-cut}
(H'', \Phi'') = (\cut_K(H), \cut_K(\Phi))\,.
\end{equation}
Therefore, we have 
\begin{equation}
\tau_{\cC}(M, \Gamma) = \cgpt^\mathrm{bi}(H, \Phi) = \cgpt^\mathrm{bi}(H'', \Phi'') = \cgpt^\mathrm{bi}(H', \Phi') = \tau_{\cC}(\cut_D(M), \cut_D(\Gamma)) 
\end{equation}
where the second equality follows from \eqref{eq:H-dbl-prime-cut} and Corollary~\ref{cor:F-H-bi-properties}, and the third equality follows from \eqref{eq:H-dbl-prime-cap} and Lemma \ref{lem:hb-cap}.
\end{proof}

\section{Open-closed TFT via the universal construction}\label{sec:ocTFT}

\subsection{Bulk-admissible bordisms}

Before defining the relevant bordism categories, we briefly set up the notations and recall some basic concepts on bordisms. For the rest of the paper, we assume that all manifolds are compact and oriented, and that for all $n \in \bN$, any embedding of an $(n-1)$-manifold into an $n$-manifold is equipped with a collar, which we usually suppress for simplicity. 

\begin{definition}\label{def:cob-sigma}
Suppose $\Sigma$, $\Sigma'$ are compact oriented surface with boundaries $O = \bd\Sigma$ and $O' = \bd\Sigma'$. 
A \emph{bordism} from $\Sigma$ to $\Sigma'$ is a compact oriented 3-manifold $M$ with corners.
The boundary $\bd M$ is equipped with a decomposition into two parts, the \textit{free boundary} $\bdf M$ and the \textit{gluing boundary} $\bdg M$ such that 
\begin{itemize}
\item $\bd M = \bdf M \cup \bdg M$, and
\item $\bdg M \cap \bdf M$ is precisely the set of corner points of $M$.
\end{itemize}
Furthermore, $M$ is equipped with an orientation preserving diffeomorphism $-\Sigma \sqcup \Sigma' \to \bdg M$ (the \textit{gluing boundary parametrisation}) such that $O \sqcup O'$ is mapped onto $\bdg M \cap \bdf M$.
We call $\Sigma$ and $\Sigma'$ the source and target of $M$, and we write $M: \Sigma \to \Sigma'$.
\end{definition}

Note that the above decomposition of $\bd M$ endows $M$ with the structure of a 3-dimensional $\langle 2 \rangle$-manifold (see for example \cite{LP08, HR24}). 

Let $(\cC, \mtr, \pi_\1)$ be as in \eqref{eq:oc-invariant-alg-input}, and let $(\Sigma, L)$ and $(\Sigma', L')$ be $\bd$-marked surfaces (Definition~\ref{def:d-marked-surface}) over $\cC$. An orientation-preserving diffeomorphism from $(\Sigma, L)$ to $(\Sigma', L')$ is an orientation-preserving diffeomorphism $f: \Sigma \to \Sigma'$ such that $f(\bd\Sigma)=\bd\Sigma'$, $f(L)=L'$, and for any marked point $x \in L$, its label equals to the label of $f(x) \in L'$. In the following, all diffeomorphisms are assumed to be orientation-preserving, and by abuse of notation, we write $f: (\Sigma, L) \to (\Sigma', L')$. 

\begin{definition}\label{def:adm-cob}
Let $(\cC, \mtr, \pi_\1)$ be as in \eqref{eq:oc-invariant-alg-input}, and let $(\Sigma, L)$, $(\Sigma',L')$ be $\bd$-marked surfaces over $\cC$. 
\begin{enumerate}
\item 
A bordism of $\bd$-marked surfaces from $(\Sigma, L)$ to $(\Sigma',L')$ is a pair $(M,\Gamma)$, where $M:\Sigma \to \Sigma'$ is a bordism of compact oriented surfaces, and $\Gamma$ is a $\cC$-graph
in the $\bd$-marked surface $(\bdf M, -L \sqcup L')$. 
We write $(M,\Gamma): (\Sigma, L) \to (\Sigma',L')$ and call $(\Sigma, L)$ and $(\Sigma',L')$ the source and target of $(M,\Gamma)$ respectively. 

\item
A bordism $(M,\Gamma): (\Sigma, L) \to (\Sigma',L')$ of $\bd$-marked surfaces is called \emph{bulk-admissible} if it satisfies the following condition: 
For each connected component $J$ of $M$ that is disjoint from the input surface $\Sigma$, there is at least one blue edge of $\Gamma$ that is contained in $\bd{J}$ and is labelled by an object in $\Proj_\cC$.
In this case, we will also call $\Gamma$ a bulk-admissible $\cC$-graph in the bordism $(M,\Gamma)$. 
\end{enumerate}
\end{definition}

\begin{remark}\label{rmk:adm}
The bulk-admissibility condition in Definition \ref{def:adm-cob} can be viewed as a generalization of that in Definition \ref{def:adm-1}. Let $(M,\Gamma): \emptyset \to \emptyset$ be a bordism from the empty set to itself. Such a bordism is also called a closed bordism. In particular, we have $\bdg M = \emptyset$, $\bdf M = \bd M$, and $\Gamma\subset\bd M$ is a graph with no free ends. In this case, $(M,\Gamma)$ is a bulk-admissible bordism in the sense of Definition \ref{def:adm-cob} if and only if $\Gamma$ is a bulk-admissible $\cC$-graph in $M$ in the sense of Definition \ref{def:adm-1}. 

Note also that if we glue two bulk-admissible bordisms of $\bd$-marked surfaces along their common gluing boundaries, then the resulting bordism is again bulk-admissible.
\end{remark}

Let $(M,\Gamma)$ and $(M',\Gamma'):(\Sigma, L) \to (\Sigma', L')$ be two bulk-admissible bordisms of $\bd$-marked surfaces. An orientation preserving diffeomorphism $\varphi: M \to M'$ is called an isomorphism from $(M,\Gamma)$ to $(M',\Gamma')$ if $\varphi$ preserves the embeddings of $-\Sigma \cup\Sigma'$ into $M$ and $M'$ respectively, and $\vphi(\Gamma) = \Gamma'$. We call the isomorphism class of $(M,\Gamma): (\Sigma, L) \to (\Sigma', L')$ the bordism class of $(M,\Gamma)$, and we denote it by $[M,\Gamma]: (\Sigma, L) \to (\Sigma', L')$ or simply $[M,\Gamma]$.

Note that $\bd$-marked surfaces, with bulk-admissible bordism classes between them as morphisms, form a category: It is routine to show that the composition of bordism classes by gluing representatives of bordism classes along their common gluing boundaries is well-defined and associative. The identity morphism of $(\Sigma, L)$ is simply $(\Sigma \times [0,1], L \times [0,1])$, where the lines $L\times [0,1]$ are coloured in blue and labelled by the same object as their corresponding end points. Moreover, this category is symmetric monoidal under disjoint union. 

\begin{definition}\label{def:Bord-C}
Let $(\cC, \mtr, \pi_\1)$ be as in \eqref{eq:oc-invariant-alg-input}. The category of 3-dimensional non-compact open-closed bordisms over $\cC$, denoted by $\ourBord(\cC)$, is the symmetric monoidal category whose objects are $\bd$-marked surfaces over $\cC$, and whose morphisms are bulk-admissible bordism classes between $\bd$-marked surfaces over $\cC$. A symmetric monoidal functor from $\ourBord(\cC)$ to a symmetric monoidal category is called a non-compact open-closed topological field theory. 
\end{definition}

\begin{exmp}\label{ex:cylinder}
Let $(\Sigma, L)$ and $(\Sigma', L')$ be $\bd$-marked surfaces over $\cC$, and $f: (\Sigma, L) \to (\Sigma', L')$ be an orientation-preserving diffeomorphism. We define the \emph{mapping cylinder} of $f$ to be the bordism $\bM(f):=(M_f, L \times [0,1])$, where $M_f = (\Sigma, L) \sqcup ((\Sigma', L') \times [0, 1]) /(x \sim (f(x), 0))$ is the cylinder over $\Sigma'$ glued with $\Sigma$ along $f$, which is viewed as a bordism from $\Sigma$ to $\Sigma'$. More specifically, the boundary identifications of $M_f$ are $(f, 0): \Sigma \to \Sigma' \times \{0\}$ and $(\id_{\Sigma'}, 1): \Sigma' \to \Sigma'\times\{1\}$. By definition, $\bM(f)$ is bulk-admissible, and it is easy to check that the bordism class of $\bM(f)$ only depends on the mapping class of $f$. In other words, $[f]=[f'] \in \MCG(\Sigma, L)$ implies $[\bM(f)] = [\bM(f')] \in \End(\Sigma, L)$ in $\ourBord(\cC)$, and we have a well-defined assignment $\MCG(\Sigma, L) \to \End(\Sigma, L)$, $[f] \mapsto [\bM(f)]$.
\end{exmp}

\subsection{Universal construction}

Let $(\cC, \mtr, \pi_\1)$ be as in \eqref{eq:oc-invariant-alg-input}.
In view of Remark \ref{rmk:adm}, the invariant $\tau_\cC$ is a multiplicative invariant of bulk-admissible closed bordism classes of $\bd$-marked surfaces. We will show that it can be extended to a topological field theory $\ourV: \ourBord(\cC) \to \Vect_{\kk}$ via the universal construction (see \cite{BHMV2} and e.g.\ \cite[Ch.\,7.7]{Geer-Patureau-book}).

For any $(\Sigma, L) \in \ourBord(\cC)$, consider the following vector spaces with basis given by Hom-spaces in $\ourBord(\cC)$,
\begin{equation}
\begin{split}
    \tbv(\Sigma, L) &:= \span_{\kk}\{[M,\Gamma] \in \Hom(\emptyset, (\Sigma, L))\} \ ,
    \\
\tbv'(\Sigma,L) &:= \span_{\kk}\{[M',\Gamma'] \in \Hom((\Sigma, L), \emptyset)\}\,.
\end{split}
\end{equation}
Then $\tau_{\cC}$ extends by linearity to a natural pairing between $\tbv'(\Sigma, L)$ and $\tbv(\Sigma,L)$:
\begin{equation}
\pair{\cdot, \cdot}: \tbv'(\Sigma, L) \times \tbv(\Sigma,L) \to \kk\,, \quad \pair{[M',\Gamma'], [M, \Gamma]} := \tau_{\cC}(M'\cup_{(\Sigma, L)} M, \Gamma' \cup \Gamma)\,.
\end{equation}
When there is more than one $\bd$-marked surface involved, we write $\pair{\cdot, \cdot}_{(\Sigma, L)}$ to be more precise.
Define
\begin{equation}\label{eq:V-on-obj}
\ourV(\Sigma,L) := \tbv(\Sigma, L)/\rad(\pair{\cdot, \cdot})\,,
\end{equation}
where $\rad(\pair{\cdot, \cdot}) = \{w \in \tbv(\Sigma, L)\mid \pair{v,w}=0\ \forall v \in \tbv'(\Sigma,L)\}$ is the right radical of the pairing $\pair{\cdot, \cdot}$. 

For any bordism class $[M,\Gamma]: (\Sigma, L)\to (\Sigma',L')$, the associativity of the gluing of bordisms implies that the linearization of the assignment 
\begin{equation}\label{eq:V-on-bordism}
\begin{split}
\Hom(\emptyset, (\Sigma,L)) &\to \Hom(\emptyset, (\Sigma',L'))\,,\\
[N,\Phi] &\mapsto [M,\Gamma]\circ[N,\Phi] = [(M \cup_{(\Sigma,L)} N, \Gamma \cup \Phi)]
\end{split}\end{equation}
induces a well-defined linear map $\ourV[M,\Gamma]: \ourV(\Sigma,L) \to \ourV(\Sigma',L')$. For simplicity, we will use $[M,\Gamma]: \emptyset \to (\Sigma, L)$ to denote both a bordism class in $\tbv(\Sigma, L)$ and its image in $\ourV(\Sigma, L)$. 

As usual in the universal construction, at this point we already have a functor from bordisms to (possibly infinite dimensional) vector spaces,
\begin{equation}\label{eq:V-is-functor}
    \ourV: \ourBord(\cC) \to \Vect_{\kk} \ ,
\end{equation}
which takes an object $(\Sigma,L)$ to $\ourV(\Sigma,L)$ as in \eqref{eq:V-on-obj}, and which acts on bordisms as \eqref{eq:V-on-bordism}. The more involved part is to show monoidality. The first step is a lax monoidal structure which we define next. The main step will be showing that this is actually a monoidal structure.

\medskip

Let $(\Sigma_1, L_1)$, $(\Sigma_2, L_2)$ be two $\bd$-marked surfaces. Define the linear map 
\begin{equation}
    \begin{split} 
\tilde\Psi : \tbv(\Sigma_1, L_1) \otimes_{\kk} \tbv(\Sigma_2, L_2) &\to \ourV(\Sigma_1 \sqcup \Sigma_2, L_1 \sqcup L_2)\\ 
v_1 \ot v_2 &\mapsto [v_1 \sqcup v_2]
\end{split}
\end{equation}
induced by taking disjoint unions. While $\tilde\Psi$ is not injective, it turns out to become injective once we descend to the quotient also in the domain vector space. The proof is standard for universal constructions (see e.g.\ \cite[Sec.\,7.7]{Geer-Patureau-book}), and we write it out for convenience of the reader.

\begin{lemma}\label{lem:disjoint-union}
$\tilde\Psi$ descends to a well-defined injective linear map 
\begin{equation}\label{eq:disjoint-union}
\Psi : \ourV(\Sigma_1, L_1) \ot_{\kk} \ourV(\Sigma_2, L_2) \to \ourV(\Sigma_{1} \sqcup \Sigma_2, L_1 \sqcup L_2)\,.
\end{equation}
\end{lemma}
\begin{proof}
For simplicity, we denote $S_1 := (\Sigma_1, L_1)$, $S_2 := (\Sigma_2, L_2)$, and $S_{1 \sqcup 2} := (\Sigma_1 \sqcup \Sigma_2, L_1 \sqcup L_2)$. For the radicals we write accordingly $R_x = \rad(\pair{\cdot,\cdot}_{S_x})$ for $x \in \{ 1,2, 1 \sqcup 2 \}$. 

As a first step, we express the tensor product of quotients as a quotient of the tensor product,
\[
\ourV(S_1) \ot_{\kk} \ourV(S_2)
=
\big( \tbv(S_1) \otimes_{\kk} \tbv(S_2) \big) / \big( 
\tbv(S_1) \otimes_{\kk} R_2
+
R_1 \otimes_{\kk} \tbv(S_2)
\big) \ .
\]
To check well-definedness, we show that $\tilde\Psi$ vanishes on the two summands of the quotient. Let thus $v \in \tbv(S_1)$ and $w \in R_2$. Then for any $u \in \tbv'(S_{1\sqcup 2})$, 
\[
\begin{split}
\pair{u,\tilde\Psi(v \ot_\kk w)}_{S_{1 \sqcup 2}} 
&= 
\pair{u,v\sqcup w}_{S_{1 \sqcup 2}} 
= 
\tau_{\cC}(u\cup_{S_{1\sqcup 2}} (v \sqcup w)) 
= 
\tau_{\cC}( (u \cup_{S_1} v) \cup_{S_2} w )\\ 
&= \pair{(u \cup_{S_1} v) , w}_{S_2} 
= 
0 \ ,
\end{split}
\]
where the last step uses that $w \in R_2$. Hence $\tilde\Psi(v \ot_\kk w) \in R_{1 \sqcup 2}$ and so is zero in $\ourV(S_{1 \sqcup 2})$. The calculation that $\tilde\Psi(v \ot_\kk w)=0$ for $v \in R_1$ and $w \in \tbv(S_2)$ is analogous. Hence $\Psi$ is well-defined. 

Next we show that $\Psi$ is injective. Let $h = \sum_i  v_i \ot_\kk w_i \in \tbv(S_1) \ot_{\kk} \tbv(S_2)$ be such that $h \notin \tbv(S_1) \otimes_{\kk} R_2
+
R_1 \otimes_{\kk} \tbv(S_2)$, i.e., $[h]$ is non-zero in $\ourV(S_1) \ot_{\kk} \ourV(S_2)$. We will show that $\tilde\Psi(h) \notin R_{1 \sqcup 2}$, which implies that $\Psi$ is injective. By assumption on $h$, there must be $u_x \in \tbv(S_x)$, $x=1,2$, such that
\[
    \sum_i \pair{u_1,v_i}_{S_1}\pair{u_2,w_i}_{S_2} \neq 0 \ .
\]
But then also 
\[
\pair{u_1 \sqcup u_2,\tilde\Psi(h)}_{S_{1 \sqcup 2}} 
= 
\sum_i 
\pair{u_1 \sqcup u_2,v_i \sqcup w_i}_{S_{1 \sqcup 2}} 
= 
\sum_i 
\pair{u_1,v_i}_{S_{1}} 
\pair{u_2,w_i}_{S_{2}} 
\neq 0 \ ,
\]
and so $\tilde\Psi(h) \notin R_{1 \sqcup 2}$.
\end{proof}

For monoidality of $\ourV$, it remains to show surjectivity of $\Psi$ in \eqref{eq:disjoint-union}. To do so, we investigate a special family of elements of $\ourV(\Sigma, L)$ called collar elements.

\begin{definition}\label{def:collar}
Let $(\Sigma, L)$ be a $\bd$-marked surface. A \emph{collar element} of $\ourV(\Sigma, L)$ is a finite linear combination of elements of the form $[\bbSigma, \Gamma]$ such that $\bbSigma \cong \Sigma\times [0,1]$ and $\Gamma$ is an admissible $\cC$-graph. Here, we view $\bbSigma$ as a bordism from $\emptyset$ to $(\Sigma, L)$ with gluing boundary $\bdg\bbSigma = \Sigma\times\{1\}$ and free boundary $\bdf\bbSigma = (\Sigma\times\{0\}) \cup ((\bd\Sigma)\times[0,1])  \cong \Sigma$, and we smoothen the corner at $\bd\Sigma\times\{0\}$. The admissible $\cC$-graph $\Gamma$ is contained in the free boundary $\bdf\bbSigma$ and ends in $L$.
\end{definition}

\begin{prop}\label{prop:collar}
For any $\bd$-marked surface $(\Sigma, L)$, $\ourV(\Sigma, L)$ is spanned by collar elements.
\end{prop}

The proof of Proposition \ref{prop:collar} will be given in Section \ref{sec:collar-pf}.

\medskip

Let $(\Sigma, L)$ be a $\bd$-marked surface and $[\bbSigma, \Gamma]$ a collar element of $\ourV(\Sigma, L)$. Recall from Section \ref{sec:sk-mod} that $\tilde{\sk}_{\adm}(\Sigma,L)$ stands for the $\kk$-linear span of all admissible blue graphs, and so by definition, $\Gamma \in \tilde{\sk}_{\adm}(\Sigma, L)$. Moreover, any admissible $\cC$-graph $\Gamma$ in $(\Sigma, L)$ gives rise to a collar element $[\bbSigma, \Gamma] \in \ourV(\Sigma, L)$ by identifying the free boundary of $\bbSigma$ with $\Sigma$ and embedding $\Gamma$ in it. By Proposition \ref{prop:collar}, we have a surjective linear map 
\begin{equation}\label{eq:T-tilde}
\tilde{E}(\Sigma, L): \tilde{\sk}_{\adm}(\Sigma,L) \to \ourV(\Sigma, L)\,.
\end{equation}
When $(\Sigma, L)$ is clear from the context, we will also write $\tilde{E}$ for $\tilde{E}(\Sigma, L)$.

Recall from Definition \ref{def:sk-adm} that the admissible skein module $\sk_{\adm}(\Sigma, L)$ of $(\Sigma, L)$ is the quotient of $\tilde{\sk}_{\adm}(\Sigma, L)$ by the subspace $\cN_{\adm}(\Sigma, L)$ spanned by admissible skein relations.

\begin{prop}\label{prop:our-V}
Let $(\cC, \mtr, \pi_\1)$ be as in \eqref{eq:oc-invariant-alg-input}, and let $(\Sigma, L)$ be a $\bd$-marked surface.
Then $\tilde{E}(\Sigma, L)$ factors through a surjective map $E(\Sigma, L): \sk_{\adm}(\Sigma, L) \to \ourV(\Sigma, L)$. In particular, $\ourV(\Sigma, L)$ is finite-dimensional. 
\end{prop}
\begin{proof}
The map $\tilde{E}(\Sigma, L)$ is already surjective. 
For well-definedness of $E(\Sigma,L)$, we need to show that $\cN_{\adm}(\Sigma, L) \subset \ker(\tilde{E}(\Sigma, L))$.
Let thus $v \in \cN_{\adm}(\Sigma, L)$ be an admissible skein relation and write $v = \sum_i c_i \Gamma_i$ as a sum of graphs on $(\Sigma,L)$. Then for all $u \in \tbv'(\Sigma,L)$ we have
\[
\pair{u,\tilde E(\Sigma,L)(v)}
=
\sum_i \tau_\cC( u \cup_\Sigma (\bbSigma,\Gamma_i) ) 
= 0
\]
as $\tau_\cC$ respects skein relations by Proposition~\ref{prop:skein} (an admissible skein relation is in particular 1-admissible). Hence $\tilde E(\Sigma,L)(u)$ is in the radical of the pairing and therefore zero in $\ourV(\Sigma,L)$.

Finally, in light of Proposition \ref{prop:finite-1}, existence of the surjective map $E(\Sigma,L)$ immediately implies that $\ourV(\Sigma, L)$ is finite-dimensional. 
\end{proof}

We can now state the main result of this paper:

\begin{thm}\label{thm:ocTFT}
Let $(\cC, \mtr, \pi_\1)$ be as in \eqref{eq:oc-invariant-alg-input}. The assignments $(\Sigma, L) \mapsto \ourV(\Sigma, L)$ and $[M, \Gamma] \mapsto \ourV[M, \Gamma]$ of objects and morphisms of $\ourBord(\cC)$ to finite-dimensional vector spaces and linear maps defines an open-closed topological field theory, i.e., a symmetric monoidal functor $\ourV: \ourBord(\cC) \to \Vect_{\kk}$. 
\end{thm}

\begin{proof}
We already noted functoriality of $\ourV$ in \eqref{eq:V-is-functor}. That $\ourV$ takes values in finite-dimensional vector spaces was shown in Proposition~\ref{prop:our-V}. 
    
In Lemma~\ref{lem:disjoint-union} we saw that the candidate monoidal structure map $\Psi$ is injective. Next we show that it is surjective. 
By definition, the collar elements in $\ourV(\Sigma\sqcup\Sigma, L \sqcup L')$ are of the form $(\bbSigma\sqcup\bbSigma', \Gamma\sqcup\Gamma')$ where $\bbSigma \cong \Sigma \times [0, 1]$ (resp.\ $\bbSigma' \cong \Sigma'\times [0, 1]$), and $\Gamma$ (resp.\ $\Gamma'$) is a $\cC$-coloured graph contained in the $\bdf\bbSigma$ (resp.\ $\bdf{\bbSigma'}$) ending in $L$ (resp.\ $L'$). In other words, all of the collar elements in $\ourV(\Sigma\sqcup\Sigma', L\sqcup L')$ are contained in the image of $\Psi$. Therefore, by Proposition \ref{prop:collar}, $\Psi$ is surjective.

Finally, the fact that $\Psi$ endows the functor $\ourV$ with a symmetric monoidal structure follows directly from definition, and we are done.
\end{proof}

Proposition~\ref{prop:our-V} gives a concrete description of the state space of the open-closed TFT $\ourV$ via skeins, while in general in the universal construction it is not clear how to parametrise states concretely as one needs to use all bordisms. 

Note that for TFTs starting from the full bordism category, the values of the TFT functor are automatically finite dimensional vector spaces, as the source bordism category is rigid. However, the non-compact bordism category $\ourBord(\cC)$ is not rigid, and so the fact that $\ourV$ is a symmetric monoidal functor does not a priori imply that $\ourV(\Sigma, L)$ is finite-dimensional. We get finite-dimensionality from that of skein modules via Proposition~\ref{prop:our-V}.

\begin{exmp}
As a simple example, consider the solid torus $H$ as a bordism $\emptyset \to \emptyset$ with embedded blue graph $\Gamma_n$ on its boundary given a circle labelled by a projective object $P$ as in \eqref{eq:solid-torus-blue-non-contractible}. There, we also computed the invariant to be
\begin{equation}
\ourV(H, \Gamma_n)
=
\tau_{\cC}(H, \Gamma_n)
=
\cgpt(H, [\Gamma_n])
=
\dim_{\kk}\cC(P, \1) \ ,
\end{equation}
where in the second step we used Proposition~\ref{prop:hdby}. In this case, the disc $D$ with a $P$-marked point on the boundary is actually a rigid object in $\ourBord(\cC)$, and so the invariant above can be computed as a trace of the the identity morphism,
\begin{equation}
\ourV(H, \Gamma_n)
= 
\dim_{\kk} \ourV(D,P)
\end{equation}
By Proposition~\ref{prop:our-V} we have a surjective map $E(D,P): \sk_{\adm}(D,P) \to \ourV(D,P)$, and we also have a surjective map
\begin{equation}
\cC(P, \1) \to \sk_{\adm}(D,P)
~~,\quad f \mapsto 
\includegraphics[valign=c,scale=0.6]{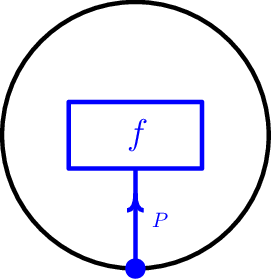}
\end{equation}
Comparing dimensions, we see that all these surjective maps are isomorphisms,
\begin{equation}
\cC(P, \1) \cong \sk_{\adm}(D,P) \cong \ourV(D,P) \ .
\end{equation}
This gives at least one concrete example of a state space of the TFT $\mathcal{V}$, and also illustrates that the theory is non-trivial.
\end{exmp}

\begin{remark}\label{rmk:compare}
In \cite{CGPV23} a non-compact TFT $\EuScript{S}$ is defined on objects and bordisms without free boundary, i.e., on closed bordisms. To a closed surface $\Sigma$, it assigns the admissible skein module, $\mathscr{S}(\Sigma) = \sk_{\adm}(\Sigma,\emptyset)$. The value on bordisms is obtained via the presentation in \cite{Juh18}. Given a closed connected 3-manifold $M$, consider $M \setminus \bB^3$ as a bordism $\bS^2 \to \emptyset$. 
It is shown in \cite[Thm.\,3.5]{CGPV23} that $\mathscr{S}$ is related to $\mathcal{K}_\cC$ (cf.\ Section~\ref{sec:CGPT-compare}) as
\begin{equation}
    \mathscr{S}(M \setminus \bB^3) = \mathcal{K}_\cC(M) \, \cgpt' ~:~ \mathscr{S}(\bS^2) = \sk_{\adm}(\bS^2,\emptyset) \to \kk \ ,
\end{equation}
where $\cgpt'$ was defined in \eqref{eq:sph-inv}. The same relation holds for $\ourV$,
\begin{equation}
    \ourV(M \setminus \bB^3) = \mathcal{K}_\cC(M) \, \cgpt' ~:~ \ourV(\bS^2,\emptyset) \cong \sk_{\adm}(\bS^2,\emptyset) \to \kk \ .
\end{equation}
Indeed, by Proposition~\ref{prop:our-V} we have a surjection $E(\bS^2,\emptyset) : \sk_{\adm}(\bS^2,\emptyset) \to \ourV(\bS^2,\emptyset)$. But $\sk_{\adm}(\bS^2,\emptyset)$ is one-dimensional with basis the class of the graph $I$ described in Section~\ref{sec:CGPT-compare}, and the surjection is not zero (compose with $\bD^3 : \bS^2 \to \emptyset$), hence an isomorphism. Acting with $\ourV(M \setminus \bB^3)$ on the collar element for $\bS^2$ with graph $I$ produces the invariant $\mathscr{I}_{\cC}(M)$ from \eqref{eq:inv-of-M}, which we saw to be equal to $\mathcal{K}_\cC(M)$ in Proposition~\ref{prop:compare-closed}.

More generally, we expect that our TFT $\ourV$ defined in Theorem \ref{thm:ocTFT} agrees with $\mathscr{S}$ on closed surfaces and bordisms. We have not been able to prove this, so this is based on the observation above as well as the surjection $\mathscr{S}(\Sigma) \to \ourV(\Sigma,\emptyset)$ for closed surfaces from Proposition \ref{prop:our-V} (which accordingly we expect to be an isomorphism). 
\end{remark}

\begin{remark}\label{rmk:atfd-2}
As is noted in Remark \ref{rmk:atfd-1}, when $\cC$ is semisimple, the alterfold invariant $Z_{\cC}$ is a scalar multiple of $\tau_{\cC}$. Since both the alterfold TFT \cite{atfd2} and the TFT $\ourV$ in this paper are constructed via the universal construction, and a non-zero constant scalar multiple does not change the radical of the pairing associated to the corresponding invariants, the two open-closed TFTs have isomorphic state spaces and are proportional to each other on bordisms.
More precisely, consider an arbitrary $\cC$-decorated 2-alterfold $(F, \gamma, L)$, with $F$ a closed surface, $\gamma$ a multi-curve on $F$ separating it into alternatively $A$-$B$-coloured regions, and $L$ are $\cC$-labelled points on $\gamma$. Then the state space $\mathbb{V}(F, \gamma, L)$ in the alterfold TFT is isomorphic to $\ourV(\Sigma, L)$ where $\Sigma$ is the $B$-coloured subsurface of $F$ with boundary $\bd\Sigma = \gamma$. We skip the details.
\end{remark}

\subsection{Proof of Proposition \ref{prop:collar}}\label{sec:collar-pf}
Recall that every triangulation is also a PLCW decomposition, and in fact, by \cite[Thm.\,6.3]{Kir12}, every PLCW complex has a subdivision which is a triangulation. 
We will need:

\begin{prop}[{\cite[Cor.\,1]{Arm67}}]\label{prop:ext}
Any triangulation of the boundary of a compact PL-manifold can be extended to a triangulation of the whole manifold.     
\end{prop}

We immediately have:
\begin{cor}\label{cor:ext}
Let $\Sigma,\Sigma'$ be compact oriented surfaces with boundary 
(possibly empty), and let $M: \Sigma \to \Sigma'$ be a bordism. Then any triangulation $\de$ of $\bdg M = -\Sigma \cup \Sigma'$ can be extended to a triangulation of $M$.
\end{cor}
\begin{proof}
Let $O := \bdf M \cap \bdg M$, then $O = \bd(\bdf M)$, and $\de$ restricts to a triangulation $\de|_O$ of $O$. So by Proposition \ref{prop:ext}, $\de|_O$ extends to a triangulation $\vphi$ on $\bdf M$. Moreover, $\de \cup_O \vphi$ is a triangulation of $\bd M = \bdf M \cup \bdg M$. Finally, applying Proposition \ref{prop:ext} again, we obtain a triangulation of $M$ extending $\de$.
\end{proof}

Now we are ready to give the proof of Proposition \ref{prop:collar}.

\begin{proof}[Proof of Proposition \ref{prop:collar}]
It suffices to show that for any $[M, \Gamma] \in \Hom(\emptyset, (\Sigma, L))$, there is a collar element $u \in \tbv(\Sigma, L)$ such that $[M, \Gamma]-u \in \rad(\pair{\cdot, \cdot})$. By Lemma \ref{lem:disjoint-union}, we may assume that $M$ is connected (while $\Sigma$ is not necessarily connected). 

To construct the collar element $u$, we first set up some notations. By the existence of collars, $\bbSigma \cong \Sigma \times [0,1]$ embeds in $M$, so that $\Sigma_0 := \Sigma\times\{0\}$ is attached to $M \setminus \bbSigma$. The gluing boundary $\bdg M$ of $M$ is $\Sigma_{1} := \Sigma\times\{1\}\subset \bbSigma\subset M$, and to emphasise this, we write $(N, \Phi) \cup_{(\Sigma_1, L)} (M, \Gamma)$ when we glue $(M, \Gamma)$ to another bordism $(N, \Phi)$.

Let $\de_M$ be a triangulation of $M$ which restricts to  triangulations of $\bbSigma$, $\Sigma_0$ and $\Sigma_1$, which are denoted by $\de_{\smbbSigma}$, $\de_{\Sigma_{0}}$ and $\de_{\Sigma_{1}}$ respectively. Define 
\[M_1 := \bbSigma \cup \nb\big((M\setminus\bbSigma)_{\de}^{(1)}\big)\,,\]
i.e., $M_1$ is the union of $\bbSigma$ with the tubular neighbourhood of the 1-skeleton of $M\setminus\bbSigma$.
Let $Q_M$ be a set of interior points of 2-cells in $(\bd M) \setminus \Sigma_{1} = \intr(\bdf M)$.

Let $\ourGraph'(M)$ be the graph obtained from $\ourGraph_{\de_M}(M)$ by deleting all of its edges corresponding to the 2-cells in $\de_M$ which are contained in $\Sigma_1$, i.e., 
\begin{equation}\label{eq:G'}
\ourGraph'(M) := \ourGraph_{\de_M}(M) \setminus (\de_{\Sigma_{1}})_{2} \ ,
\end{equation}
where $(\de_{\Sigma_{1}})_{2}$ stands for the set of 2-cells in $\de_{\Sigma_{1}}$. 
Fix a choice of a spanning tree $\spnTree'(M)$ of $\ourGraph'(M)$. Similar to what we did in Definition \ref{def:hb-of-M}, deform $\Gamma$ according to the interior points $Q_M$ chosen above, so that $\Gamma$ is contained in $\bd M_1$. Then we add red loops along the belts of 2-cells in $\de_{M}\setminus \de_{\smbbSigma}$ which correspond to edges in $\ourGraph'(M)\setminus \spnTree'(M)$. Denote the resulting bichrome graph in $M_1$ by $\Gamma_1$ (recall from Remark \ref{rmk:cycle-switch} that we denote the set of edges of $\spnTree'(M)$ also by $\spnTree'(M)$):
\begin{equation}\label{eq:gamma-1}
\Gamma_1 = \Gamma \cup \bigcup_{\twoCell{W} \in S(M\setminus\smbbSigma)} \gamma_{\belt(\twoCell{W})}\,, \quad\text{with $S(M\setminus\bbSigma) = (\de_{M\setminus \smbbSigma})_2 \setminus \spnTree'(M)$}\,,
\end{equation}
where $\gamma_{\belt(\twoCell{W})}$ is the red loop along the belt of the 2-cell $\twoCell{W}$ with respect to $Q_M$
(suppressed for simplicity), see \eqref{eq:belt}.

Applying the red-to-blue map $\operatorname{RB}$ to all the red loops in $\Gamma_1$, we obtain a bulk-admissible
graph $\rb(\Gamma_1)$ in $M_1$. Let $D(M_1\setminus\bbSigma)$ be the collection of meridian disks of all the 1-handles (thickened 1-cells) in $M_1\setminus\bbSigma$. By connectedness, we can isotope $\rb(\Gamma_1)$ so that the boundaries of the disks in $D(M_1\setminus\bbSigma)$ have non-empty intersection with edges in $\rb(\Gamma_1)$ that are coloured by projective objects. Applying the cutting move \eqref{eq:tau-cut-0} to $(M_1, \rb(\Gamma_1))$ along each of the disks in $D(M_1\setminus\bbSigma)$, we will obtain the disjoint union of $\bbSigma$ with finitely many 3-balls $\bD^3$, each containing admissible graphs in them. Evaluating the 3-balls with the admissible graphs by $\cgpt$, we obtain a linear combination of admissible graphs in $\bbSigma$, which by definition is a collar element. Denote it by $u$, then with an abuse of notation, we write 
\[u = \cut_{D(M_1\setminus\smbbSigma)}(M_1, \rb(\Gamma_1))\,.\]

We show that $u$ is the desired collar element by showing that for any $[N, \Phi] \in \tbv'(\Sigma, L)$, $\tau_{\cC}((N, \Phi)\cup_{(\Sigma_1, L)}(M, \Gamma))$ is equal to $\tau_{\cC}((N, \Phi) \cup_{(\Sigma_1, L)} u)$. Note that by construction and Proposition \ref{prop:cut}, we have 
\begin{equation}\label{eq:u-M1}
\tau_{\cC}((N, \Phi) \cup_{(\Sigma_1, L)} u) = \tau_{\cC}(N\cup_{\Sigma_1}M_1, \Phi\cup_{L}\rb(\Gamma_1))\,,
\end{equation}
so it remains to show that $\tau_{\cC}(N\cup_{\Sigma_1}M, \Phi\cup_L \Gamma)$ is equal to $\tau_{\cC}(N\cup_{\Sigma_1}M_1, \Phi\cup_{L} \rb(\Gamma_1))$.

Start with the triangulation $\de_M$ that we picked to construct $u$, and note that $\de_{M}$ restricts to a triangulation $\de_{\Sigma_1}$. Then by Corollary \ref{cor:ext}, we can extend $\de_{\Sigma_1}$ to a triangulation 
$\de_N$ on $N$. Let 
\[\tilde{\de} := \de_N \cup_{\de_{\Sigma_1}} \de_M\]
be the triangulation on $N \cup_{\Sigma} M$ that restricts to $\de_M$ on $M$ and $\de_N$ on $N$.

The closure of the 1-skeleton of $M \setminus\bbSigma$, denoted by $\clsr(M\setminus\bbSigma)^{(1)}_{\de}$, is the union of $(M\setminus\bbSigma)^{(1)}_{\de}$ with all the 0-cells in $\Sigma_0$ that are boundaries of 1-cells in $M \setminus\bbSigma$. 
It is a 1-dimensional PLCW complex. 
By the connectedness of $M$, we can iteratively contract 1-cells in $M\setminus\bbSigma$ which have one boundary 0-cell in $\Sigma_0$ and the other one in $M\setminus\bbSigma$ until all the 0-cells of $\clsr(M\setminus\bbSigma)^{(1)}_{\de}$ are contained in $\Sigma_0$.
For example:
\begin{equation}
\includegraphics[valign=c,scale=0.6]{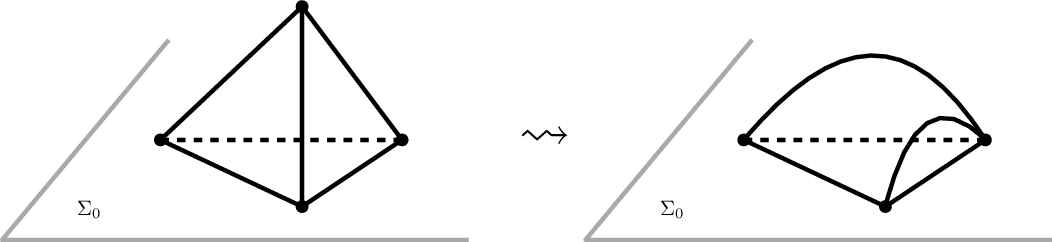}
\end{equation}
In this way, we obtain a new 1-dimensional PLCW complex $Y$ such that the set of 0-cells $Y_0\subset(\de_{\Sigma_0})_0$. Moreover, we have 
    \[\bbSigma \cup \nb((M\setminus\bbSigma)^{(1)}_\de) = M_1 \cong \bbSigma\cup\nb(Y)\,.\]
In other words, $M_1$ is homeomorphic to the 3-manifold obtained from attaching 1-handles to $\Sigma_0 \subset \bbSigma$.

We can now choose a PLCW decomposition $\vphi$ of $M_1$ as follows: For $\bbSigma$, we still choose $\de_{\smbbSigma}$. For any 1-handle in $\nb(Y^{(1)})$ attached to $\bbSigma$ along a 1-cell $\oneCell{X}$, we pick a PLCW decomposition of it according to whether $\oneCell{X}$ has one or two boundary 0-cells, which is described as follows.

Firstly, suppose $\oneCell{X}$ has two 0-cells, which are denoted by $\zeroCell{Z}_+$ and $\zeroCell{Z}_-$ respectively. Pick a 2-cell $\twoCell{V}_+$ in $\bbSigma$ containing $\zeroCell{Z}_+$, add a 1-cell loop $\oneCell{X}_+$ around $\zeroCell{Z}_+$, we obtain a new 2-cell $\twoCell{W}_+$ bounded by $\oneCell{X}_+$ inside $\twoCell{V}_+$, and the rest of $\twoCell{V}_+$ is again a 2-cell, which, by an abuse of notation, is still denoted by $\twoCell{V}_+$. Do the same for $\zeroCell{Z}_{-}$. Attach another 2-cell $\twoCell{W}$ along $\oneCell{X}_+$, $\oneCell{X}_-$ and $\oneCell{X}$, and finally add a 3-cell $\threeCell{A}$ that is bounded by $\twoCell{W}_+$, $\twoCell{W}_-$ and $\twoCell{W}$, we obtain a PLCW decomposition of the 1-handle attached to $\bbSigma$. The above construction is depicted by:
\begin{equation}\label{eq:collar-2} 
\includegraphics[valign=c,scale=0.6]{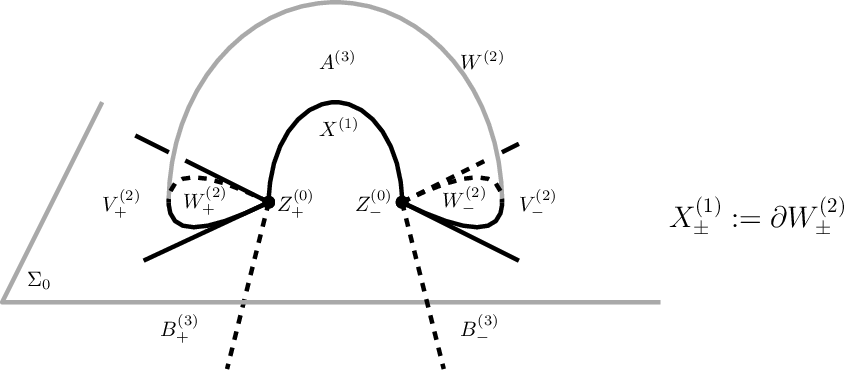}
\end{equation}
Suppose $\twoCell{V}_{\pm}$ are boundary 2-cells of 3-cells $\threeCell{B}_{\pm}$ in $\de_{\smbbSigma}$ respectively, then attaching the 1-handle with the above PLCW decomposition to $\bbSigma$ results in adding the following local graph to $\ourGraph_{\de_{\smbbSigma}}(\bbSigma)$, where the orange edges are the 2-cells in $\vphi$ but not in $\de_{\smbbSigma}$:
\begin{equation}
\includegraphics[valign=c,scale=0.7]{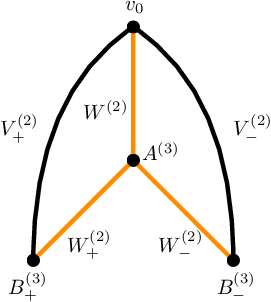}
\end{equation}
If $\oneCell{X}$ has only one 0-cell $\zeroCell{Z}$, we choose a two-cell $\twoCell{V}$ in $\bbSigma$ containing $\zeroCell{Z}$ and add two 1-cell loops around $\zeroCell{Z}$, then the rest of the construction is similar to the above:
\begin{equation}\label{eq:collar-6}
\includegraphics[valign=c,scale=0.6]{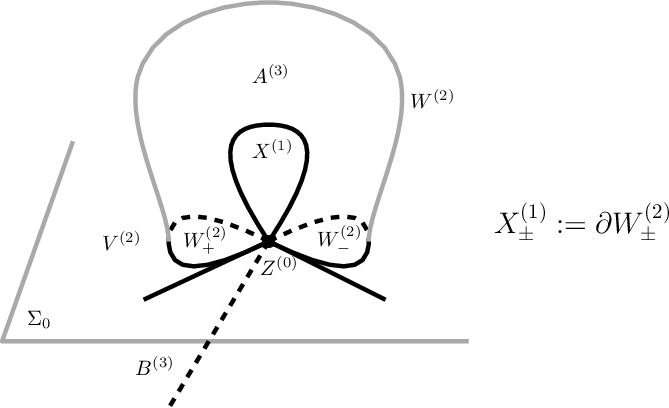}
\end{equation}
In this case, attaching the 1-handle with the above PLCW decomposition results in adding the following local graph to $\ourGraph_{\de_{\smbbSigma}}(\bbSigma)$:
\begin{equation}
\includegraphics[valign=c,scale=0.7]{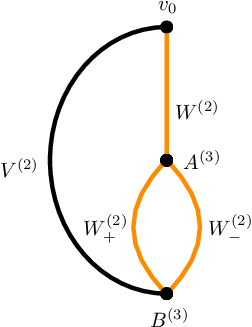}
\end{equation}
By the above construction, it is easy to see that the 1-skeleton $(M_1)_{\vphi}^{(1)}$ of $M_1$ can be obtained by attaching two 1-cells $\oneCell{X}_{\pm} (= \bd \twoCell{W}_{\pm})$ to $\oneCell{M}_{\de}$ for each 1-cell $\oneCell{X}$ in $Y$, i.e., we have 
\begin{equation}\label{eq:1cell-M1}
\oneCell{(M_{1})}_{\vphi} \cong \oneCell{M}_{\de} \cup  \bigcup_{\oneCell{X} \in Y} 
(\oneCell{X}_{+} \cup \oneCell{X}_{-})    
\end{equation}
and the set of 2-cells in $M_1$ is given by 
\begin{equation}\label{eq:2cell-M1}
\vphi_2 = (\de_{\smbbSigma})_2 \cup \bigcup_{\oneCell{X} \in Y} 
\Big(\twoCell{W}(\oneCell{X}) \cup \twoCell{W}_+(\oneCell{X}) \cup \twoCell{W}_-(\oneCell{X})\Big) \ .
\end{equation}

Let $\ourGraph'(M_1)$ be the graph (with half-edges) obtained from deleting the half-edges in $\ourGraph_{\vphi}(M_1)$ corresponding to the 2-cells in $\Sigma_1$ from the boundary vertex side. Define $\ourGraph'(\bbSigma)$ and $\ourGraph'(N)$ similarly (with respect to $\de_{\smbbSigma}$ and $\de_N$). Then by the discussions above, we have $\ourGraph'(M_1) \supset \ourGraph'(\bbSigma)\subset \ourGraph'(M)$ (see \eqref{eq:G'}), depicted as follows:
\begin{equation}\label{eq:collar-8}
\begin{tikzcd}
\includegraphics[valign=c,scale=0.55]{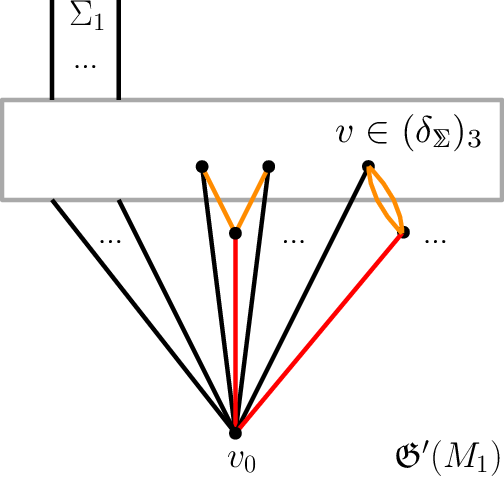} & & \includegraphics[valign=c,scale=0.55]{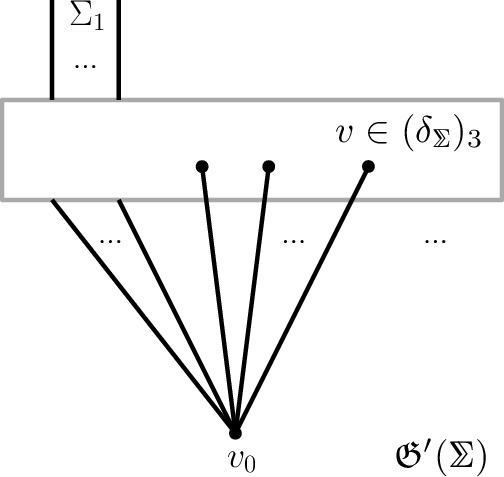}\ar[ll, hook']\ar[d, hook]\\
&&\includegraphics[valign=c,scale=0.55]{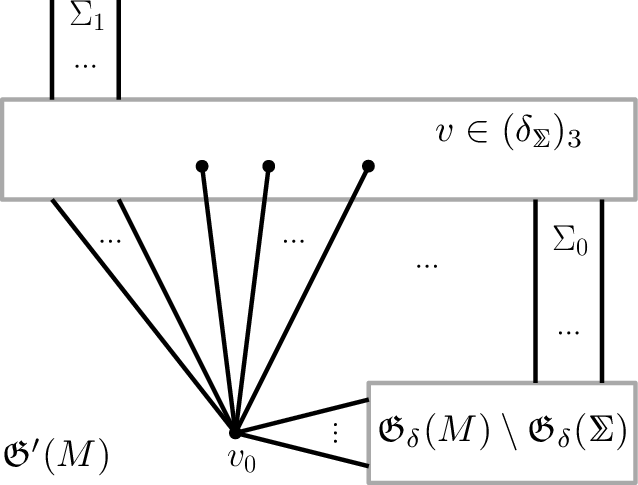}
\end{tikzcd}
\end{equation}
Here, the gray boxes labelled by $v \in (\de_{\smbbSigma})_3$ represent the part of the graphs whose vertices are contained in $(\de_{\smbbSigma})_3$, i.e., those correspond to 3-cells in $\bbSigma$.

Since $\vphi$ and $\de_N$ are identical on $\Sigma\times\{1\}$, by gluing them together along $\Sigma_1$, we obtain a PLCW decomposition 
\[\tilde{\vphi} := \vphi \cup_{\de_{\Sigma_1}} \de_N\]
of $N \cup_{\Sigma_{1}} M_1$ that restricts to $\vphi$ on $M_1$ and $\de_N$ on $N$. The corresponding (2,3)-graph $\ourGraph_{\tilde\vphi}(N\cup_{\Sigma_{1}}M_1)$ can be obtained from connecting $\ourGraph'(M_1)$ and $\ourGraph'(N)$ along the half-edges corresponding to 2-cells in $\Sigma_1$ and at the boundary vertex, and $\ourGraph_{\de}(N\cup_{\Sigma}M)$ can be obtained similarly. This is illustrated below:
\begin{equation}
\includegraphics[valign=c,scale=0.5]{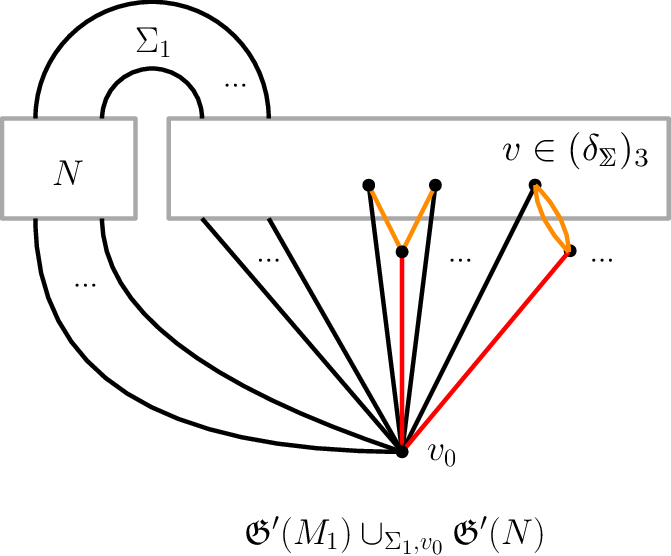}
\quad
\includegraphics[valign=c,scale=0.5]{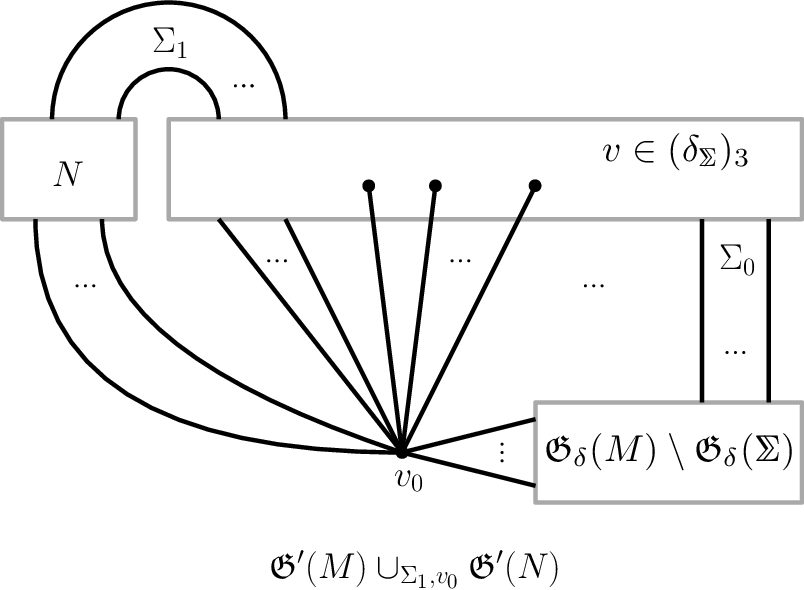}
\end{equation}

Recall that to define the bichrome graph $\Gamma_1$ in \eqref{eq:gamma-1}, we chose a spanning tree $\spnTree'(M)$ of $\ourGraph'(M)$ (ignoring half-edges). Let $\spnTree'(\bbSigma) := \spnTree'(M) \cap \ourGraph'(\bbSigma)$, then by the above construction, in particular, by \eqref{eq:2cell-M1}, the subgraph 
\begin{equation}\label{eq:T-M1}
\spnTree'(M_1) := \spnTree'(\bbSigma) \cup 
\bigcup_{\oneCell{X} \in Y} 
\twoCell{W}(\oneCell{X})
\end{equation}
of $\ourGraph'(M_1)$ is a spanning tree of $\ourGraph'(M_1)$. Indeed, the set of vertices that are not in $\spnTree'(\bbSigma)$ is in one-to-one correspondence with 1-cells in $Y$ (the 3-cells of the tubular neighbourhood of these 1-cells), and by adding 
$\bigcup_{\oneCell{X} \in Y} 
\twoCell{W}(\oneCell{X})$ to $\spnTree'(\bbSigma)$, we cover all the vertices in $\ourGraph'(M_1)$. 
Then, by Remark \ref{rmk:cycle-switch} (1), $\spnTree'(M_1)$ is a tree.
The edges that are added to $\spnTree'(\bbSigma$) are coloured in red on the left hand side of \eqref{eq:collar-8}.

By definition, $\ourGraph_{\de_N}(N)$ is of the following form: 
\begin{equation}\includegraphics[valign=c,scale=0.6]{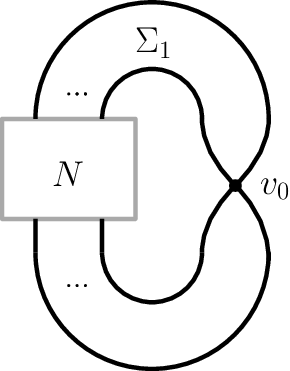}\end{equation}
Pick a spanning tree $\spnTree(N)$ of $\ourGraph_{\de_N}(N)$, then 
\[\spnTree(N\cup_{\Sigma_{1}}M) := \spnTree(N)\cup_{\bd} \spnTree'(M)\]
is a spanning tree of $\ourGraph_{\tilde\de}(N\cup_{\Sigma_{1}}M)$, where we identify edges in $\ourGraph_{\de_N}(N)$ associated to 2-cells in $\Sigma_1$ with the corresponding edges in $\ourGraph_{\tilde\de}(N \cup_{\Sigma_1}M)$. 
Indeed, the right hand side of the above equation contains all the vertices in $\ourGraph_{\tilde\de}(M \cup_{\Sigma_{1}}N)$, and by computing the Euler characteristic, it is easy to see that $\spnTree(M\cup_{\Sigma_{1}}N)$ is a tree. 
Similarly, define 
\[\spnTree(N\cup_{\Sigma_{1}}M_1) := \spnTree(N)\cup_{\bd} \spnTree'(M_1)\,.\]
Recall from \eqref{eq:gamma-1} that we denote $S(M\setminus\bbSigma) = (\de_{M\setminus\smbbSigma})_2 \setminus \spnTree'(M)$. 
Similarly define $S(\bbSigma) := (\de_{\smbbSigma})_2 \setminus \spnTree'(\bbSigma)$ and $S(N) := \de_N(N)_2 \setminus \spnTree(N)$, we have
\[S(N\cup_{\Sigma_1}M) := \tilde\de_2 \setminus \spnTree(M\cup_{\Sigma_1}N) = S(\bbSigma) \cup S(M\setminus\bbSigma) \cup S(N)\,.\]
Pick a set of interior points of 2-cells in $\bdf N$, then $Q_N\cup Q_M$ is a set of interior points of $\bdf(N\cup_{\Sigma_1}M)$, with which we can deform the graph $\Phi\cup_L\Gamma$ to the boundary of $\nb((N \cup_{\Sigma_1}M)_{\tilde\de}^{(1)})$, and we have 
\[\begin{split}
&\tau_{\cC}(N\cup_{\Sigma_1}M, \Phi\cup_{L}\Gamma)\\
&= 
\cgpt^{\bi}\Big(\nb((N \cup_{\Sigma_1}M)_{\tilde\de}^{(1)}) \,,\, (\Phi\cup_L \Gamma) \cup\bigcup_{\twoCell{W} \in S(N\cup_{\Sigma_1}M)}\gamma_{\belt(\twoCell{W})}\Big)\,.
\end{split}
\]
On the other hand, 
\eqref{eq:T-M1} says that the edges corresponding to the 2-cells $\twoCell{W}_{\pm}(\oneCell{X})$ for $\oneCell{X} \in Y$ are not included in $\spnTree'(M_1)$.
So by \eqref{eq:2cell-M1}, we have 
\[
S(N\cup_{\Sigma_1} M_1) := \tilde\vphi_2
\setminus \spnTree(N\cup_{\Sigma_1}M_1)= S(\bbSigma) \cup S(N) \cup S'\,,
\]
where 
$S' := \bigcup_{\oneCell{X} \in Y} 
\{\twoCell{W}_+(\oneCell{X}), \twoCell{W}_-(\oneCell{X})\}$.

Pick a set of interior points $Q_{M_{1}}$ of 2-cells in $\bdf M_1$, then deform $\Phi\cup_L \rb(\Gamma_1)$ according to $Q_{M_1} \cup Q_{N}$ so that it is contained in $\nb((N \cup_{\Sigma_1} M_{1})_{\tilde\vphi}^{(1)})$. Then we have
\[\begin{split}
&\tau_{\cC}(N\cup_{\Sigma_1}M_1, \Phi\cup_L \rb(\Gamma_1))\\ 
&= \cgpt^{\bi}\Big(\nb((N \cup_{\Sigma_1}M_1)_{\tilde\vphi}^{(1)}), (\Phi\cup_L\Gamma_1)\cup\bigcup_{\twoCell{W} \in S(N\cup_{\Sigma_1}M_1)} \gamma_{\belt_{\twoCell{W}}}\Big)\\
&= \cgpt^{\bi}\Big(\nb((N \cup_{\Sigma_1}M_1)_{\tilde\vphi}^{(1)}), (\Phi\cup_L\Gamma)\cup\bigcup_{\twoCell{W} \in S(N\cup_{\Sigma_1}M)} \gamma_{\belt({\twoCell{W}})}\cup\bigcup_{\twoCell{W} \in S'} \gamma_{\belt({\twoCell{W}})}\Big)
\end{split}\]
By \eqref{eq:1cell-M1}, 
\[\nb((N \cup_{\Sigma_1} M_1)_{\tilde\vphi}^{(1)}) = \nb(N \cup_{\Sigma_{1}}M)_{\tilde\de}^{(1)} \cup 
\bigcup_{\oneCell{X} \in Y} 
(\nb(\oneCell{X}_+) \cup \nb(\oneCell{X}_-))\]
and $\bigcup_{\twoCell{W} \in S'} \gamma_{\belt({\twoCell{W}})}$ are exactly the red loops along the longitudes of the tubes 
$\bigcup_{\oneCell{X} \in Y} 
(\nb(\oneCell{X}_+) \cup \nb(\oneCell{X}_-))$. Hence, we have 
\[\begin{split}
&\Cap_{S'}\Big(\nb((N \cup_{\Sigma_1}M_1)_{\tilde\vphi}^{(1)}), (\Phi\cup_L\Gamma)\cup\bigcup_{\twoCell{W} \in S(N\cup_{\Sigma_1}M)} \gamma_{\belt({\twoCell{W}})}\cup\bigcup_{\twoCell{W} \in S'} \gamma_{\belt({\twoCell{W}})}\Big) \\
&= \Big(\nb((N \cup_{\Sigma_1}M)_{\tilde\de}^{(1)}), (\Phi\cup_L \Gamma) \cup\bigcup_{\twoCell{W} \in S(N\cup_{\Sigma_1}M)}\gamma_{\belt(\twoCell{W})}\Big)
\end{split}\]
where $\Cap_{S'}$ stands for capping all the red loops $\gamma_{\belt(\twoCell{W})}$ for $\twoCell{W} \in S'$. Finally, by Lemma \ref{lem:hb-cap}, we have 
\[\begin{split}
&\tau_{\cC}(N\cup_{\Sigma_1}M_1, \Phi\cup_L\rb(\Gamma_1))\\    
&= \cgpt^{\bi}\Big(\nb((N \cup_{\Sigma_1}M)_{\tilde\de}^{(1)}), (\Phi\cup_L \Gamma) \cup\bigcup_{\twoCell{W} \in S(N\cup_{\Sigma_1}M)}\gamma_{\belt(\twoCell{W})}\Big)\\
&= \tau_{\cC}(N\cup_{\Sigma_1}M, \Phi\cup_{L}\Gamma)\,,
\end{split}\]
and we are done.
\end{proof}

\end{document}